\documentclass{amsart}

\newenvironment{acknowledgements}[1]
   {\noindent {\bf Acknowledgements}\quad #1}

\usepackage{amsmath} 
\usepackage{amssymb}
\usepackage{mathrsfs}

\newtheorem{theorem}{Theorem}[section] 
\newtheorem{claim}[theorem]{Claim}

\newtheorem{cc}[theorem]{Crucial Claim}

\newtheorem{conclusion}[theorem]{Conclusion}
\newtheorem{observation}[theorem]{Observation}

\theoremstyle{definition}
\newtheorem{definition}[theorem]{Definition}

\newtheorem{convention}[theorem]{Convention}

\newtheorem{hypothesis}[theorem]{Hypothesis}

\theoremstyle{remark}
\newtheorem{remark}[theorem]{Remark}
\newtheorem{question}[theorem]{Question}
\newtheorem{qwa}[theorem]{Question}
\newtheorem{notation}[theorem]{Notation}

\newcommand{\am}{{\rm am}}

\newcommand{\nb}{{\rm nb}}

\newcommand{\awc}{{\rm awc}}
\newcommand{\inc}{{\rm inac}}
\newcommand{\inac}{{\rm inac}}
\newcommand{\incr}{{\rm incr}}

\newcommand{\non}{{\rm non}}

\newcommand{\tr}{{\rm tr}}

\newcommand{\Sym}{{\rm Sym}}

\newcommand{\add}{{\rm add}}
\newcommand{\Add}{{\rm Add}}

\newcommand{\Levy}{{\rm Levy}}

\newcommand{\real}{{\rm real}}

\newcommand{\suc}{{\rm suc}}

\newcommand{\otp}{{\rm otp}}
\newcommand{\cov}{{\rm cov}}

\newcommand{\Cohen}{{\rm Cohen}}
\newcommand{\Borel}{{\rm Borel}}

\newcommand{\bd}{{\rm bd}}

\newcommand{\wc}{{\rm wc}}

\newcommand{\meagre}{{\rm meagre}}

\newcommand{\cf}{{\rm cf}}

\newcommand{\id}{{\rm id}}

\newcommand{\lh}{{\ell g}}
\newcommand{\rest}{{\restriction}}
\newcommand{\dom}{{\rm dom}}

\newcommand{\set}{{\rm set}}
\newcommand{\wilog}{{\rm without loss of generality}}
\newcommand{\Wilog}{{\rm Without loss of generality}}

\newcommand{\then}{{\underline{then}}}
\newcommand{\when}{{\underline{when}}}

\newcommand{\Then}{{\underline{Then}}}

\newcommand{\Iff}{{\underline{iff}}}
\newcommand{\mn}{{\medskip\noindent}}

\newcommand{\cA}{{\mathscr A}}
\newcommand{\bB}{{\bold B}}

\newcommand{\cJ}{{\mathscr J}}

\newcommand{\bbI}{{\mathbb I}}
\newcommand{\bbZ}{{\mathbb Z}}
\newcommand{\gb}{{\mathfrak b}}

\newcommand{\cP}{{\mathscr P}}

\newcommand{\bbR}{{\mathbb R}}
\newcommand{\bbN}{{\mathbb N}}

\newcommand{\gd}{{\mathfrak d\/}}

\newcommand{\gx}{{\mathfrak x\/}} 

\newcommand{\cH}{{\mathscr H}}

\newcommand{\cI}{{\mathscr I}}

\newcommand{\bbL}{{\mathbb L}}

\newcommand{\bbP}{{\mathbb P}}
\newcommand{\gp}{{\mathfrak p}}

\newcommand{\bbQ}{{\mathbb Q}}

\newcommand{\mbR}{{\mathbb R}}

\newcommand{\cT}{{\mathscr T}}
\newcommand{\gt}{{\mathfrak t}} 
\newcommand{\cU}{{\mathscr U}}
\newcommand{\cX}{{\mathscr X}}

\newcommand{\cY}{{\mathscr Y}}

\newcommand{\pr}{{\rm pr}}

\newcommand{\vare}{\varepsilon}
\newcommand{\nst}{{\bf nst}}

\newcount\skewfactor
\def\mathunderaccent#1#2 {\let\theaccent#1\skewfactor#2
\mathpalette\putaccentunder}
\def\putaccentunder#1#2{\oalign{$#1#2$\crcr\hidewidth
\vbox to.2ex{\hbox{$#1\skew\skewfactor\theaccent{}$}\vss}\hidewidth}}
\def\name{\mathunderaccent\tilde-3 }

\newenvironment{PROOF}[2][\proofname.]
   {\begin{proof}[#1]\renewcommand{\qedsymbol}{\textsquare$_{\rm #2}$}}
   {\end{proof}}

\begin{document}

\title[Null ideal for inaccessible $\lambda$]{A parallel to the null 
ideal for inaccessible $\lambda$. Part I}
\author {Saharon Shelah}
\address{Einstein Institute of Mathematics\\
Edmond J. Safra Campus, Givat Ram\\
The Hebrew University of Jerusalem\\
Jerusalem, 91904, Israel\\
 and \\
 Department of Mathematics\\
 Hill Center - Busch Campus \\ 
 Rutgers, The State University of New Jersey \\
 110 Frelinghuysen Road \\
 Piscataway, NJ 08854-8019 USA}
\email{shelah@math.huji.ac.il}
\urladdr{http://shelah.logic.at}
\thanks{Research supported by the United-States-Israel Binational
 Science Foundation (Grant No. 2006108). Publication 1004.}



\date{December 30, 2016}

\begin{abstract}
  It is well known how to generalize the meagre ideal replacing $\aleph_0$
  by a (regular) cardinal $\lambda > \aleph_0$ and requiring the ideal to be
  $(<\lambda)$-complete.  But can we generalize the null ideal?  In terms of
  forcing, this means finding a forcing notion similar to the random real
  forcing, replacing $\aleph_0$ by $\lambda$.  So naturally, to call it a
  generalization we require it to be $(< \lambda)$-complete and
  $\lambda^+$-c.c. and more.  Of course, we would welcome additional
  properties generalizing the ones of the random real forcing.  Returning to
  the ideal (instead of forcing) we may look at the Boolean Algebra of
  $\lambda$-Borel sets modulo the ideal.  Common wisdom have said that there
  is no such thing because we have no parallel of Lebesgue integral, but
  here surprisingly \underline{first} we get a positive = existence answer
  for a generalization of the null ideal for a ``mild'' large cardinal
  $\lambda$ - a weakly compact one.  Second, we try to show that this
  together with the meagre ideal (for $\lambda$) behaves as in the countable
  case.  In particular, we consider the classical Cicho\'n diagram, which
  compares several cardinal characterizations of those ideals.  We shall
  deal with other cardinals, and with more properties of related forcing
  notions in subsequent papers \cite{Sh:F1199, Sh:F1538, Sh:F1580, Sh:1100, Sh:E82}
  and Cohen and Shelah \cite{CnSh:1085} and a joint work with
  Baumhauer and Goldstern.
\end{abstract}

\maketitle
\numberwithin{equation}{section}
\setcounter{section}{-1}
\newpage

\centerline {Annotated Content}
\bigskip

\noindent
\S0 \quad Introduction, pg. \pageref{Introduction}

\S(0A) \quad Aim: for general audience, pg. \pageref{Aim}

\smallskip

\S(0B) \quad For set theorist, pg. \pageref{Forset}
\begin{enumerate}
\item[${{}}$]  [We describe the content of the paper.]
\end{enumerate}
\smallskip

\S(0C) \quad Preliminaries, pg.\pageref{Prelim}
\begin{enumerate}
\item[${{}}$]  [We quote some results and definitions.]
\end{enumerate}
\smallskip

\noindent
\S1 \quad Like random real forcing for weakly compact $\kappa$,
pg. \pageref{Like} 
\smallskip

\S(1A) \quad Adding $\eta \in {}^\kappa 2$, pg.7
\begin{enumerate}
\item[${{}}$]  [For a weakly compact cardinal $\kappa$ we 
construct a forcing notion $\bbQ_\kappa$ which is
$\kappa^+$--c.c. $\kappa$-strategically complete (even can be 
represented as $\kappa$-complete) and is $\kappa$--bounding and adds a new  
$\name\eta \in {}^\kappa \kappa$, its generic.  Moreover, it is very
explicitly defined (and is $\kappa$--Borel).  ]
\end{enumerate}

\S(1B) \quad Adding a dominating $\eta \in \prod\limits_{\varepsilon <
\kappa} \theta_\varepsilon$,  pg. \pageref{Adding}
\begin{enumerate}
\item[${{}}$]  [We deal with a generalization.]
\end{enumerate}
\smallskip

\noindent
\S2 \quad What are the desired properties of the ideal, pg. \pageref{The} 
\smallskip

\S(2A) \quad Desirable properties: first list, pg. \pageref{Desirable} 

\S(2B) \quad Desirable properties: second list, pg. \pageref{TwoDes} 
\smallskip

\noindent
\S3 \quad On $\bbQ_\kappa$, $\kappa$--Borel sets and $\id(\bbQ_\kappa)$, 
pg. \pageref{idBorel}
\begin{enumerate}
\item[${{}}$]  [We introduce and investigate the ideal $\id(\bbQ_\kappa)$ of
  subsets of ${}^\kappa2$ determined by our forcing notion $\bbQ_\kappa$. We
  study properties of $\kappa$--Borel subsets of ${}^\kappa 2$ related to
  this ideal.]
\end{enumerate}
\smallskip

\noindent
\S4 \quad On $\add(\bbQ_\kappa)$ and $\cf(\bbQ_\kappa)$,
pg. \pageref{addQ} 
\begin{enumerate}
\item[${{}}$]  [We give a representation of the additivity number
  $\add(\bbQ_\kappa)$ of our ideal.]
\end{enumerate}
\smallskip

\noindent
\S5 \quad The parallel of the Cicho\'n Diagram, pg. \pageref{CichKappa}
\begin{enumerate}
\item[${{}}$]  [We prove several inequalities between cardinal coefficients
  of the ideals $\id(\bbQ_\kappa)$ and present a parallel of Cicho\'n Diagram.]
\end{enumerate}
\smallskip

\noindent
\S6 \quad $\bbQ_\kappa$ vs $\Cohen_\kappa$, pg. \pageref{VS}
\smallskip

\S(6A) \quad Effect on the ground model, pg. \pageref{effect} 
\begin{enumerate}
\item[${{}}$]  [We show that $\bbQ_\kappa$ makes the ground model ${}^\kappa
  2$ meagre and $\Cohen_\kappa$ makes the ground model ${}^\kappa 2$ to
  belong to $\id(\bbQ_\kappa)$.]
\end{enumerate}

\S(6B) \quad When does $\bbQ_\kappa$ add a Cohen real?, pg. \pageref{addCoh}  
\begin{enumerate}
\item[${{}}$]  [We investigate the class $S_{\awc}$ of all inaccessible
  $\kappa$ for which $\bbQ_\kappa$ adds a Cohen real.]
\end{enumerate}
\smallskip

\noindent
\S(7) \quad What about the parallel to ``amoeba forcing''?,
pg. \pageref{amoeba} 
\begin{enumerate}
\item[${{}}$]  [We introduce the forcing notion $\bbQ^\am_\kappa$ adding a
  generic condition $\name{p}_\kappa\in \bbQ_\kappa$]
\end{enumerate}

\noindent
\S(8) \quad Generics and absoluteness, pg. \pageref{absolut}
\begin{enumerate}
\item[${{}}$]  [We investigate $\kappa$--Borel and
  $\kappa$--stationary--Borel sets and show that some relations associated
  with $\bbQ_\kappa$ are absolute.]
\end{enumerate}

\newpage

\section {Introduction} \label{Introduction}

\subsection {Aims: for general audience} \label{Aim}\
\bigskip

The ideals of null sets and of meagre sets on the reals are certainly
central in mathematics.  From the forcing point of view we speak of random
real forcing and Cohen forcing.  The Cohen forcing has natural
generalizations (and relatives) when we replace $\cP(\bbN)$ by
$\cP(\lambda)$, or the set of the characteristic functions of subsets of
$\lambda$, for a regular uncountable cardinal $\lambda$, replacing finite by
``of cardinality ${<}\lambda$".  But we lack a generalization of random real
forcing to higher cardinals $\lambda$, replacing reals by $\lambda$-reals,
e.g. members of ${}^\lambda 2$.  It has seemed that this lack is due to
nature; the reason being that on the one hand the Baire category theorem
generalizes naturally (when we are allowed to approximate in $\lambda$-steps
and information of size $< \lambda$ instead finite; all this for regular
$\lambda$), but on the other hand we know nothing remotely like Lebesgue
measure.

Surprisingly, at least for me, there is a generalization: not of the
Lebesgue measure, \underline{but} of the ideal of null sets, i.e., the ones
of Lebesgue measure zero.  This is done here (i.e., in this part) for a mild
large cardinal $\lambda$: weakly compact.  The solution for more cardinals
will be dealt with in a continuation (at some price).  The present
definition should be examined in two ways.  First, we may list the well
known properties of the null ideal (and of random real forcing) and try to
prove (or disprove) them for our ideal.  Second, random real forcing was
used quite extensively in independence results; in particular for related
cardinal invariants, so it is natural to try to generalize such applications.

The first issue is dealt with in \S2 (assuming Definition \ref{n5} and
intended for wider audience) and then \S3--\S8 here.  The second is treated in
the continuation.  Whereas success in the second issue should be easy to judge,
concerning the first issue the reader may first list what are reasonable
hopes and compare them with the discussion and description in \S3. This is
not done in the present section in order to help the reader to make a list
of expectations independent of what we have done. 

A set theoretically uninitiated reader may read the rest of \S(0A) to see
what are those large cardinals, look casually at Definition \ref{n5}, just
enough to see that the definition of $\bbQ_\kappa$, the parallel of the
family of all closed subsets of $[0,1]_{\bbR}$ or ${}^\omega 2$ which are
not Lebesgue null for $\kappa$ strongly inaccessible, is natural
and simple, then jump to \S2 to see what we hope for and what is done.

Let us describe for the non-set-theoretic reader, what are these ``large
cardinals".  Note that $\aleph_1$ is parallel in some respect to $\aleph_0$,
whereas $\aleph_0$ is ``the first infinite cardinal"; the number of natural
numbers; $\aleph_1$ is the first uncountable cardinal, and is the number of
countable ordinals (that is, isomorphism types of countable linear well
orderings).  Also both are so called regular: the union of less than
$\aleph_\ell$ sets each of cardinality $<\aleph_\ell$ is $< \aleph_\ell$.
But $\aleph_0$ is strong limit: $\kappa < \aleph_0 \Rightarrow 2^\kappa <
\aleph_0$ whereas $\aleph_1$ is not.  We can prove that there are strong
limit cardinals: let $\beth_0 = \aleph_0,\beth_{n+1} =
2^{\beth_n},\beth_\omega = \sum\limits_{n<\omega} \beth_n$, now
$\beth_\omega$ is a strong limit cardinal but alas is not regular.  We say a
cardinal $\lambda$ is (strongly) inaccessible when $\lambda$ is regular and
strong limit, it is called ``large cardinal" because we cannot prove its
existence in ZFC but, modulo this, it is considered a very reasonable, small
one.  Similarly, the weakly compact ones which we now introduce: an
uncountable cardinal is weakly compact when it is strongly inaccessible and
satisfies the analog of the infinite Ramsey theorem: every graph with
$\lambda$ nodes has a subgraph with $\lambda$ nodes which is complete or
empty (alternatively, it satisfies the generalization of K\"onig lemma).  So
weakly compact cardinals are more similar to $\aleph_0$ than other
cardinals, so it is not unnatural assumption when trying to generalize the
null ideal.  \bigskip

\subsection {For Set Theorists} \label{Forset} \
\bigskip

In the present paper we prove that for a weakly compact cardinal $\lambda$
there are (naturally defined) forcing notions adding a new
$\eta \in {}^\lambda 2$ which have not few parallels (replacing ``finite'' by
``of cardinality $< \lambda$'') of the properties associated with random real
forcing (and we define the relevant ideal).  It seems natural to hope this
will enable us to understand better related problems, in particular cardinal
invariants of $\lambda$; on cardinal invariants for $\lambda = \aleph_0$,
i.e. the continuum see Blass \cite{Bls10}; in higher cases see Cummings and
Shelah \cite{CuSh:541}; in particular on strongly inaccessible see
Ros{\l}anowski and Shelah \cite{RoSh:777,RoSh:888, RoSh:889, RoSh:942} and
also \cite{Sh:945}.

In \S1 we show for $\lambda$ weakly compact that there is a (non-trivial)
$\lambda$--bounding $\lambda^+$-c.c. ($< \lambda$)-strategically complete
forcing notion and even a $\lambda$-complete one, see \ref{z3d}. We also
generalize the construction for adding a member of
$\prod\limits_{\vare<\lambda}\theta_\vare$.

In the second section we discuss desirable properties of the ideal.
In Sections 3--8 we try to deal systematically with parallels of properties of
the null ideal. 

The ideal $\id(\bbQ_\kappa)$ (of subsets of ${}^\kappa2$) determined by our
forcing notion $\bbQ_\kappa$ is introduced in \S3. There we also study the
properties of $\kappa$--Borel subsets of ${}^\kappa 2$ related to this
ideal.

Cardinal characteristics of the ideal $\id(\bbQ_\kappa)$ and their relations
to $\gb_\kappa,\gd_\kappa$ and the characteristics of the $\kappa$--meagre
ideal are investigated in Sections 4 and 5. We present a parallel of
Cicho\'n Diagram in Theorem \ref{u14}. 

In \S6 we compare $\bbQ_\kappa$ and $\Cohen_\kappa$. We note that forcing
with one makes the set of ground model $\kappa$--reals small in the dual
sense. We also investigate the class $S_{\awc}$ of all inaccessible
cardinals $\kappa$ for which $\bbQ_\kappa$ adds a Cohen real.

In the next section we introduce a parallel to ``amoeba forcing'' --- a
forcing notion $\bbQ^\am_\kappa$ adding a generic condition
$\name{p}_\kappa\in \bbQ_\kappa$. And then, in \S8, we investigate
$\kappa$--Borel and $\kappa$--stationary--Borel sets and show that some
relations associated with $\bbQ_\kappa$ are absolute.
\medskip

We shall continue in successive papers, things delayed for various reasons.
In particular in Cohen and Shelah \cite{CnSh:1085} we shall eliminate the
assumption ``$\lambda$ is weakly compact'' and in \cite[\S 1]{Sh:E82} we
will investigate non-inaccessible case. A work with Baumhaver and Goldstern
(see \cite{Sh:F1580}) will deal with consistency results complimentary to
the ZFC implications (i.e., inequalities) here.  In \cite[\S 1]{Sh:E82} we
investigate adding many ``$\lambda$--randoms''. Further research concerning
consistency results using iteration of creature forcing will be presented in
\cite{Sh:1100}. We will also consider there constructions starting not with
Cohen but other nice forcing notions and more.

\subsection {Preliminaries} \label{Prelim}
\bigskip

\begin{definition}
\label{y6}
0)  We say $\eta$ is a $\lambda$-real when $\eta \in {}^\lambda 2$.

\noindent
1) We define when $\bold B \subseteq {}^\lambda 2$ is a $\lambda$-Borel set
naturally (see \cite{Sh:630}), that is $X \subseteq {}^\lambda 2$ is a
basic $\lambda$-Borel set if there exists $\nu \in {}^{\lambda>}2$
such that $X = ({}^\lambda 2)^{[\nu]} = \{\eta \in {}^\lambda 2:\nu
\triangleleft \eta\}$.  The family of $\lambda$-Borel sets is the
closure of the basic ones under unions and intersections of at most
$\lambda$ members, hence also by complements.

Note: actually $\bold B$ is an absolute definition of a subset of
${}^\lambda 2$ so $\bold B^{\bold V}$, ``$\bold B$ as interpreted in
the universe $\bold V$'', is well defined for suitable $\bold V$.

\noindent
2) ``$F$ is a $\lambda$--Borel function'' is defined similarly.

\noindent
3) $\bold B \subseteq {}^\lambda 2$ is a $\Sigma^1_1(\lambda)$--set when
$\bold B = \{\langle \eta(2 \alpha):\alpha < \lambda\rangle: \eta \in \bold
B_1\}$ for some $\lambda$--Borel set $\bold B_1$.

\noindent
4) $\bold B \subseteq {}^\lambda 2$ is a $\lambda$--stationary Borel set
\when \, for some $\lambda$--Borel function $F:{}^\lambda 2 \rightarrow
\cP(\lambda)$ we have $\eta \in B \Leftrightarrow F(\eta)$ is stationary.

\noindent 
5) A set $X \subseteq {}^\lambda \cH(\lambda)$ is $\lambda$--nowhere stationary
Borel \Iff \, there is a $\lambda$--Borel function $\bold B$ from
${}^\lambda \cH(\lambda)$ to $\cP(\lambda)$ such that for every
$\eta \in {}^\lambda \cH(\lambda)$ we have: $\eta \in X$ \Iff \, $F(\eta)$ is
a nowhere stationary subset of $\lambda$ (see \ref{z6}(2)).  The complements
of such $X$ are $\lambda$-somewhere stationary sets.

\noindent 
6) Similarly replacing ${}^{\lambda >}2$ by other trees with $\lambda$
levels and $\lambda$ nodes.
\end{definition}

\begin{definition}
\label{y8}
1) We say that a set $B \subseteq {}^\lambda 2$ is $\lambda$--closed \when
\,: 
\begin{enumerate}
\item[$\bullet$]  $\eta \in {}^\lambda 2 \wedge (\forall \alpha <
\lambda)(\exists \nu \in B)(\eta \rest \alpha = \nu \rest \alpha)
\quad \Rightarrow\quad \eta \in B$, 
\end{enumerate}
equivalently
\begin{enumerate}
\item[$\bullet$]  for some sub-tree $T \subseteq {}^{\lambda >} 2$ we have
\[B = {\lim}_\lambda(T) \stackrel{\rm def}{=} \{\eta:\eta \mbox{ a sequence of
  length $\lambda$ such that }\alpha < \lambda \Rightarrow \eta \rest \alpha
\in T\}.\] 
\end{enumerate}

\noindent
2) Let  $\bbQ$ be a family of subtrees of ${}^{\lambda >}2$ (or a quasi
order with such set of elements). We say that $B \subseteq {}^\lambda 2$ is
a $\bbQ$--basic set \when \, $B =\lim_\lambda(p)$ for some $p \in \bbQ$. 

\noindent
3) Similarly replacing ${}^{\lambda >}2$ by other trees, as in \ref{y6}(6).
\end{definition}

\begin{definition}
\label{z3}
1) We say that a forcing notion $\bbP$ is $\alpha$-strategically complete
\underline{when} the player COM has a winning strategy in the following game
$\Game_\alpha(p,\bbP)$ for each $p \in \bbP$.

The game $\Game_\alpha(p,\bbP)$ involves two players, COM and INC. A play
lasts $\alpha$ moves; in the $\beta$-th move, first the player COM 
chooses $p_\beta \in \bbP$ such that $p \le_{\bbP} p_\beta$ and $\gamma <
\beta \Rightarrow q_\gamma \le_{\bbP} p_\beta$ and second the player INC
chooses $q_\beta \in \bbP$ such that $p_\beta \le_{\bbP} q_\beta$.

The player COM wins a play if it has a legal move for every $\beta <
\alpha$.

\noindent
2) We say that a forcing notion $\mathbb P$ is $({<}\lambda)$-strategically
complete \underline{when} it is $\alpha$-strategically complete for every
$\alpha < \lambda$.
\end{definition}

\begin{remark}
\label{z3d}
The difference between ``$\bbP$ is $\lambda$-strategically complete"
and ``$\lambda$-complete" is not real, i.e., when we do not distinguish
between equivalent forcing, those properties are very close (as in
\cite[Ch.XIV]{Sh:f}), and here the difference does not matter, see
e.g.~\ref{n9}(2). 
\end{remark}

\begin{definition}
\label{z4}
1) The $\lambda$-Cohen forcing is $({}^{\lambda >}2,\triangleleft)$.

\noindent
2) A forcing notion $\bbQ$ is $\lambda$-bounding or ${}^\lambda
\lambda$-bounding \when \, $\Vdash_{\bbQ}$ ``for every function $f$ from
$\lambda$ to $\lambda$ there is $g \in ({}^\lambda \lambda)^{\bold V}$ such
that $f \le g$, i.e., $\alpha < \lambda \Rightarrow f(\alpha) \le
g(\alpha)"$.

\noindent
3) We say that a $\bbQ$--name $\name\eta \in {}^\alpha \beta$ is a generic
of $\bbQ$ \when \, for some sequence $\langle \tau_p:p \in \bbQ\rangle$,
$\tau_p$ an absolute function definable in $\bold V$ (or even a $(|\alpha| +
|\beta|)$-Borel one) from ${}^\alpha \beta$ into $\{0,1\}$ we have $\Vdash
``p \in \name{\bold G}$ iff $\tau_p(\name\eta) = 1"$.
\end{definition}

\begin{definition}
\label{z6}
\begin{enumerate}
\item Let $S_\inc$ be the class of all (strongly) inaccessible cardinals and
  let $S^\kappa_{\inc} = \{\partial:\partial < \kappa$ is inaccessible$\}$. 
\item We say ``$S$ is nowhere stationary" \when \, $S$ is a set of ordinals,
  and for every ordinal $\delta$ of uncountable cofinality, $S \cap \delta$
  is not a stationary subset of $\delta$.
\item For a set $p$ of sequences of ordinals and $\eta$ let
  $p^{[\eta]} = \{\nu \in p:\nu \trianglelefteq \eta$ or
  $\eta \trianglelefteq \nu\}$ and
  $p^{[\ge \eta]} = \{\nu \in p:\eta \trianglelefteq \nu\}$.
\end{enumerate}
\end{definition}

\begin{definition}
\label{z26}
For an ideal $\bbI$ of subsets of $X$, including all singletons for
simplicity, we define ``the four basic cardinal invariants of the ideal":
\begin{enumerate}
\item[(a)]  $\cov(\bbI)$, the covering number is $\min\{\theta$:
  there are $A_i \in \bbI$ for $i < \theta$ whose union is $X\}$,
\item[(b)] $\add(\bbI)$, the additivity of $\bbI$ is $\min\{\theta$: there
  are $A_i \in \bbI$ for $i < \theta$ whose union is not in $\bbI\}$, 
\item[(c)]  $\cf(\bbI)$, the cofinality of $\bbI$ is $\min\{\theta$: there
  are $A_i \in \bbI$ for $i < \theta$ such that $(\forall A \in
  \bbI)(\exists i)(A \subseteq A_i)\}$,
\item[(d)]  $\non(\bbI)$, the uniformity of $\bbI$ is $\min\{|Y|:Y
  \subseteq X$ but $Y \notin \bbI\}$. 
\end{enumerate}
\end{definition}

\begin{remark}
\label{z30}
We may use, e.g., $\cov(\meagre_\lambda)$ and $\cov(\Cohen_\lambda)$,
they denote the same number.
\end{remark}

\begin{observation}
\label{z27}
For any ideal $\bbI$:
\begin{enumerate}
\item[$(a)$]  $\add(\bbI) \le \cov(\bbI) \le \cf(\bbI)$,
\item[$(b)$]  $\add(\bbI) \le \non(\bbI) \le \cf(\bbI)$
\end{enumerate}
\end{observation}

\section{Like random real forcing for weakly compact $\kappa$} 
\label{Like}  

We consider the following question.
\begin{qwa}
\label{a0}
\begin{enumerate}
\item Is there a non-trivial forcing notion which is $\lambda^+$--c.c.,
  $({<} \lambda)$--strategically complete and which does not add a
  $\lambda$--Cohen sequence from ${}^\lambda 2$ ?
\item Moreover is $\lambda$-bounding ?
\end{enumerate}
\end{qwa}

Recall that for $\lambda = \aleph_0$, ``random real forcing" is such 
forcing notion but we do not know to generalize measure to $\lambda$
with $\lambda$--completeness or so, whereas for Cohen forcing and many
other definable forcing notions which add a Cohen real we know how to
generalize. 

We have wondered about this a long time, see \cite{Sh:945} and some papers
of Ros{\l}anowski and Shelah \cite{RoSh:777,RoSh:860,RoSh:888,RoSh:942}.  Up to
recently, we were sure that the answer was negative. Surprisingly for
$\lambda$ weakly compact there is a positive answer, a posteriori a
straightforward one.

We will define a forcing notion $\bbQ_\kappa$ by induction on the
inaccessible $\kappa$.  Now, for $\kappa$ the first inaccessible
$\bbQ_\kappa$ is the $\kappa$-Cohen forcing. In fact, if $\kappa$ is
inaccessible but not a limit of inaccessible cardinals, then $\bbQ_\kappa$
is equivalent to the $\kappa$--Cohen forcing.  If $\kappa$ is a limit of
inaccessibles, the conditions are such that the generic $\name\eta \in
{}^\kappa 2$ satisfies for many inaccessibles $\partial<\kappa$, that
$\name\eta \rest \partial$ is somewhat $\partial$--Cohen, e.g., if $\langle
\cI_\partial:\partial \in S \rangle$ is a sequence such that $\cI_\partial$
is a dense open subset of ${}^\partial 2$ and $S =
\{\partial<\kappa:\partial$ is the first strong inaccessible in
$(\alpha,\kappa)$ for some $\alpha < \kappa\}$, \then \, for  every large
enough $\partial \in S$ we have $\name\eta \rest \partial \in \cI_\partial$.   

At first glance this may look ridiculous: $\name\eta$ is made more
Cohen--like, but still in the end, i.e., for $\kappa$ weakly compact, it has
an antithetical character.  
\bigskip

\subsection {Adding an $\eta \in {}^\kappa 2$} \label{Adding} \
\bigskip

\begin{notation}
\label{n3}
1) Here $\partial,\kappa$ will denote strongly inaccessible cardinals.

\noindent
2) For $\cT \subseteq {}^{\alpha >}2$ and $\eta \in {}^{\alpha >}2$ let
$\cT^{[\eta]} = \{\nu:\nu \trianglelefteq \eta$ or $\eta \trianglelefteq \nu
\in \cT\}$.

\noindent
3) For $\cT \subseteq {}^{\delta >}2$ let $\lim_\delta(\cT) = \{\nu \in
{}^\delta 2:(\forall \alpha < \delta)(\nu \rest \alpha \in \cT)\}$.
\end{notation}

\begin{definition}
\label{n5}
We define a forcing notion $\bbQ_\kappa = \bbQ^2_\kappa$ by induction
on inaccessible $\kappa$:
\begin{enumerate}
\item[(A)]  $p \in \bbQ_\kappa$ \Iff \, there is a witness
$(\varrho,S,\bar\Lambda)$ which means:
\begin{enumerate}
\item[(a)]  $p$ is a subtree of ${}^{\kappa >}2$, i.e., a non-empty subset
  of ${}^{\kappa >}2$ closed under initial segments, 
\item[(b)]  
\begin{enumerate}
\item[$(\alpha)$] $S \subseteq \kappa$ is not stationary, moreover
\item[$(\beta)$]   for every strongly inaccessible $\partial\leq\kappa$ the 
  set  $S\cap\partial$ is not stationary,  
\item[$(\gamma)$] every member of $S$ is (strongly) inaccessible,
\end{enumerate}
\item[(c)]  $\varrho = \tr(p)$ is the trunk of $p$ which means:
\begin{enumerate}
\item[$(\alpha)$] $\varrho \in {}^{\kappa >}2$,
\item[$(\beta)$] $\alpha \le \ell g(\varrho) \Rightarrow p \cap {}^\alpha 2
  = \{\varrho \rest \alpha\}$, hence $\tr(p) \in p$,
\item[$(\gamma)$] both $\varrho \char 94 \langle 0 \rangle$ and $\varrho
  \char 94\langle 1 \rangle$ belongs to $p$,
\end{enumerate}
\item[(d)]  if $\varrho \trianglelefteq \eta \in p$ then $\eta \char
  94 \langle 0 \rangle,\eta \char 94 \langle 1 \rangle \in p$,
\item[(e)] [continuity] if $\delta \in \kappa \backslash S$ is a limit
  ordinal $> \ell g(\varrho)$ and $\eta \in {}^\delta 2$ \then \, 
\[\eta \in p\ \mbox{ iff }\ (\forall \alpha < \delta)(\eta \rest \alpha \in
p), \]
\item[(f)]  
\begin{enumerate}
\item[$(\alpha)$] $\bar\Lambda = \langle\Lambda_\partial:\partial \in
  S\rangle$,
\item[$(\beta)$] $\Lambda_\partial$ is a set of $\le \partial$ dense open
  subsets of $\bbQ_\partial$,
\end{enumerate}
\item[(g)]  if $\partial \in S$ and $\partial > \ell g(\varrho)$ and
  $\eta\in {}^\partial 2$, then
\begin{enumerate}
\item[$(\alpha)$]  $p \cap {}^{\partial >}2 \in \bbQ_\partial$,
\item[$(\beta)$] $\eta \in p$ \Iff \, $(\forall \alpha<\partial)(\eta \rest
  \alpha \in p)$ and $(\forall \cI \in \Lambda_\partial)(\exists q \in
  \cI)[\eta \in \lim_\partial(q)]$. 
\end{enumerate}
\end{enumerate}
\item[(B)]  $\bbQ_\kappa \models ``p \le q"$ \Iff \, $p \supseteq q$.
\item[(C)]  
\begin{enumerate}
\item[(a)] Let $S_p = \{\delta < \kappa:\delta > \ell g(\tr(p)),\delta$ is
  a limit ordinal and  $\neg(\forall \eta \in {}^\delta 2)[\eta \in p \leftrightarrow
(\forall \alpha < \delta)(\eta \rest \alpha \in p)]\}$,
so $S_p\subseteq S$ when $(\tr(p),S,\bar\Lambda)$ is a witness.
\item[(b)] We say $(\tr(p),S,\bar\Lambda,E)$ is a full witness for $p \in
  \bbQ_\kappa$ if $(\tr(p),S,\bar\Lambda)$ is a witness for $p \in
  \bbQ_\kappa$ and $E$ is a club of $\kappa$ disjoint to $S$ and to $[0,\ell
  g(\tr(p)))$, 
\end{enumerate}
\end{enumerate}
\end{definition}

\begin{claim}
\label{n7}
1) For any $\kappa$ and $\eta \in {}^{\kappa >}2$ we have $({}^{\kappa
  >}2)^{[\eta]}$ is a member of $\bbQ_\kappa$ with $\tr(({}^{\kappa
  >}2)^{[\eta]}) = \eta$.

\noindent
2) If $p \in \bbQ_\kappa$ and $\ell g(\tr(p)) < \partial < \kappa$ \then \,
$p \cap {}^{\partial >}2$ belongs to $\bbQ_\partial$.

\noindent
3) If $p \in \bbQ_\kappa$ and $\eta \in p$ \then \, $p^{[\eta]} \in
\bbQ_\kappa$ and $p \le p^{[\eta]}$ and $\tr(p^{[\eta]})$ is $\eta$ if $\ell
g(\eta) \ge \ell g(\tr(p))$ and is $\tr(p)$ otherwise.

\noindent
4) ${}^{\kappa >}2$ is the minimal member of $\bbQ_\kappa$.

\noindent
5) If $(\tr(p),S,\bar\Lambda)$ is a witness for $p \in \bbQ_\kappa$ and
$\ell g(\tr(p)) \ge \sup(S)$ \then \, $p = ({}^{\kappa >}2)^{[\tr(p)]}$.

\noindent
6) Any triple $(\varrho,S,\bar\Lambda)$ is a witness for at most one $p$.

\noindent
7) If $(\varrho,S,\bar\Lambda)$ satisfies clauses
$(c)(\alpha),(b)(\alpha),(\beta),(\gamma),(f)(\alpha),(\beta)$ of Definition
\ref{n5}(A) \then \, there is one and only one $p \in \bbQ_\kappa$ which it
witnesses. 

\noindent
8) If $(\varrho,S,\bar\Lambda)$ witnesses $p \in \bbQ_\kappa$, \then \, also
$(\varrho,S_p,\bar\Lambda \rest S_p)$ witnesses it recalling Definition
\ref{n5}(C)(a).

\noindent
9) For every $p \in \bbQ_\kappa$ there is a maximal antichain $\cI$ to which
$p$ belongs and $q_1 \ne q_2 \in \cI \Rightarrow \lim_\kappa(q_1) \cap
\lim_\kappa(q_2) = \emptyset$ hence $\{q \in \bbQ_\kappa:p \le_{\bbQ_\kappa}
q$ \underline{or} $\lim_\kappa(q) \cap \lim_\kappa(p) = \emptyset\}$ is
dense open.
\end{claim}

\begin{PROOF}{\ref{n7}}
1) Let $S = \emptyset$. Then $(\eta,\emptyset,<>)$ is a witness.

\noindent
2) If $(\tr(p),S,\langle \Lambda_\theta:\theta \in S\rangle)$ witnesses
$p\in \bbQ_\kappa$, then $(\tr(p),S \cap \partial, \langle
\Lambda_\theta:\theta\in S\cap \partial\rangle)$ witnesses $p \cap
{}^{\partial >}2 \in \bbQ_\partial$. 

\noindent
3) - 8)  Easy, too.

\noindent
9) Let $\cI = \{({}^{\kappa >}2)^{[\rho]}:\rho \in {}^{\kappa >}2
   \backslash p$ and $\alpha < \ell g(\rho) \Rightarrow \rho \rest
   \alpha \in p\} \cup \{p\}$.
\end{PROOF}

\begin{claim}
\label{n9}
1) If $p \in \bbQ_\kappa$ and $\rho \in p$, \then \, there is $\eta$ such 
that $\rho \trianglelefteq \eta \in \lim_\kappa(p)$.

\noindent
2) If $\bar p = \langle p_i:i < \delta\rangle$ is a sequence of members of
$\bbQ_\kappa$, $\bar p$ is increasing or at least $i < j < \delta
\Rightarrow \tr(p_j) \in p_i$, $\langle \tr(p_i):i < \delta\rangle$ is
$\trianglelefteq$--increasing and 
\begin{enumerate}
\item[$(\odot)$] $\alpha < \delta\quad \Rightarrow\quad\min\big(S_{p_\alpha} 
  \setminus \sup\{\ell g(\tr(p_i) +1:i < \delta\}\big) > \delta$,
\end{enumerate}
\then \, $p_\delta = \bigcap\{p_i:i < \delta\}$ is a $\le_{\bbQ_\kappa}$--lub
of $\bar p$.

\noindent
3) If $\delta < \kappa$, $p_i \in \bbQ_\kappa$ is
$\le_{\bbQ_\kappa}$--increasing with $i <\delta$, $(\eta_i,S_i,
\bar\Lambda_i,E_i)$ is a full witness for $p_i$ satisfying $i<j< \delta
\Rightarrow E_j \subseteq E_i \wedge \min(E_i) < \ell g(\tr(p_j))$, \then \,
the sequence $\langle p_i:i < \delta\rangle$ has a
$\le_{\bbQ_\kappa}$--upper bound. 

\noindent
4) If $p \in \bbQ_\kappa$ and $\cI_i$ is a dense subset of $\bbQ_\kappa$ for
$i<i(*)$ and $i(*) < \kappa^+$ and $\rho \in p$ \then \, there is $\eta$
such that $\rho \triangleleft \eta \in \lim_\kappa(p)$ and $(\forall i <
i(*))(\exists q \in \cI_i)(\eta \in \lim_\kappa(q))$.

\noindent
5) In (2) we may replace the demand $(\odot)$ with 
\begin{enumerate}
\item[$(\otimes)$]  
  \begin{enumerate}
\item[(a)] $\sup\{\ell g(\tr(p_i)):i<\delta\}\notin S_{p_\alpha}$ for  
    $\alpha<\delta$, 
\item[(b)] if $\langle \tr(p_i):i<\delta\rangle$ is eventually constant, say 
  $\rho$, \then \, $\min\big(S_{p_\alpha}\setminus (\ell g(\rho)+1)
  \big)>\delta$. 
  \end{enumerate}
\end{enumerate}
\end{claim}

\begin{PROOF}{\ref{n9}}
We prove by induction on the inaccessibles  $\kappa$ that the five parts of
the claim hold. 
\bigskip

\noindent
1) Let $(\tr(p),S,\bar\Lambda)$ be a witness for $p$. By \ref{n7}(3)
   \wilog \, $\rho \trianglelefteq \tr(p)$.
\medskip

\noindent
\underline{Case 1}:  In $S$ there is a last member $\partial$ and 
$\partial > \ell g(\tr(p))\geq \ell g(\rho)$. \\ 
By \ref{n7}(2), $p_1 = p \cap {}^{\partial >}2$ belongs to $\bbQ_\partial$. 
Apply the induction hypothesis \ref{n9}(4) for $\partial$ with $p 
\cap {}^{\partial >}2, \Lambda_\partial$ here standing for $p,\langle
\cI_i:i < i(*)\rangle$ there to find $\varrho$ such that $\rho 
\triangleleft \varrho \in p \cap {}^\partial 2$. Now $p^{[\varrho]} =
({}^{\kappa >}2)^{[\varrho]}$ by \ref{n7}(5), so the rest should be clear. 
\medskip

\noindent
\underline{Case 2}:  $\sup(S) \le \ell g(\tr(p))$.\\
By \ref{n7}(5) we know that $p=({}^{\kappa >}2)^{[\tr(p)]}$. 
\medskip

\noindent
\underline{Case 3}: Neither Case 1 nor Case 2, i.e., $\sup(S) > \ell
g(\tr(p))$ and $S$ has no last element.\\
Let $\theta = \cf(\otp(S))$ and let $\langle \alpha_\varepsilon:\varepsilon
< \theta\rangle$ be increasing continuous with limit $\sup(S)$.  \Wilog \,
$\alpha_0 = \ell g(\tr(p))$ and $\varepsilon < \theta \Rightarrow
\alpha_{\varepsilon +1} \in S$ and $\omega \varepsilon < \theta \Rightarrow
\alpha_{\omega \varepsilon} \notin S$; recalling that every member of $S$ is
strongly inaccessible and $S$ is nowhere stationary this is clear.  Now we
choose $\eta_\varepsilon \in p \cap {}^{\alpha_\varepsilon}2$ by induction
on $\varepsilon<\theta$ such that $\eta_0=\tr(p)$ and $\zeta<\varepsilon
\Rightarrow \eta_\zeta \trianglelefteq \eta_\varepsilon$.  

If $\varepsilon< \theta$ is limit, then we let $\eta_\varepsilon =
\bigcup\{\eta_\zeta:\zeta < \varepsilon\}$ and we note that it belongs to
$p$ by clause (A)(e) of Definition \ref{n5} (because $\alpha_\varepsilon
\notin S$). 

If $\varepsilon = \zeta +1<\theta$, then we use the induction hypothesis of
part (4) for $\partial = \alpha_\varepsilon$, because
$\alpha_\varepsilon\in S$, a set of inaccessibles.

After the inductive construction is carried out, if $\theta = \kappa$,
i.e., $\sup(S) = \kappa$ then $\eta_\theta:=\bigcup\{\eta_\varepsilon:
\varepsilon < \kappa\}$ is as  required. If $\theta<\kappa$, i.e.,
$\sup(S)<\kappa$ then $\eta_\theta :=\bigcup\{\eta_\varepsilon:\varepsilon <
\theta\}\in p\cap {}^{\sup(S)}2$ (remember Definition \ref{n5}(A)(e)) and
again by \ref{n7}(5) we have $p^{[\eta_\theta]} = ({}^{\kappa
  >}2)^{[\eta_\theta]}$ so we can easily finish.
\bigskip

\noindent
2) Let $(\eta_i,S_i,\bar\Lambda_i)$ be a witness for $p_i \in \bbQ_\kappa$
for $i < \delta$, \wilog \, $S_i = S_{p_i}$, see clause (C) of Definition
\ref{n5} or Claim \ref{n7}(8).  By our assumptions the sequence $\langle
\eta_i:i < \delta\rangle$ is $\trianglelefteq$--increasing and let
$\eta_\delta = \bigcup\{\eta_i:i < \delta\}$.  Now if $i,j < \delta$ and
$i<j$ then $\eta_j = \tr(p_j)\in p_i$ and if $j<i$ then
$\eta_j\trianglelefteq\eta_i= \tr(p_i)$. Hence $\eta_i \in \bigcap\{p_j:j <
\delta\} = p_\delta$ for all $i < \delta$. Consequently, recalling $i <
\delta \Rightarrow \min(S_i\backslash \sup\{\ell g(\tr(p_j))+1:j < \delta\})
> \delta$, we get $\eta_\delta\in p_i$ for all $i <\delta$ and thus
$\eta_\delta\in p_\delta$.   

Let $S := \bigcup\{S_i:i < \delta\} \backslash (\ell g(\eta_\delta)+1)$
and $\bar\Lambda_i = \langle \Lambda_{i,\partial}:\partial \in
S_i\rangle$ and for $\partial \in S$ let $\Lambda_\partial :=
\bigcup\{\Lambda_{i,\partial}:i < \delta$ and $\partial \in S_i\}$.  So
clearly $\Lambda_\partial$ is a set of $\le |\delta| \cdot \partial$
dense subsets of $\bbQ_\partial$.  Also $\partial \in S \Rightarrow
\partial > \delta$ because if $\partial \in S$ then for some $i <
\delta$, $\partial \in S_i$ and by an assumption $\min(S_i \backslash
\sup\{\ell g(\tr(p_i)+1:i < \delta\}) > \delta$ hence $\partial > \delta$.
It follows that $|\Lambda_\partial| \le \partial$.  Now one easily shows
that $\eta_\delta,S,\langle \Lambda_\partial:\partial \in S\rangle$ witness
that $p_\delta = \bigcap\{p_i:i < \delta\}$ belongs to $\bbQ_\kappa$; being
a $\le_{\bbQ_\kappa}$-lub of $\bar p$ is obvious by the definition of
$\le_{\bbQ_\kappa}$.  
\bigskip 

\noindent 
3) \Wilog \, $\delta$ is a limit ordinal. The assumptions on $p_i,E_i$ imply
that $\eta_i \vartriangleleft \eta_j$ when $i<j<\delta$ and  $\delta\leq
\sup\{\ell g(\eta_i):i<\delta\}\in \bigcap_{\alpha<\delta} E_\alpha$. Consequently, 
\[\min\big(S_{p_\alpha}\setminus \sup\{\ell g(\tr(p_i)):i<\delta\}\big)>
\sup\{\ell g(\tr(p_i)):i<\delta\}\geq \delta\]
and we may apply part (2). 
\bigskip 

\noindent
4) \Wilog \, $\rho \trianglelefteq \tr(p)$ (recalling \ref{n7}(3)) and
$i(*)=\kappa$.

First, if $\kappa > \delta_* := \sup\{\partial:\partial < \kappa$
inaccessible$\}$ then by part (1) which, for $\kappa$, was already proven
there is $\eta \in p$ such that $\ell g(\eta) > \delta_*, \ell
g(\tr(p))$. Then $p \le_{\bbQ_\kappa} p^{[\eta]} = ({}^{\kappa
  >}2)^{[\eta]}$ and $p^{[\eta]}\leq_{\bbQ_\kappa} q\ \Rightarrow\
q=\big({}^{\kappa>} 2\big)^{[\tr(q)]}$. Consequently, the claim becomes a
case of the Baire category theorem for ${}^\kappa 2$. 

So we assume that $\delta_* = \kappa$ and by induction on $i<\kappa$ we
choose $p_i,\eta_i,S_i, \bar\Lambda_i,E_i$ such that:
\begin{enumerate}
\item[(a)] $p_i \in \bbQ_\kappa$ and $(\eta_i,S_i,\bar\Lambda_i,E_i)$ is a
  full witness for this, 
\item[(b)] $p\leq p_0$, and $i<j<\kappa\quad \Rightarrow\quad p_i
  \le_{\bbQ_\kappa} p_j$, 
\item[(c)]  $i<j < \kappa\quad \Rightarrow\quad E_j \subseteq E_i \wedge 
  \min(E_i) < \ell g(\tr(p_j))$,
\item[(d)]  for every $i<\kappa$, for some $q_i \in \cI_i$ we have $q_i\leq
  p_i$. 
\end{enumerate}
Why can we carry out the induction? At stage $\delta$ of the construction we
use part (3) which we have already proved to find an upper bound $q$ to
$\{p_i:i<\delta\}\cup\{p\}$. Then, as $\cI_\delta$ is dense, we may pick
$q_\delta\in \cI_\delta$ stronger than $q$. Let $\partial<\kappa$ be an
inaccessible cardinal larger than $\ell g(\tr(q_\delta))$ and
$\sup\{\min(E_i)+1:i<\delta\}$. By part (1) which we have already 
proved there exists $\eta_\delta\in q_\delta\cap {}^\partial 2$. Now it should be
clear that we may choose $p_\delta,S_\delta, \bar\Lambda_\delta,E_\delta$
such that $(\eta_\delta,S_\delta, \bar\Lambda_\delta,E_\delta)$ is a full
witness for $p_\delta\in\bbQ_\kappa$ and $q_\delta\leq p_\delta$ and
$E_\delta\subseteq \bigcap_{i<\delta} E_i$. 

Having carried out the induction, $\eta := \bigcup\{\tr(p_i):i < \kappa\}$
is as required.
\bigskip 

\noindent
5) It can be easily reduced to part (2), but let us elaborate. Without loss
of generality $\delta=\cf(\delta)$ and let
$\nu=\bigcup\{\tr(p_i): i<\delta\}$. For each $i< \delta$, we have
$j\in (i,\delta)\Rightarrow\tr(p_i)\trianglelefteq \tr(p_j)\in p_j$ and
$j<i\Rightarrow \tr(p_i)\in p_j$, so together we have
$\tr(p_i)\in \bigcap \{p_j: j<\delta)\}$. Hence, remembering $(\otimes)$(a),
we have $\nu\in \bigcap\limits_{i<\delta} p_i$. If
$\langle\tr(p_i) : i < \delta\rangle$ is not eventually constant, then
$\lg(\nu)\geq \cf(\delta$), and hence $(\odot)$ of part (2) holds and we are
done. If $\langle\tr(p_i): i< \delta\rangle$ is eventually constant then
also $(\odot)$ of part (2) holds so we are done too. By the last two
sentences we are done.
\end{PROOF}

\begin{claim}
\label{n12}
Assume
\begin{enumerate}
\item[(a)] $\alpha \le \beta < \kappa$,
\item[(b)] $\eta \in {}^\beta 2$,
\item[(c)] $(\tr(p_i),S_i,\bar\Lambda_i)$ witness $p_i\in \bbQ_\kappa$ for 
  $i < \alpha$, 
\item[(d)] $\tr(p_i) \trianglelefteq \eta \in p_i$,
\item[(e)] $S = \bigcup\{S_i:i < \alpha\} \backslash (\ell g(\eta)+1)$,
\item[(f)] for $\partial \in S$ we let $\Lambda_\partial :=
  \bigcup\{\Lambda_{i,\partial}: \partial \in S_i\}$ (so it is a set of
  $\le \partial$ dense subsets of  $\bbQ_\partial$). 
\end{enumerate}
Then $\bigcap\{p^{[\eta]}_i:i < \alpha\} \in \bbQ_\kappa$ is a
$\le_{\bbQ_\kappa}$--lub of $\{p_i^{[\eta]}:i < \alpha\}$ and has the
witness $(\eta,S,\langle \Lambda_\partial:\partial \in S\rangle)$. 
\end{claim}

\begin{PROOF}{\ref{n12}}
  Should be clear.
\end{PROOF}

\begin{observation}
\label{p10}
\begin{enumerate}
\item  If $p,q \in \bbQ_\kappa$ and $\bbQ_\kappa \models ``p \nleq q"$ \then
  \, for some $r$, we have $q \le_{\bbQ_\kappa} r$ and $r,p$ are
  incompatible (so $\lim_\kappa(p),\lim_\kappa(r)$ are disjoint). 
\item If $p_1,p_2 \in \bbQ_\kappa$ \then \, the following conditions are
  equivalent: 
\begin{enumerate}
\item[(a)]  $p_1,p_2$ are compatible, 
\item[(b)]  the sets $\lim_\kappa(p_1),\lim_\kappa(p_2)$ are not disjoint, 
\item[(c)]  $\tr(p_1) \in p_2$ and $\tr(p_2) \in p_1$,
\item[(d)]  $\tr(p_1) \trianglelefteq \tr(p_2) \in p_1$ \underline{or}
  $\tr(p_2) \trianglelefteq \tr(p_1) \in p_2$. 
\end{enumerate}
\item If $p\in\bbQ_\kappa$, \then \, there is a maximal antichain above $p$
  of cardinality $\kappa$.
\item The $\bbQ_\kappa$--name $\name{\eta}_\kappa = \bigcup\{\tr(p):p
  \in \name{\bold G}_{\bbQ_\kappa}\}$ is a name for a $\kappa$--real which
  is generic for $\bbQ_\kappa$, i.e., $\name{\bold G}_{\bbQ_\kappa}$ is
  computable from $\name{\eta}_\kappa$ over $\bold V$. 
\end{enumerate}
\end{observation}

\begin{PROOF}{\ref{p10}}
(1)\quad As $p \nleq q$, by the definition of $\le_{\bbQ_\kappa}$ we have
$q \nsubseteq p$, so we can choose $\nu \in q \setminus p$. Let $r =
q^{[\nu]}$, so $q \le r$ by \ref{n7}(3). Since $\tr(r)=\nu\notin p$, we are
done by (2). 
\medskip

\noindent
(2)\quad First, $(a) \Rightarrow (b)$ as letting $r$ be a common upper bound 
of $p_1,p_2$ we have $\lim_\kappa(r) \subseteq \lim_\kappa(p_1) \cap
 \lim_\kappa(p_2)$ and recall  $r \in \bbQ_\kappa \Rightarrow
\lim_\kappa(r) \ne \emptyset$ by \ref{n9}(1).

Second, $(b) \Rightarrow (c)$ as $\eta \in \lim_\kappa(p_\ell) \ 
\Rightarrow \ \tr(p_\ell) \trianglelefteq \eta \wedge
\{\eta\rest\alpha:\alpha<\kappa\} \subseteq p_\ell$.

Third, $(c) \Rightarrow (d)$ trivially.

Fourth, $(d) \Rightarrow (a)$ as \wilog \, $\tr(p_1) \trianglelefteq
\tr(p_2) \in p_1$, hence $p^{[\tr(p_2)]}_1,p_2$ are members of
$\bbQ_\kappa$ with the same trunk so are compatible by \ref{n12}.  As
$\bbQ_\kappa \models ``p_1 \le p_1^{[\tr(p_2)]}"$, we are done.
\medskip

\noindent
(3)\quad Let $\eta\in\lim_\kappa(p)$ and for $\alpha\in [\ell g(\tr(p)),
\kappa)$ let $\nu_\alpha=(\eta\rest \alpha) \char 94 \langle 1-\eta(\alpha)
\rangle$. Then $\{p^{[\nu_\alpha]}:\alpha\in [\ell g(\tr(p)),\kappa) \}$ is
as required. 
\medskip

\noindent
(4)\quad Should be clear.
\end{PROOF}

\begin{claim}
\label{n11}
\begin{enumerate}
\item $\bbQ_\kappa$ is $\kappa$--strategically closed.
\item $\bbQ_\kappa$ satisfies the $\kappa^+$--c.c.
\end{enumerate}
\end{claim}

\begin{PROOF}{\ref{n11}}
(1) Immediate by \ref{n9}(3).
\medskip

\noindent
(2) Obviously
\begin{enumerate}
\item[$(*)_1$]  ${}^{\kappa >}2$ has cardinality $\kappa$ (recall that
  $\kappa$ is inaccessible), and 
\item[$(*)_2$]  if $p_1,p_2 \in \bbQ_\kappa$ have the same trunk \then \,
  they are compatible.
\end{enumerate}
Together we are clearly done. 
\end{PROOF}

\begin{claim}
\label{n13}
1) If $\kappa$ is weakly compact \then \, $\bbQ_\kappa$ is
$\kappa$-bounding, i.e. for every $f \in ({}^\kappa \kappa)^{\bold
  V[\bbQ_\kappa]}$ there is $g \in ({}^\kappa \kappa)^{\bold V}$ such that
$f \le g$, that is, $\alpha < \kappa \Rightarrow f(\alpha) \le g(\alpha)$.

\noindent
2) Moreover, if $p \Vdash_{\bbQ_\kappa} ``\name f \in {}^\kappa
\kappa"$ and $\beta < \kappa$ then for some $\bar\beta$ and
$q\in\bbQ_\kappa$ we have:
\begin{itemize}
\item  $p \le q$,
\item $p \cap {}^{\beta \ge} 2 = q \cap {}^{\beta \ge}2$,
\item $\bar \beta = \langle \beta(i):i <\kappa\rangle$ is increasing
  continuous, $\beta(0)\geq\beta$, $\beta(i)<\kappa$,
\item if $\nu \in q \cap {}^{\beta(i+1)}2$ then $q^{[\nu]}$ forces a value
  to $\name f(i)$. 
\end{itemize}
\end{claim}

\begin{PROOF}{\ref{n13}}
  2)\quad Let $p \Vdash ``\name f \in{}^\kappa \kappa"$.  By induction on $i
  < \kappa$ we choose $p_i,\beta(i),\varrho_i,S_i,\bar\Lambda_i$ and $E_i$
  such that 
\begin{enumerate}
\item[(i)] $p_i \in \bbQ_\kappa$,
\item[(ii)] $\langle \beta(j):j \le i\rangle$ is an increasing continuous
  sequence of ordinals $< \kappa$,
\item[(iii)] $p_0 = p$ and $\beta(0) = \max\big\{\beta,\ell g(\tr(p))
  +1\big\}$,   
\item[(iv)] $(\varrho_i,S_i,\bar\Lambda_i,E_i)$ is a full witness for $p_i
  \in \bbQ_\kappa$, 
\item[(v)] if $j<i$ then 
\begin{enumerate}
\item[$(\alpha)$] $p_j \le_{\bbQ_i} p_i$,
\item[$(\beta)$]   $p_j \cap {}^{\beta(j) \ge} 2 = p_i \cap {}^{\beta(j)
    \ge} 2$ (hence $\varrho_i = \varrho_0$), and $S_j\cap \big(\beta(j)+
  1\big)= S_i\cap\big(\beta(j)+1\big)$, $\bar{\Lambda}_j\rest
  \big(\beta(j)+1\big)=\bar{\Lambda}_j\rest \big(\beta(j)+1\big)$,
\item[$(\gamma)$] $\beta(i) \in E_j$,
\item[$(\delta)$] $E_i \subseteq E_j$ and if $i$ is limit then
  $E_i=\bigcap_{\alpha<i} E_\alpha$,
\end{enumerate}
\item[(vi)] if $i = j +1$ and $\nu \in p_i \cap {}^{\beta(i)}2$ then
  $p^{[\nu]}_i$ forces a value to $\name f(j)$. 
\end{enumerate}

For $i=0$ choose a full witness $(\varrho_0,S_0,\bar\Lambda_0,E_0)$
for $p$, and use clause (iii) to define $p_0,\beta(0)$.

For a limit $i<\kappa$ work as in the proof of \ref{n9}(2).

For a successor $i$, say $i=j+1$, we shall use the definition of ``$\kappa$
is weakly compact''.  Let $\langle q_{j,\beta}: \beta < \beta(*)\rangle$ be
a maximal antichain of $\bbQ_\kappa$ such that $q_{j,\beta} \Vdash ``\name
f(j) = \gamma"$ for some $\gamma = \gamma_{j,\beta}$ and $q_{j,\beta}$ is
$\le_{\bbQ_\kappa}$--above $p_j$ or $\lim_\kappa(q_{j,\beta}) \cap
\lim_\kappa(p_j) = \emptyset$, recalling \ref{n7}(9). Since $\bbQ_\kappa$
satisfies the $\kappa^+$-c.c., see \ref{n11}(2), we know that $\beta(*) \le
\kappa$, so by \ref{p10}(3) \wilog \, $\beta(*) = \kappa$.  
Recalling each $S_{q_{j,\beta}}$ is nowhere stationary, clearly there is a
club $E$ of $\kappa$ such that 
\[\beta < \delta \in E\quad \Rightarrow\quad \delta \in E_j \setminus 
S_{q_{j,\beta}}\mbox{ and hence also }\delta \notin S_{p_j}.\]  
By the weak compactness there is a strongly inaccessible cardinal
$\partial(j) > \beta(j)$ belonging to $E$ such that  $\{q_{j,\beta} \cap
{}^{\partial(j)   >}2:\beta < \partial(j)\}$ is a pre-dense subset of 
$\bbQ_{\partial(j)}$. Let  
\[\cI=\big\{q \in \bbQ_{\partial(j)}:\mbox{ for some }\beta<\partial(j)
\mbox{ we have }(q_{j,\beta} \cap {}^{\partial(j)>}2) \le_{\bbQ_{\partial(j)}}
q\big\}.\]
Clearly, $\cI$ is a dense open subset of $\bbQ_{\partial(j)}$.  Let 
\[\cX=\big\{\eta \in p_j \cap {}^{\partial(j)}2:\big(\exists 
\beta < \partial(j)\big)\big(\eta \in q_{j,\beta} \cap 
{}^{\partial(j)}2\big)\big\}.\]
For each $\rho \in \cX$ there is $r_{j,\rho} \geq p_j$ such that
$\tr(r_{j,\rho})=\rho$ and $r_{j,\rho}$ forces a value to
$\name{f}(j)$. Indeed, there is $\beta < \partial(j)$ such that
$\rho\in q_{j,\beta} \cap {}^{\partial(j)}2$, so by our
assumptions on the $q_{j,\beta}$'s necessarily $p_j \le q_{j,\beta}$, so
$q^{[\rho]}_{j,\beta}$ can serve as $r_{j,\rho}$.  Let
$(\rho,S_{j,\rho},\bar\Lambda_{j,\rho})$ witness
$r_{j,\rho} \in \bbQ_\kappa$. Lastly, we let
\begin{enumerate}
\item[(a)]   $p_i = \bigcup\{r_{j,\rho}:\rho \in \cX\}$, 
\item[(b)]   $\beta(i) = \min \big(E_j\setminus(\partial(j) +1)\big)$,  
\item[(c)]   $S_i = S'_i \cup S''_i \cup \{\partial(j)\}$, where
\[S'_i =\bigcup\big\{S_{r_{j,\rho}}:\rho \in (p_i \cap {}^{\partial(j)}2)
\big\}\backslash (\partial(j)+1) \quad\mbox{ and }\quad S''_i = S_j 
\cap \partial(j),\]  
\item[(d)]  $\bar\Lambda_i = \langle \Lambda_{i,\partial}:\partial\in
  S_i\rangle$, where  
\begin{enumerate}
\item[$(\alpha)$] $\Lambda_{i,\partial}$ is $\Lambda_{j,\partial}$ if
  $\partial \in S''_i$, and 
\item[$(\beta)$] $\Lambda_{i,\partial}$ is
  $\bigcup\{\Lambda_{j,\rho, \partial}: \rho 
  \in p_i \cap {}^{\partial(j)}2$ and $\partial\in S_{r_{j,\rho}}\}$ if
  $\partial \in S'_i$,  
\item[$(\gamma)$] $\Lambda_{i,\partial(j)}$ is $\{\cI\}$, 
\end{enumerate}
\item[(e)] $E_i$ is $E\setminus (\beta(i)+1)$ or just a club of $\kappa$
  which is $\subseteq E_j\backslash\beta(i)$ and is disjoint to
  $S_{r_{j,\rho}}$ for every $\rho \in \cX$.    
\end{enumerate}
It should be clear that the objects defined above have the desired
properties. 

So we can carry out the induction on $i < \kappa$.  After it is completed we
define 
\begin{enumerate}
\item[$(*)_1$] $q = \bigcap\{p_i:i < \kappa\}$,
\item[$(*)_2$] $S = \bigcup\{S_i:i < \kappa\}$,
\item[$(*)_3$] $\bar\Lambda = \langle \Lambda_\partial:\partial \in
  S\rangle$ where $\Lambda_\partial = \bigcup\{\Lambda_{i,\partial}:i <
  \kappa$ satisfies $\partial \in S_i\}$ and
\item[$(*)_3$]  $E = \{\delta < \kappa:\delta = \beta(\delta)\mbox{ is a limit ordinal such
that }i < \delta \Rightarrow \delta \in E_i\}$.
\end{enumerate}
It easily follows from conditions (i)--(vi) that:
\begin{enumerate}
\item[$(\oplus)_1$]  $q \in \bbQ_\kappa$ has trunk $\varrho_0$,
\item[$(\oplus)_2$]  $(\varrho_0,S,\bar\Lambda,E)$ is a full witness for $q
  \in \bbQ_\delta$, 
\item[$(\oplus)_3$] $p \le_{\bbQ_\kappa} q$ and $p\cap {}^{\beta\geq} 2
  =q\cap {}^{\beta\geq} 2$,  
\item[$(\oplus)_4$] if $\nu \in q\cap {}^{\beta(j+1)}2$, then
  $q^{[\nu]}$ forces a value to $\name f(j)$. 
\end{enumerate}
\medskip

1)\quad Follows from (2) proven above: $(\oplus)_4$, that is the last bullet 
in \ref{n13}(2),  suffices for defining a function $g \in \bold V$ such that
$q$ forces that it bounds $\name f$, we are done. 
\end{PROOF}

\begin{conclusion}
\label{a27}
\begin{enumerate}
\item If $\kappa$ is a weakly compact cardinal \then \, there is a
  $({<} \kappa)$--strategically complete, $\kappa^+$-c.c.,
  $\kappa$--bounding forcing notion (hence not adding a $\kappa$-Cohen), and
  of course, adding a new $\eta \in {}^\kappa 2$.
\item In fact, the forcing is $\kappa$--Borel and is
  $\kappa$--strategically complete and it is equivalent to a $({<}
  \kappa)$--complete forcing notion (which necessarily is
  $\kappa^+$--c.c. $\kappa$--bounding adding a new subset to $\kappa$).
  Also, the forcing is definable even without parameters.
\end{enumerate}
\end{conclusion}

\begin{PROOF}{\ref{a27}}
(1)\quad See above.
\medskip

\noindent
(2)\quad Note that when $\kappa$ is not weakly compact, $\bbQ_\kappa$ is not
$\kappa$--Borel because ``nowhere stationary'' is not.  However, if we
replace the conditions by full witnesses of conditions with the natural
order, this becomes easy.  
\end{PROOF}
\bigskip

\subsection {Adding a dominating member of 
$\prod\limits_{\varepsilon < \lambda} \theta_\varepsilon$}
\label{Addinga} 

Here we present a variant of the forcing from \S(1A), this time
dealing with sequences from $\prod\limits_{\varepsilon < \lambda}
\theta_\varepsilon$ instead of ${}^\lambda 2$ and we have an
$|\varepsilon|^+$-complete filter $D_\varepsilon$ on
$\theta_\varepsilon$ for $\varepsilon < \lambda$.  The main case is
$D_\varepsilon=\{a\subseteq \theta_\varepsilon: |\theta_\varepsilon
\setminus a|<\theta_\varepsilon\}$, so we write only this case, but the
changes needed for the general case are minor. This is also true for
$\langle \theta_\eta,D_\eta:\eta\in\cT\rangle$ and $\cT=\{\nu: \varepsilon <
{\ell g}(\nu)\Rightarrow \nu(\varepsilon)<\theta_{\nu\rest \varepsilon}\}$. 
So our starting point, e.g. the forcing for the first $\kappa$, is not the  
$\kappa$-Cohen forcing but $\bbQ_\theta$ of \cite{Sh:945}, which is
the parallel for $\kappa$ of the forcing of \cite{Sh:326} for
$\lambda=\aleph_0$. 

Note that Definitions \ref{a3}, \ref{a4} are used in \cite{Sh:F1580}, too. Also
note that $\bbQ_{\bar\theta}$ is the ``one step" forcing on which we shall
build later.  

The reader may ignore the version with $\bar{\cP}$, i.e., use the default
$\cP_\kappa=\cP(\cH(\kappa))$. 

\begin{remark}
  For $\bar\theta = \langle \theta_\alpha:\alpha <
  \kappa\rangle$, $\bbQ_{\bar\theta}=\bbQ^1_\kappa$ was designed to
  make the old $\kappa$-reals $\kappa$--meagre, we still have to
  expect it to behave like random real forcing and do this indeed.
\end{remark}

\begin{definition}
\label{a4}
1) Recall the weakly compact ideal on $\lambda$ is $I^{\wc}_\lambda =
\{A \subseteq \lambda$: for some first order formula $\varphi(X,Y)$
and $B \subseteq \cH(\lambda)$ we have $(\forall X \subseteq
\cH(\lambda))(\cH(\lambda) \models \varphi[X,B])$ but for no strongly
inaccessible $\kappa \in A$ do we have $(\forall X \subseteq
\cH(\kappa))(\cH(\kappa) \models \varphi[X,B \cap \cH(\kappa)]\}$.

\noindent
2) $\diamondsuit_{S_*,I^{\wc}_\lambda}$ means that some $\bar A =
\langle A_\alpha:\alpha \in S_*\rangle$ is an
$I^{\wc}_\lambda$-diamond sequence, which means: for every $A
\subseteq \cH(\lambda)$ the set $\{\kappa \in S_*: A \cap \cH(\kappa)
= A_\kappa\}$ is $\ne \emptyset \mod I^{\wc}_\lambda$.

\noindent
3) We say $\bar{\cP} = \langle \cP_\alpha:\alpha \in S_*\rangle$ is
$I^{\wc}_\lambda$-positive \when \, $S_* \in (I^{\wc}_\lambda)^+$ and
$(\cP_\alpha,\alpha,\in)$ and $(\cP(\alpha),\alpha,\in)$ have the same first 
order theory,  and 
moreover $(a) \Rightarrow (b)$ where \mn
\begin{enumerate}
\item[(a)]  $\varphi(X,Y)$ is first order,
$A \subseteq \cH(\lambda)$ satisfies $X \subseteq
   \cH(\lambda) \Rightarrow (\cH(\lambda),\in) \models \varphi[X,A]$,  
\item[(b)]  $(\exists^{I^{\wc}_\lambda}\kappa \in S_*)
[A \cap \cH(\kappa) \in \cP_\kappa$ and $X \subseteq \cH(\kappa)
   \Rightarrow (\cH(\kappa),\in) \models \varphi[X,A \cap \kappa]]$. 
\end{enumerate}

\noindent
4) The default value of $\bar{\cP}$ is $\langle\cP(\cH(\kappa)):\kappa\in
S_\gx\rangle$.  
\end{definition}

\begin{definition}
\label{a3}
1) We say $\gx$ is a 1-ip when $\gx$ consists of:
\begin{enumerate}
\item[(A)]  a weakly compact cardinal $\lambda$,
\item[(B)]   a sequence $\bar\theta = \langle \theta_\varepsilon:
  \varepsilon < \lambda\rangle$, where 
\[\varepsilon < \lambda\quad\Rightarrow\quad  (2 \le
\theta_\varepsilon < \aleph_0) \vee (\varepsilon < \theta_\varepsilon
= \cf(\theta_\varepsilon) < \lambda),\] 
\item[(C)]  a stationary set $S_{\gx} \subseteq \lambda$ of strongly
  inaccessible cardinals satisfying 
\[\zeta < \kappa \in S_{\gx}\quad \Rightarrow\quad
  \prod\limits_{\varepsilon < \zeta} \theta_\varepsilon < \kappa,\]  
\item[(D)]  
\begin{enumerate}
\item[(a)] $\diamondsuit_{S_{\gx},I^{\wc}_\lambda}$, i.e. diamond on $S_{\gx}$ 
  holds even modulo the weakly compact ideal, or just  
\item[(b)] $\bar{\cP} = \langle \cP_\kappa \subseteq \cP(\cH(\kappa)):
  \kappa \in S_{\gx}\rangle$ is $I^{\wc}_\lambda$--positive, see
  Definition \ref{a4}(3)  above, so necessarily $S_{\gx} \in 
  (I^{\wc}_\lambda)^+$; the default value is $\cP_\kappa =
  \cP(\cH(\kappa))$,
\end{enumerate} 
\item[(E)]  $S^*_{\gx} := \{\kappa \le \lambda:\kappa$ weakly
compact and $S_{\gx}\cap \kappa \in (I^{\wc}_\kappa)^+$ moreover the
sequence $\bar{\cP} \rest (S_{\gx}\cap \kappa)$ is
$I^{\wc}_\kappa$--positive (see \ref{a4}(3)) $\}$. 
\end{enumerate}
\noindent 2) If $\kappa \in S^*_{\gx}$ we may say ``$\kappa$ is
$\gx$-weakly compact". 

\noindent 3) Let $\bar\theta = \langle \theta_\varepsilon: \varepsilon
<\lambda\rangle$ be as in clause 1(B) (we will fix it for this
sub-section). Define $\bold T_\alpha = \prod\limits_{\varepsilon <
  \alpha} \theta_\varepsilon$ for $\alpha < \lambda$ and $\bold T_{<
  \alpha} = \bigcup\{\bold T_\beta:\beta < \alpha\}$ for $\alpha \le 
\lambda$.
\end{definition}

\begin{convention}
\label{a7}
For this subsection

\noindent
0) $\gx$ is as in Definition \ref{a3}.

\noindent
1) Let $\kappa,\partial$ denote members of $S_{\gx}$.

\noindent
2) Always $p$ is a subtree of $\bold T_{<\kappa}$, for some $\kappa
\le \lambda$, typically it belongs to $\bbQ^1_\kappa$ for some $\kappa
\le \lambda$ and for $\eta \in p$ let $p^{[\eta]} = \{\nu \in p:\nu
\trianglelefteq \eta$ or $\eta \trianglelefteq \nu\}$.
\end{convention}

\begin{definition}
\label{a8}
We define the forcing notion $\bbQ^1_\kappa$ by induction on $\kappa$
(so $\kappa \in S_{\gx}$) as follows: 
\begin{enumerate}
\item[(A)]  $p \in \bbQ^1_\kappa$ \Iff \, some $S \subseteq \kappa
\cap S_{\gx}$ witnesses it, which means
\begin{enumerate}
\item[(a)]  $p$ is a subtree of $\bold T_{< \kappa}$,
\item[(b)]  $p$ has trunk $\tr(p) \in \bold T_{< \kappa}$ that is
\begin{itemize}
\item  $\beta \le \ell g(\tr(p)) \Rightarrow p \cap \bold T_\beta =
  \{\tr(p) \rest \beta\}$ but 
\item  $(\exists^{\ge 2} \alpha)(\tr(p) \char 94 \langle\alpha \rangle
  \in p)$, 
\end{itemize}
\item[(c)]  if $\eta \in p \wedge \ell g(\tr(p)) \le \ell g(\eta) <
  \beta < \kappa$ then $(\exists \nu)(\eta \triangleleft \nu \in p
  \cap \bold T_\beta)$, follows from the rest,
\item[(d)]  if $\eta \in p$ and $\ell g(\tr(p))\le\ell g(\eta) <
  \kappa$ then\footnote{Remember ``$\forall^\infty i<\theta$'' means ``for
    all but boundedly many $i<\theta$''.}   
\begin{itemize}
\item if $\theta_{\ell g(\eta)} \ge \aleph_0$ \then \,
  $(\forall^\infty i < \theta_{\ell g(\eta)})[\eta \char 94 \langle i
  \rangle \in p]$,
\item if $\theta_{\ell g(\eta)} < \aleph_0$ \then \, $(\forall i <
  \theta_{\ell g(\eta)})(\eta \char 94 \langle i \rangle \in p)$,
\end{itemize}
\item[(e)]  if $\delta \in \kappa \backslash S$ is a limit ordinal
  and $\eta \in \bold T_\delta := \prod\limits_{\varepsilon < \delta}
  \theta_\varepsilon$,

then $\eta \in p \Leftrightarrow (\forall \beta < \delta)(\eta \rest
\beta \in p)$, 
\item[(f)]  if $\partial \in \kappa \cap S$ hence $\partial \in S_{\gx}$  
  so is strongly inaccessible, \then \, $p \cap \bold T_{<\partial}
  \in \bbQ^1_\partial$ and for some predense  subsets $\cI_i$ of
  $\bbQ^1_\partial$ for $i < i_* \le \partial$,  [if we have
  $\bar{\cP}$ also $\cI_i \in \cP_\kappa$]  for every $\eta\in \bold
  T_\partial$ we have:  
\begin{itemize}
\item  $\eta \in p$ iff $(\forall \beta <
\partial)(\eta \rest \beta \in p)$ and $(\forall i < i_*)(\exists q
\in \cI_i)(\forall \beta < \partial) (\eta \rest \beta \in q)$, 
\end{itemize}
\item[(g)] $S \subseteq \kappa \cap S_{\gx}$ is not stationary in any
inaccessible $\partial \le \kappa$, even if $\partial \notin S_{\gx}$ (yes
also for $\partial = \kappa$), equivalently for any limit $\delta \le
\kappa$ as $S_{\gx}$ is a set of inaccessibles and $S \subseteq S_{\gx}$. 
\end{enumerate}
\item[(B)]  $\le_{\bbQ^1_\kappa}$ is the inverse inclusion.
\end{enumerate}
\end{definition}

\begin{claim}
\label{a11}
1) $\bold T_{< \kappa}$ belongs to $\bbQ^1_\kappa$ and 
\begin{itemize}
\item $p \in \bbQ^1_\kappa\quad \Rightarrow\quad \bbQ^1_\kappa \models
  ``\bold T_{< \kappa} \le  p$'', and 
\item $\eta \in p \in \bbQ^1_\kappa\quad \Rightarrow\quad p
  \le_{\bbQ^1_\kappa} p^{[\eta]} \in \bbQ^1_\kappa$.
\end{itemize}

\noindent 2) For $p \in \bbQ^1_\kappa$ and $\alpha < \kappa$ the set
$\{p^{[\eta]}:\eta    \in p \cap \bold T_\alpha\}$ is predense in
$\bbQ^1_\kappa$ above $p$. 

\noindent 3) If $p \in \bbQ^1_\kappa$ and $\ell g(\tr(p)) < \partial <
\kappa$ \then \, $p \cap \bold T_{<\partial} \in \bbQ^1_\partial$. Moreover,
if $p_\ell\in\bbQ_\kappa^1$, $\ell g(\tr(p_\ell))<\partial<\kappa$ for
$\ell=1,2$, then  
\[p_1\leq_{\bbQ_\kappa^1} p_2\quad \Rightarrow\quad p_1\cap {\bold 
  T}_{<\partial} \leq _{\bbQ^1_\partial} p_2\cap {\bold 
  T}_{<\partial},\]
and 
\[p_1\perp_{\bbQ_\kappa^1} p_2\quad \Rightarrow\quad p_1\cap {\bold 
  T}_{<\partial} \perp_{\bbQ^1_\partial} p_2\cap {\bold 
  T}_{<\partial}.\]
\noindent
4) $\bbQ^1_\kappa$ is a forcing notion and it satisfies the
$\kappa^+$-c.c. Moreover, it is $\kappa^+$-centered as if $p,q \in
\bbQ^1_\kappa$ have the same trunk \then \, $p,q$ are compatible, in
fact, $p \cap q$ belongs to $\bbQ^1_\kappa$ and is a
$\le_{\bbQ^1_\kappa}$-lub with the same trunk.

\noindent
5) Suppose that $\nu \in \bold T_\gamma$ and $p_i \in\bbQ^1_\kappa$,
$\tr(p_i)=\nu$ for $i < i(*)$  and assume that 
\begin{enumerate}
\item[$(\boxdot)$] either $i(*) \le \gamma$, or   
\[(\forall \varepsilon)[\ell g(\nu) \le \varepsilon 
< \kappa \wedge \theta_\varepsilon\geq\aleph_0 
\Rightarrow i(*) < \theta_\varepsilon]\quad\mbox{ and }\quad 
i(*) < \min(S_{\gx} \backslash (\ell g(\nu) +1)).\] 
\end{enumerate}
Then $p = \bigcap\{p_i:i <i(*)\}$ belongs to $\bbQ^1_\kappa$, has the
trunk $\nu$ and is a $\le_{\bbQ^1_\kappa}$--lub of $\{p_i:i < i(*)\}$.  

\noindent 6) $p,q \in \bbQ^1_\kappa$ are incompatible \Iff \, $\tr(p)
\notin q \vee \tr(q) \notin p$.

\noindent
7) If $\nu \in {\bold T}_\gamma$, $p_i \in \bbQ^1_\kappa$, and
${\rm tr}(p_i)\trianglelefteq \nu \in p_i$ for $i<i(*)$ and $(\boxdot)$ of
part (5) holds, \then\, $p = \bigcap\{p^{[\nu]}_i:i < i(*)\}$ is a lub of
$\{p_i^{[\nu]}:i<i(*)\}$ in $\bbQ^1_\kappa$ and has trunk $\nu$.

\noindent
8) $\name\eta = \bigcup\{\tr(p):p \in \name{\bold G}_{\bbQ^1_\kappa}\}$ is a
$\bbQ^1_\kappa$-name of a member of
$\prod\limits_{\varepsilon <\kappa} \theta_\varepsilon$. 

\noindent
9) If $\nu \in \prod\limits_{\varepsilon < \kappa} \theta_\varepsilon$
\then \, $\Vdash_{\bbQ^1_\kappa}$ ``for arbitrarily large $\varepsilon
< \kappa$ we have $\name\eta(\varepsilon) \ne \nu(\varepsilon)$ and
for every $\varepsilon < \kappa$ large enough $\theta_\varepsilon \ge
\aleph_0 \Rightarrow \name\eta(\varepsilon) > \nu(\varepsilon)"$.

\noindent
10) $\name\eta$ is a new branch of $\bold T_{< \kappa}$ and is generic
for $\bbQ^1_\kappa$, i.e. $\name{\bold G} = \{p \in
\bbQ^1_\kappa:\name\eta$ is a branch of $p\}$.

\noindent
11) $\bbQ^1_\kappa$ is $(< \kappa)$-strategically complete. 
\end{claim}

\begin{PROOF}{\ref{a11}}
1), 2), 3)  Straightforward (for the second sentence of (3) use part (6)).

Concerning parts (4), (5) and (6), see more in \ref{a15} and \ref{a17}.

\noindent
4) By (7) and the number of possible trunks of $p \in \bbQ^1_\kappa$
is $|\bold T_{< \kappa}| = \kappa$.

\noindent
5) By (7).

\noindent
6) Clearly if $\tr(p) \notin q$ then $p,q$ are incompatible, and
similarly if $q \notin \tr(p)$ so the implication ``if" holds.  For
the other direction assume $\tr(p) \in q \wedge \tr(q) \in p$, and we
shall prove that $p,q$ are compatible.  By symmetry \wilog \, $\ell
g(\tr(p)) \le \ell g(\tr(q))$, let $\nu = \tr(q)$. Now $p^{[\nu]}$ and
$q =q^{[\nu]}$ have the same trunk, so we are done by part (4).

\noindent
7) Let $S_i$ be a witness for $p_i \in \bbQ^1_\kappa$, and let $S =
\bigcup\{S_i:i < i(*)\} \backslash (\ell g(\nu)+1)$. We shall prove
that $S$ witnesses that $p = \bigcap\{p_i^{[\nu]}:i <i(*)\}$ belongs
to $\bbQ^1_\kappa$, then we are done as obviously $i < i(*)
\Rightarrow p \subseteq p_i^{[\nu]}$ by the choice of $p$. 

If $\partial \le \ell g(\nu)$ then $\partial \cap S = \emptyset$ and
if $\ell g(\nu) < \partial < \kappa$, \then \, each $S_i
\cap \partial$ is not a stationary subset of $\partial$ for $i<i(*)$.
Also $i(*) <\partial$.

\noindent [Why?  If $i(*) \le \ell g(\nu)$ clear, if $i(*) > \ell
g(\nu)$, \then \, $S\cap [\ell g(\nu),i(*)] = \emptyset$ by
assumption as $\partial > \ell g(\nu)$ clearly $i(*) < \partial$.]
Together also $S = \bigcup\{S_i:i<i(*)\}$ is not stationary in
$\partial$; that is, clause (g) of \ref{a8}(A) holds.

Now obviously $p$ is a subtree of $\bold T_{< \kappa}$, i.e. (a) of
\ref{a8}(A) holds.  Also obviously $\alpha \le \ell g(\nu) \Rightarrow
p \cap \bold T_\alpha = \{\nu \rest \alpha\}$ and $p \cap \bold
T_{\ell g(\nu) +1} \subseteq \{\nu \char 94 \langle \iota
\rangle:\iota < \theta_{\ell g(\nu)}\}$.  To prove clauses (b), (d)
assume that $\eta \in p \cap \bold T_\varepsilon$ and $\nu
\trianglelefteq \eta$.  If $\theta_\varepsilon < \aleph_0$ then
clearly $n < \theta_\varepsilon \wedge i < i(*) \Rightarrow \eta \char
94 \langle n \rangle \in p_i$ hence $\{\eta \char 94 \langle \iota
\rangle:\iota < \theta_\varepsilon\} \subseteq p \cap \bold T_{\ell
  g(\eta)+1}$ so equality holds.  Hence clause (d) holds in this case,
and for $\varepsilon = \ell g(\nu)$ so $\eta = \nu$ then $\nu$ is
indeed the trunk of $p$ and \ref{a8}(A)(b) holds.

If $\theta_\varepsilon \ge \aleph_0$ then $\theta_{\ell g(\eta)} =
\cf(\theta_{\ell g(\eta)}) >i(*)$. Now, for each $i<i(*)$
there is $\iota(i) < \theta_\varepsilon$ such that $\{\eta \char 94 \langle
\iota):\iota \in [\iota(i),\theta_\varepsilon)\} \subseteq p_i$ and hence 
$\iota(*) = \sup\{\iota(i):i<i(*)\} < \theta_\varepsilon$. Thus $\{\eta
\char 94 \langle \iota \rangle:\iota \in [\iota(*), \theta_\varepsilon)\}
\subseteq p$ and again clause (d) holds in this case, and for $\varepsilon =
\ell g(\nu)$ so $\eta = \nu$, clearly $\tr(p)$ is well defined and equal to
$\nu$, so \ref{a8}(b) holds. 

The proof of clause \ref{a8}(A)(c) follows from the rest.

The proofs of clauses (e), (f) are straightforward and clause (g)
holds by the choice of $S$.

\noindent 8)--11)  Left to the reader. 
\end{PROOF}

\begin{observation}
\label{a12}
If $p \le_{\bbQ^1_\kappa} q$ and $S$ is a witness for $q$ and $\tr(p)
   = \tr(q)$ \then \, $S$ is a witness for $p$.
\end{observation}

\begin{definition}
\label{a15}
Let $\kappa \in S_{\gx}$.

\noindent
1) For $\gamma<\kappa$ let $\bold S^{\incr}_{\kappa,\gamma}$ be the set of
sequences $\langle (p_\alpha,q_\alpha,E_\alpha):\alpha < \gamma\rangle$ 
satisfying\footnote{may add: $(h) \quad$ if $\delta < \gamma$ is a
  limit ordinal then $p_\delta = \cap\{p_\alpha:\alpha < \delta\}$, we
  do not use this} \mn
\begin{enumerate}
\item[(a)]  $p_\alpha \in \bbQ^1_\kappa$,
\item[(b)]  $q_\alpha \in \bbQ^1_\kappa$.
\item[(c)]  $\beta < \alpha \Rightarrow q_\beta \le_{\bbQ^1_\kappa}
  p_\alpha$, 
\item[(d)]  $E_\alpha$ is a club of $\kappa$ disjoint to some witness
  for $q_\beta \in \bbQ^1_\kappa$ for every $\beta < \alpha$, 
\item[(e)]  $p_\alpha \le_{\bbQ^1_\kappa} q_\alpha$, 
\item[(f)]  $\ell g(\tr(p_\alpha)) \ge \alpha$, 
\item[(g)]  $\ell g(\tr(p_\alpha)) \in \bigcap\{E_\beta:\beta <
  \alpha\}$. 
\end{enumerate}

\noindent
2) For $\gamma\leq \kappa$ let $\bold S^{\incr}_{\kappa,< \gamma} =
\bigcup\{\bold S^{\incr}_{\kappa,\beta}: \beta < \gamma\}$ and $\bold  
S^{\incr}_\kappa = \bold S^{\incr}_{\kappa,< \kappa}$.

\noindent
3) For $\gamma \le \kappa$ let $\bold S^{\pr}_{\kappa,\gamma}$ be the
set of sequences $\langle (p_\alpha,q_\alpha,E_\alpha):\alpha <
   \gamma\rangle$ such that
\begin{enumerate}
\item[(a)]  $p_\alpha,q_\alpha \in \bbQ^1_\kappa$ have trunks $\tr(p_0)$, 
\item[(b)]  $E_\alpha$ is a club of $\kappa$ disjoint to $\ell g(
  \tr(p_0))$ such that for every $\beta <\alpha$, $E_\alpha$ is
  disjoint to some witness of $q_\beta\in \bbQ^1_\kappa$, 
\item[(c)]  $\min(E_\alpha) \ge \alpha$ is increasing (for
  transparency), 
\item[(d)]  $p_\alpha \le_{\bbQ^1_\kappa} q_\alpha$,
\item[(e)]  $q_\beta \le_{\bbQ^1_\kappa} p_\alpha$ when $\beta <
  \alpha$,
\item[(f)]  if $\beta < \alpha$ \then \, $q_\beta \cap \bold
T_{\min(E_\beta)} \subseteq p_\alpha$, 
\item[(g)]  if $\delta < \gamma$ is a limit ordinal \then \, 
\[p_\delta = \bigcap\big\{p_\alpha:\alpha < \delta\big\}\ \mbox{ and
}\ p_\delta \cap \bold T_{\min(\cap\{E_\alpha:\alpha < \delta\})}
\subseteq q_\beta \mbox{ for }\beta \in [\delta,\gamma).\]
\end{enumerate}
4) $\bold S^{\pr}_{\kappa,< \gamma} = \bigcup\{\bold
S^{\pr}_{\kappa,\beta}:\beta <\gamma\}$ and $\bold S^{\pr}_\kappa =
\bigcup\{\bold S^{\pr}_{\kappa,\gamma}:\gamma <\kappa\}$.
\end{definition}

\begin{claim}
\label{a17}
1) For every $p \in \bbQ^1_\kappa$ the sequence $\langle
(p,p,\kappa)\rangle$ belongs to $\bold S^{\incr}_\kappa$.

\noindent
2) $\bold S^{\incr}_\kappa$ is closed under unions of
$\vartriangleleft$--increasing chains of length $< \kappa$.

\noindent
3) If $\bar{\bold x} = \langle (p_\alpha,q_\alpha,E_\alpha):\alpha <
\beta\rangle \in \bold S^{\incr}_\kappa$ \then \, for some $p_\beta$ we
have: $\alpha < \beta \Rightarrow q_\alpha \le p_\beta$ and if $p_\beta \le
q_\beta$ and $E_\beta$ is a club of $\kappa$ disjoint to some witness of
$q_\beta$ or just of $p_\beta$ or just of $q_\gamma$ for every $\gamma <
\beta$ \then \, $\bar{\bold x} \char 94 \langle
(p_\beta,q_\beta,E_\beta)\rangle \in \bold S^{\incr}_\kappa$. 
\end{claim}

\begin{PROOF}{\ref{a17}}
  1) For $\gamma =1$ we have $\langle (p,p,\kappa)\rangle \in \bold
  S^{\incr}_{\kappa,\gamma}$ (note that clause (d) of Definition
  \ref{a15}(1) is trivially satisfied) and $\bold
  S^{\incr}_{\kappa,\gamma} \subseteq \bold S^{\incr}_\kappa$.
\medskip

\noindent
2) Obvious.
\medskip

\noindent
3) If $\beta$ is a successor ordinal this is easier, so we assume
$\beta$ is a limit ordinal.  Let $\nu_\alpha = \tr(q_\alpha)$ for
$\alpha < \beta$ hence $\langle \nu_\alpha:\alpha < \beta\rangle$ is a
$\trianglelefteq$--increasing sequence of members of $\bold T_{<
  \kappa}$ and $\ell g(\nu_\alpha) \ge \alpha$.  Hence $\nu_\beta :=
\cup\{\nu_\alpha < \beta\} \in \bold T_{\le \kappa}$ has length $\ge
\beta$.  As $\beta < \kappa$ and $\kappa$ is regular, necessarily
$\ell g(\nu_\beta) < \kappa$ so $\nu_\beta \in \bold T_{< \kappa}$.
Also recall $\alpha_1 < \alpha_2 < \beta \Rightarrow \ell
g(\nu_{\alpha_2}) \in E_{\alpha_1}$, but $E_{\alpha_1}$ is a club of
$\kappa$ hence $\alpha_1 < \beta \Rightarrow \ell g(\nu_\beta) \in
E_{\alpha_1}$.  As $\alpha_1+1 < \alpha_2 < \beta \Rightarrow
\nu_{\alpha_2} \in q_{\alpha_1}$ and $E_{\alpha_1+1}$ is disjoint to a
witness for $q_{\alpha_1}$ and by the previous sentence $\ell
g(\nu_\beta) \in E_{\alpha_1+1}$ we can deduce $\nu_\beta =
\bigcup\{\nu_{\alpha_2}:\alpha_2 \in (\alpha_1 +1,\beta)\} \in
q_{\alpha_1}$.  So clearly $\nu_\beta \in \bigcap\limits_{\alpha <
  \beta} q_\alpha$ hence $\langle q^{[\nu_\beta]}_\alpha:\alpha <
\beta\rangle$ is an increasing sequence of members of $\bbQ^1_\kappa$
with fixed trunk $\nu_\beta$ of length $\ge \beta$ as $\alpha < \beta
\Rightarrow \ell g(\nu_\beta) \ge \ell g(\nu_\alpha) = \ell
g(\tr(q_\alpha)) \ge \alpha$, see \ref{a15}(1)(f).  So by \ref{a11}(5)
we have $p_\beta := \bigcap\{q^{[\nu_\beta]}_\alpha:\alpha < \beta\}
\in \bbQ^1_\kappa$ has trunk $\nu_\beta$ and is equal to
$\big(\bigcap\{q_\alpha:\alpha < \beta\}\big)^{[\nu_\beta]}$.  Let $E_\beta
= \bigcap\{E_\alpha:\alpha < \beta\}$ and clearly $p_\beta,E_\beta$ are as
required.  
\end{PROOF}

\begin{claim}
\label{a19}
1) For every $p \in \bbQ^1_\kappa$ the sequence $\langle
(p,p,\kappa)\rangle$ belongs to $\bold S^{\pr}_\kappa$.

\noindent
2) If $\gamma < \kappa$ and $\bar{\bold x} = \langle
(p_\alpha,q_\alpha,E_\alpha):\alpha < \gamma\rangle \in \bold
S^{\pr}_{\kappa,\gamma}$ \then \, there are $(p_\gamma,E)$ with $E$ a
club of $\kappa$ and $p_\gamma = \bigcap\{p_\alpha:\alpha < \gamma\}$
such that:

if $p_\gamma \le q_\gamma,\beta < \gamma\ \Rightarrow\
q_\beta \cap \bold T_{\le \min(E_\gamma)} \subseteq q_\gamma$ and
$E_\gamma \subseteq E$ is a club of $\kappa$,

 then $\bar{\bold x} \char 94 \langle
 (p_\gamma,q_\gamma,E_\gamma)\rangle \in \bold S^{\pr}_\kappa$. 

\noindent
3) The union of a $\vartriangleleft$--increasing sequence of members
of $\bold S^{\pr}_\kappa$ of length $< \kappa$ belongs to $\bold
S^\pr_\kappa$. 

\noindent
3A) If $\langle \bar{\bold x}_\beta:\beta < \delta\rangle$ is
$\trianglelefteq$--increasing, $\bar{\bold x}_\beta = \langle
(p_\alpha,q_\alpha,E_\alpha):\alpha < \gamma_\beta\rangle \in \bold
S^{\pr}_\kappa$ and $\langle \gamma_\beta:\beta < \delta\rangle$ is
$\le$--increasing and $\gamma := \bigcup\{\gamma_\beta:\beta <
\delta\} < \kappa$ \then \, $\langle (p_\alpha,q_\alpha,E_\alpha):
\alpha < \gamma\rangle \in \bold S^{\pr}_{\kappa,\gamma}$.

\noindent
3B) If in (3A), $\gamma = \kappa$ \then \, $p_\kappa =
\bigcap\{p_\alpha:\alpha<\kappa\}$ belongs to $\bbQ^1_\kappa$ and is 
a $\le_{\bbQ^1_\kappa}$-lub of $\{p_\alpha,q_\alpha:\alpha <
\kappa\}$.
\end{claim}

\begin{PROOF}{\ref{a19}}
Straightforward.
\end{PROOF}

\begin{cc}
\label{a21}
If $\kappa = \lambda$ or just $\kappa \in S^*_{\gx}$ (see \ref{a3}), $\gamma
< \kappa$, $\bar{\bold x} = \langle (p_\alpha,q_\alpha,E_\alpha):\alpha \le 
\gamma\rangle \in \bold S^{\pr}_{\kappa,\gamma +1}$ and $\name\tau$
is a $\bbQ^1_\kappa$-name of a member of $\bold V$ \then \, we can find
$(p_{\gamma +1},q_{\gamma +1},E_{\gamma +1})$ such that
\mn
\begin{enumerate}
\item[(a)]  $\bar{\bold x} \char 94 \langle (p_{\gamma +1},q_{\gamma
+1},E_{\gamma +1})\rangle \in \bold S^{\pr}_\kappa$,
\item[(b)]  if $\eta \in q_{\gamma +1} \cap \bold T_{\min(E_{\gamma
+1})}$ then $q^{[\eta]}_{\gamma +1}$ forces a value to $\name\tau$.
\end{enumerate}
\end{cc}

\begin{PROOF}{\ref{a21}}
Let
\begin{enumerate}
\item[$(*)_1$]   $\cY = \{\tr(p):p \in \bbQ^1_\kappa$ forces a value 
to $\name\tau$ and $\tr(p)$ has length $> \min(E_\gamma)\}$.
\end{enumerate}
For $\eta \in \cY$ let $p^*_\eta$ exemplify $\eta \in \cY$, i.e.
\begin{enumerate}
\item[$(*)_2$]  $\tr(p^*_\eta) = \eta$ and $p^*_\eta$ forces a value to
$\name\tau$, necessarily $\ell g(\eta) > \min(E_\gamma)$.
\end{enumerate}
Clearly 
\begin{enumerate}
\item[$(*)_3$]  
\begin{enumerate}
\item[(a)] $\cY \subseteq \bold T_{< \kappa}$, 
\item[(b)] if $p \in \bbQ^1_\kappa$ then for some $\eta \in \cY$ we
  have $\tr(p) \trianglelefteq \eta \in p$. 
\end{enumerate}
\end{enumerate}
By Convention \ref{a7}, there is $\partial \in S_{\gx} \cap \kappa\cap
E_\gamma$ but $> \min(E_\gamma)$  such that letting $\cY_\partial =
\cY \cap \bold T_{< \partial}$ we have 
\begin{enumerate}
\item[$(*)_4$]  
\begin{enumerate}
\item[(a)] $\ell g(\tr(p_\gamma)) < \partial$, 
\item[(b)] if $p \in \bbQ^1_\partial$ then $\{\eta:\tr(p)
  \trianglelefteq \eta \in p\} \cap \cY_\partial \ne \emptyset$, 
\item[(c)] recalling \ref{a3}(D)(b), $\{(\eta,\nu):\eta \in \cY \cap
  \bold T_{< \partial}$ and $\nu \in p^*_\eta \cap \bold
  T_{< \partial}\} \in \cP_\partial$. 
\end{enumerate}
\end{enumerate}
Define:
\begin{itemize}
\item  $p_{\gamma +1} = \{\eta \in p_\gamma$: if $\ell g(\eta)
  \ge \partial$ and $\{\eta\rest\varepsilon:\varepsilon<\partial\}\cap
  \cY\neq\emptyset$ and $\zeta < \partial$ is minimal such that $\eta
  \rest \zeta \in \cY$ then $\eta\in p^*_{\eta \rest \zeta} \}$, 
\item $q_{\gamma +1} = p_{\gamma +1}$,
\item $E_{\gamma +1} \subseteq E_\gamma \backslash (\partial +1)$ is a
  club of $\kappa$ such that if  $\eta \in q_{\gamma +1} \cap \bold
  T_{< \partial}$ then $E_{\gamma +1}$ is disjoint to some witness for
  $p^*_\eta$. 
\end{itemize}
Clearly $(p_{\gamma +1},q_{\gamma +1},E_{\gamma +1})$ is as required.
\end{PROOF}

\begin{claim}
\label{a24}
If $\kappa \in S_{\gx}$ \then \, $\bbQ^1_\kappa$ is 
$\kappa$--bounding, i.e. $\Vdash_{\bbQ^1_\kappa} ``({}^\kappa
\kappa)^{\bold V}$ is $\le_{J^{\bd}_\kappa}$-cofinal in ${}^\kappa
\kappa"$.
\end{claim}

\begin{PROOF}{\ref{a24}}
By \ref{a21} and Claim \ref{a19}.
\end{PROOF}

\section{What are the desired properties of the ideal} \label{The}

Our original aim was to disprove the existence of a forcing notion for
$\lambda$ having the properties of random real forcing equivalently, finding
for an uncountable cardinal $\lambda$, a $\lambda$-complete ideal on
$\cP({}^\lambda 2)$ parallel to the ideal on null sets on ${}^{\bbN}2$.
Having constructed one raises hopes for generalizing independence results
about reals to ${}^\lambda 2$, so deriving independence results on
$\lambda$-cardinality invariants.

In this section we try systematically to go over basic properties of the
null ideal (and its relation with the meagre ideal). This results in a list
of possible test problems for our ideal. Some of these questions are
addressed in the present work, some are left for further research.
The case of $\bbQ_{\bar\theta}=\bbQ^1_\kappa$ (of \S \ref{Addinga}) is
similar and we intend to comment on it in Part II, i.e. \cite{Sh:F1199}.  

On the meagre and null ideals (for $\lambda = \aleph_0$) see Oxtoby
\cite{Oxt80}. On the measure algebra and random reals see Fremlin's treatise
\cite{Fr0x} or Bartoszy\'nski and Judah \cite{BaJu95}.

How do we measure success? The main properties of the null ideal which come
to my mind are: 
\begin{enumerate}
\item[$\boxplus$]
\begin{enumerate}  
\item[(a)] an $\aleph_1$-complete ideal (with no atoms), 
\item[(b)] the quotient Boolean Algebra satisfies the c.c.c., i.e. there is
  no uncountable family of non-null pairwise disjoint Borel sets,  
\item[(c)] the forcing is bounding: this means the  quotient Boolean Algebra
is $(\aleph_0,\infty)$--distributive, that is if for each $n$, $\langle\bold
B_{n,k}:k \in \bbN\rangle$ is a Borel partition of a non-null Borel set
$\bold B$ \then \, for some function $f:\bbN \rightarrow \bbN$, the set
$\bigcap\limits_{n} \bigcup\limits_{k<f(n)} \bold B_{n,k}$ is not null.   
\end{enumerate}
\end{enumerate}
A priori, for the set theoretic purposes, generalizing (a),(b),(c) was  the
aim.  But for the ideal itself, a prominent property of the
null ideal, and a very nice one, is 
\begin{enumerate}
\item[(d)] the Fubini theorem:  for a Borel set $A \subseteq [0,1] \times
  [0,1]$ the following are equivalent:  
\begin{enumerate}
\item[(i)]  for all but null many $x$, for all but null many $y$ we have
  $(x,y) \in A$, 
\item[(ii)]  for all but null many $y$, for all but null many $x$ we have
  $(x,y) \in A$. 
\end{enumerate}
\end{enumerate}
But alas, this fails, see Claim \ref{u24}.

Maybe it is helpful to stress, that 
\begin{enumerate}
\item[$\boxtimes$] we are looking for $\lambda^+$-complete, $\lambda^+$-c.c.,
  ideal with no atoms.
\end{enumerate}
Below we make a list of statements generalizing the null ideal case,
including the natural analogs of the properties listed above, delaying a try
on some further properties.

A reader who goes first to this section can note just that
\begin{enumerate}
\item[$\oplus$] 
\begin{enumerate}
\item[(a)] the forcing notion $\bbQ_\lambda$ is a set of subtrees of
  ${}^{\lambda >}2$ representing $\lambda$--closed subsets $\lim_\lambda(p)$
  of ${}^\lambda 2$, where $\lim_\lambda(p)=\{\eta\in {}^\lambda 2:
  (\forall\zeta<\lambda)(\eta\rest\zeta\in p)\}$, parallel to the closed
  subsets of $[0,1]_{\bbR}$ with positive Lebesgue measure,  partially
  ordered by inverse inclusion,   
\item[(b)] ${}^\lambda 2$ is the set of functions from $\lambda$ to
  $2=\{0,1\}$. 
\end{enumerate}
\end{enumerate}

\begin{definition}
\label{p1}
Let $\lambda$ be an inaccessible cardinal and let
$\bbQ_\lambda=\bbQ^2_\lambda$ be the forcing notion introduced in \S\
\ref{Adding}.  
\begin{enumerate}
\item For $\eta \in {}^\lambda 2$ and $\cI \subseteq \bbQ_\lambda$, saying
  $\eta$ fulfills $\cI$ means $(\exists q \in \cI)(\eta \in \lim_\lambda(q))$.  
\item  For $\cI \subseteq \bbQ_\lambda$ let $\set(\cI) = \{\eta \in
  {}^\lambda 2:\eta$ fulfills $\cI\}$ and  for a set $\Lambda$ of subsets of
  $\bbQ_\lambda$ let $\set(\Lambda) = \bigcap\{\set(\cI):\cI \in \Lambda\}$.
\item We define $\id(\bbQ_\lambda) = \{A \subseteq {}^\lambda 2$: there are
  $i(*) \le \lambda$ and dense open subsets $\cI_i$ of $\bbQ_\lambda$ for $i
  < i(*)$ such that $\eta \in A \wedge i <i(*) \Rightarrow \eta$ does not
  fulfill $\cI_i\}$. 
\item A $\lambda$--real is $\eta \in {}^\lambda 2$.
\end{enumerate}
\end{definition}
\medskip

\begin{convention}
\label{pp8}
$\lambda,\partial,\kappa$ vary on inaccessibles.
\end{convention}
\medskip

We have consulted several people on additional properties to be
examined. For instance T. Bartoszy\'nski suggested (P),(S),(U) of the first
list below.
\medskip

\subsection {Desirable Properties: First List} \label{Desirable} \
In this subsection we list various desirable properties and questions and
sometimes give a relevant reference (in this paper) \underline{but} we do
not prove anything (whereas \S3 on contain proofs). 
\begin{enumerate}
\item[(A)]  
\begin{enumerate}
\item[$(\alpha)$] The ideal $\id(\bbQ_\lambda)$ is $\lambda^+$--complete,
  i.e. closed under union of $\le \lambda$ sets.
\item[$(\beta)$] The forcing notion $\bbQ_\lambda$ is $\lambda$--complete
  (or at least $\lambda$-strategically complete, depending on the choice of
  the order).  
\item[$(\gamma)$] The Boolean Algebra of $\lambda$--Borel subsets of
  ${}^\lambda 2$ modulo the ideal $\id(\bbQ_\lambda)$ satisfies the
  $\lambda^+$-c.c.,  see \ref{p19}(2). Note that modulo $\id(\bbQ_\lambda)$,
  $\bbQ_\lambda$ is dense in this Boolean Algebra, this is (E) below.
\item[$(\delta)$] The forcing notion $\bbQ_\lambda$ is $\lambda$--bounding, 
  see \ref{z4}(2), \S1, when $\lambda$ is a weakly compact cardinal.  
\end{enumerate}
\item[(B)]   The definability of $\bbQ_\lambda$, i.e., $\bbQ_\lambda$ is
nicely definable (with no parameters), see  the definition by induction in
\S1; if $\lambda$ is weakly compact then $\bbQ_\lambda$ is $\lambda$--Borel,
the ideal is similarly definable, see \ref{k3}; for other inaccessible
cardinals $\lambda$ the ``nowhere stationary'' is $\Sigma^1_1(\lambda)$ but
by a somewhat cumbersome definition giving an equivalent forcing it is
$\lambda$-Borel, see the proof of \ref{a27}.   
\item[(C)]   Generalizing ``adding (forcing) a Cohen real makes the set of
  old reals null'', see \ref{p28}. 
\item[(D)]   Generalizing ``adding (i.e. forcing) a random real makes the
  old real meagre'', see \ref{p15}. 
\item[(E)]  Modulo the ideal $\id(\bbQ_\lambda)$, every $\lambda$--Borel set
  is  equal to a union of at most $\lambda$ sets of the form
  $\lim_\lambda(p)$, $p \in \bbQ_\lambda$, see \ref{p19}. 
\item[(F)] Can we define integral?  We do not know; may we replace
  $[0,1]_{\bbR}$ as a set of values by some complete linear order, e.g. by
  ``nice'' ordered fields?  Are symmetrically complete real closed fields
  relevant (see \cite{Sh:757})?  If we waive linearity does it help?  
\item[(G)]   Modulo the ideal, every $\lambda$--Borel function 
can be approximated by ``steps function of level $\alpha$'' for many (so
unboundedly) many $\alpha < \lambda$; where ``step function'' is being
interpreted as: $f(\eta) \rest \alpha$ is determined by $\eta \rest \alpha$
for $\eta \in {}^\lambda 2$, see \ref{p22}. 
\item[(H)]   The Lebesgue density theorem, see \ref{p23}, (it means: if the
  $\lambda$--Borel set $\bold B \subseteq {}^\lambda 2$ is
  $\id(\bbQ_\kappa)$-positive, then for some $\bold B_1 \in
  \id(\bbQ_\lambda)$ for every $\eta \in \bold B \backslash \bold B_1$ for
  some $\alpha < \lambda$ we have $({}^\lambda 2)^{[\eta \rest \alpha]}
  \backslash \bold B \in \id(\bbQ_\lambda)$).   
\item[(I)]  The Fubini theorem, symmetry, unfortunately fails, see
  \ref{u24}. However we intend to present some weak versions of symmetry in
  a continuation. 
\item[(J)]  The translation invariance, see \ref{p80}(1).
\item[(K)] The permutation invariance (i.e. for permutations of $\lambda$):
  this works only for a variation on our forcing. 
\item[(L)]  Generalizing ``if $A$ is a Borel subset of $[0,1]_{\bbR} \times
  [0,1]_{\bbR}$ of positive measure then $A$  contains a perfect rectangle
  (even half square)".  But what is perfect? Not a copy of ${}^\lambda 2$
  but $\lambda$--closed set, e.g. the $\lambda$-limit of a $\lambda$--Kurepa
  tree, actually  one with ``little pruning in limit levels"; specifically
  it is $\lim_\lambda(p)$ for some $p \in \bbQ_\lambda$, so
  $\lambda$--closed.
\item[(M)]   Generalize the random algebra on ${}^\chi 2$ for $\chi$
  possibly $> \lambda$. This will be addressed in a continuation, see
  \cite[\S 1]{Sh:E82}, \cite{Sh:1100}.  
\item[(N)]   Generalize ``modulo the null ideal every Borel set is equal to
  a union of $\le \lambda$ sets, each $\lambda$--closed" see (E) above and
  see \ref{p19}. 
\item[(O)]   Generalize ``the set of reals is a union of a null set and a
  meagre set", see \ref{p24}.
\item[(P)]  Generalize Erd\"os--Sierpi\'nski theorem: if $2^\lambda =
  \lambda^+$ or suitable cardinal invariants are equal to $\lambda^+$ \then
  \, there is a permutation of ${}^\lambda 2$ interchanging the null and
  meagre ideal. 
\end{enumerate}

In fact, this is not hard now:

\begin{enumerate}
\item[$(*)_1$]  Assume that for $\ell=1,2$:
\begin{enumerate}
\item[(a)] $J_\ell$ is an ideal of subsets of $I$, 
\item[(b)] $J_\ell$ is $|I|$--complete and generated by a family of $\le |I|$
  sets, 
\item[(c)] if $A_1\in J_\ell$ then for some $A_2 \in  J_\ell$ we have $|A_2
  \backslash A_1| = |I|$, and
\item[(d)] there is $A \in J_1$ such that $I \backslash A \in J_2$. 
\end{enumerate}
Then  there is a permutation of $I$ interchanging $J_1$ with $J_2$.
\item[$(*)_2$]  If $2^\lambda = \lambda^+$ and $I = {}^\lambda 2$ then
  the $\lambda$--meagre ideal and $\id(\bbQ_\lambda)$ satisfy (a)--(d) of
  $(*)_1$. 
\end{enumerate}
[Why?  Clause (d) here holds by \ref{p24}.] 

\begin{enumerate}
\item[(Q)]   Generalize the Borel conjecture:  though not connected to 
  random.  Now consider: 
\begin{enumerate}
\item[$(\alpha)$]  the equivalence of the ``for every $\langle
  \varepsilon_n:n\rangle$ the set is covered by $\bigcup\limits_{n} I_n$,
  $I_n$ is an interval of length $\le \varepsilon_n$''  and ``the set can be
  translated away from any meagre set'',
\item[$(\beta)$] the $\varepsilon_n$'s version has an obvious generalization,
\item[$(\gamma)$] try shooting through a normal ultrafilter
\end{enumerate}
\item[(R)]   The dual Borel conjecture might be adressed in Part II. Now the
  question is: 
\begin{enumerate}
\item[$(*)$] We are given an old set $X$ of $\lambda$-reals of cardinality
  $\lambda^+$, say $X = \{\nu_\alpha:\alpha<\lambda^+\}$.  View
  Cohen$_\lambda$ as adding a $\lambda$--null set:  e.g., for $\bar p =
  \langle p_\eta:\eta \in {}^{\lambda>}2\rangle$, $p_\eta \in \bbQ_\lambda$,
  $\tr(p_\eta)=\eta$, and clearly   $p_\eta$ is a nowhere-dense cone, but we
  shall need more. 
\end{enumerate}
\item[(S)]  (Selectors)\quad Every $\Sigma^1_1$--relation have a reasonably
  definable, e.g. $\lambda$--Borel, choice function on a positive closed set
  even in any positive Borel set.
\item[(T)]   The Hausdorff paradox and even Banach-Tarski paradox hold for  
  $\mbR^3$. Do they hold for ${}^\lambda 2\times {}^\lambda 2\times
  {}^\lambda 2$? 
\item[(U)] We know that ``for every meagre set $A$ there is a meagre set $B$
  such that: every $\le \lambda$ translates of $A$ can be covered by one
  translates of $B$'', but fail for null, even for ${}^\omega \bbZ$.
  Generalize to $\lambda$.
\end{enumerate}
On raising further problems see \cite{Sh:F1199}, concerning characters,
differentiability, monotonicity (of functions) and going back to the case
$\lambda = \aleph_0$.
\medskip

We have not looked at clauses (L),(Q), (S)--(U).
\medskip

\subsection {Desirable Properties: Second List}
 \label{TwoDes}
Next we consider generalizing results more set theoretic in nature, related
to forcing (maybe (B)(c),(d) from \S(2A) should be here; from the problems
listed below, (A) is treated here, on the others see part II, if at all)

\begin{enumerate}
\item[(A)]  Cicho\'n's Diagram.
\end{enumerate}
This diagram sums up the provable inequalities between the basic cardinal
invariants of the null ideal, the meagre ideal, $\gd$ (the dominating
number) and $\gb$ (the unbounding number).  The basic cardinal invariants of
an ideal are the covering number, the additivity number, the cofinality and
the non(= uniformity) of the ideal, see \ref{z26}.

The diagram gives the provable inequalities among any two invariants (and
two equalities each on three invariants).  Moreover, under
$2^{\aleph_0} \le \aleph_2$ there are no more connections.  \underline{Here}
we generalize the ZFC part (for $\lambda$ inaccessible limit of
inaccessibles), but the situation is different, e.g., there are more
inequalities connecting 3 of the cardinal invariants, see \ref{u14}.

We will deal with the complementary consistency results (about
inequalities of any pair) in continuations, \cite {Sh:F1580} and others. 

\begin{enumerate}
\item[(B)]   Generalizing the amoeba forcing
\end{enumerate}
The amoeba forcing is the one adding a measure zero set including all the
old ones; the conditions are closed subsets of $[0,1]_{\bbR}$ of measure $>
\frac 12$.

This is natural as the amoeba forcing has been important in set theory of
the reals and is closely related to measure, see Section 7.

\begin{enumerate}
\item[(C)]  The consistency of ``every $A \in \cP(\bbR)^{\bbL[\bbR]}$ is
  Lebesgue  measurable'' (from $\chi > \lambda$ inaccessible).
\end{enumerate}
Solovay \cite{So70} classical work proved for $\lambda = \aleph_0$ that if
we Levy collapse the first inaccessible cardinal to being $\aleph_1$, this
holds.

The problem is: we have names $\name\eta$ of $\lambda$--reals such that
$\Levy(\lambda,{<}\chi)/\name\eta$ is not $\Levy(\lambda,{<}\chi)$ when
$\lambda$ is uncountable. Another formulation of the problem: there are
$\Levy(\lambda,{<}\chi)$--names $\name\eta_1,\name\eta_2$ of
$\lambda$--Cohen reals and no automorphisms of the completion of
$\Levy(\lambda,{<}\chi)$ mapping one to the other.

This certainly occurs for $\lambda$--Cohen reals and probably for any other;
that is we may add a $\lambda$--Cohen $\name\eta \in {}^\lambda 2$ and
compose it with a forcing shooting a club through $\name\eta^{-1}\{\ell\}$.  

A possible avenue is to consider only ``nice
$\Levy(\lambda,{<}\chi)$--names'', i.e. such that the quotient is
$\Levy(\lambda,< \chi)$.  In this case there is a ``positive'' set of
$\lambda$--reals such that for subsets of it our aim is achieved.  We can
even define this set of reals.  The question is whether we consider this is
a ``reasonable'' or a ``forced, artificial'' solution?

Alternatively we may replace $\lambda$--Cohen by another forcing (or ideal)
and/or change the collapse; in particular should check the failure for
$\bbQ_\lambda$.  We also may change the notion of a $\lambda$--real,
e.g. replace it by $A$/(the non-stationary ideal) or use a filter generated by
$\le \lambda$ subsets of $\lambda$!  All this is delayed for later parts.  We
should also check what occurs to sweetness in our present case (see
\cite{RoSh:672,RoSh:856}). 

We may consider $\{\eta \in {}^\lambda 2:\eta$ is $(\bbQ,\name\eta)$-generic
over $\bold V_0$ such that every subset of $\lambda$ from $\bold V_0[\eta]$
which is stationary in it, is also stationary in $\bold V\}$, or more.  A
related question is the complexity of maximal antichains, see \ref{k4d}, maybe
use measurable cardinals.  

What about $\cP(\lambda)$ for $\lambda$ singular strong limit of cofinality
$\aleph_0$? 

\begin{enumerate}
\item[(D)] Can we characterize $\Cohen_\lambda$ and $\bbQ_\lambda$ among
  (nicely definable) $\lambda$--Borel ideals?  Recall Solecki-Kechris
  characterization of Cohen and random (for the ideals). We have not looked
  at it; there are limitations even for $\lambda = \aleph_0$, see e.g.,
  \cite{RoSh:628}.  
\item[(E)]  In \cite{Sh:480} we showed that: for any Suslin c.c.c. forcing,
  if it adds an undominated real, it adds a Cohen real.  
\end{enumerate}
Subsequently some works show relatives (for other properties), on this see
\cite{Sh:711}, \cite{Sh:723}. Related to this, by \cite{Sh:630}, the only
``nice'' c.c.c. forcing commuting with Cohen is Cohen itself.  Do we have a
parallel?

For a broader generalization of the case of $\aleph_0$ we may consider
forcing, ultrafilters and forcing notions definable from ultrafilters. 
\begin{enumerate}
\item[(F)] We know much on ultrafilters on $\bbN$.  Also we have
  considerable knowledge about $\lambda$--complete ultrafilters on $\lambda$
  or higher cardinals when $\lambda$ is a measurable cardinal.  After the
  seventies there were set theoretic advances on non-regular ultrafilters,
  but not much set theoretic work was done on regular ultrafilter. However,
  in recent years there were studies of reasonable ultrafilters in
  \cite{Sh:830}, Ros{\l}anowski and Shelah \cite{RoSh:889,RoSh:890} and
  recently on ultrafilters related to saturation of ultra-powers and Keisler
  order, see Malliaris and Shelah \cite{MiSh:996,MiSh:998} on cuts and $\gp
  = \gt$.
\end{enumerate}

On characters of ultrafilters on $\bbN$ see Brendle and Shelah
\cite{BnSh:642} and later \cite{Sh:846}, \cite{Sh:915}; for an ultrafilter
$D$ on $\lambda$ recall that $\chi(D)$ is the character = minimal
cardinality of a subset generating it, $\pi \chi(D)$ pseudo-character =
minimal cardinality of $\cA \subseteq [\lambda]^\lambda$ such that $(\forall
B \in D)(\exists A \in \cA)[A \subseteq B]$, note that $A \in \cA$ is not
necessarily in $D$!  As in \cite{RoSh:889, RoSh:890} dealing with the so
called reasonable ultrafilters we may consider the Borel version (i.e. the
minimal number of Borel subsets of $D$ which generate it) and $\lambda$-real
version.  Then as in ``reasonable ultrafilter'', can we show CON(for every
uniform ultrafilter $D$ on $\lambda,\pi\chi_{\lambda-\real}(D) = \lambda^+ <
2^\lambda$)?

What about the ultrafilter forcing?  Can reasonable ultrafilters on $\lambda$
be generated by $< 2^\lambda$ sets?  We can force a creature condition 
diagonalizing a uniform ultrafilter on $\lambda$.

\begin{enumerate}
\item[(G)]  Related is Galvin-Prikry theorem which says that for any
Borel (or even $\Sigma^1_1$) subset $\bold B$ of $\cP(\bbN)$ for some set $A
  \in [\bbN]^{\aleph_0}$, the set $[A]^{\aleph_0}$ is included in
  or disjoint from $\bold B$.  Concerning a relative using a group from
  \cite{Sh:273}, generalizations to $\lambda$ are considered by the
  author in some later works: \cite{Sh:664}, \cite{ShVs:718},
  \cite{ShVs:719}, \cite{Sh:724}, see also \cite{GrSh:302},
  \cite{Sh:771}, less related \cite{MShS:121}, \cite{MShS:144}, 
\cite{HkSh:662}
\item[(H)]  The consistency of Moore conjecture; so we should consider
  a topological space $X$ which is $\lambda$--first countable (analog of
  first countable). Of course we can prove it using Dow lemma which holds
  for adding many $\lambda$-Cohens, so not clear how interesting.
\item[(I)]  Preserving ``$\eta$ is $\bbQ_\lambda$--generic over $N$''
  parallel to \cite[Ch.XVIII,\S3]{Sh:c}, \cite[Ch.VI,\S3]{Sh:c}.
\item[(J)]  
\begin{enumerate}
\item[(a)] Try to connect $\cf(\bbQ_\lambda)$ and Cicho\'n's diagram and
  number of  reasonable generators of an ultrafilter, see \cite{Sh:830}.
\item[(b)] Note that for the number of generators of an ultrafilter we have
  the following bounds. 
\end{enumerate}
\end{enumerate}

\begin{claim}
\label{p34d}
\begin{enumerate}
\item Letting $\name\eta_\lambda$ be the $\bbQ_\lambda$--name of the
  generic, for $\alpha<\lambda$ we have that  $\Vdash_{\bbQ_\lambda}$
  ``there is $\bold G' \subseteq  \bbQ_\lambda$ such that: $\bold G'$ is a
  generic subset of $\bbQ_\lambda$ over $\bold V,\bold V[\bold G'] = \bold
  V[\name{\bold G}]$ and  $\name\eta_\lambda[\bold G'] =
  \name\eta_{\lambda,\alpha}[\name{\bold G}]"$ where $\eta_{\lambda,\alpha}
  \in {}^\lambda 2$ is defined by: 
\[\name\eta_{\lambda,\alpha}(i) = \begin{cases}
\name\eta_\lambda(i) \quad &\text{ if } \quad i < \alpha \\
1 - \name\eta_\lambda(i) \quad &\text{ if } \quad i \in 
[\alpha,\lambda).
\end{cases}\]
\item Similarly when for some $A\in\cP(\lambda)^{\bold V}$
\[\name\eta_{\lambda,\alpha}(i) = \begin{cases}
\name\eta_\lambda(i) \quad &\text{ if } \quad i \in A \\
1 - \name\eta_\lambda(i) \quad &\text{ if } \quad i \in 
\lambda\setminus A.
\end{cases}\]
\item $\Vdash_{\bbQ_\lambda} ``\name\eta_\lambda \rest A \ne_{J^{\bd}_A} 
i_A$ for $i=0,1$ for any $A \in ([\lambda]^\lambda)^{\bold V}"$.
\item $\chi(\lambda) := \min\{{\rm gen}(D):D$ a uniform ultrafilter on  
$\lambda\}$ is $\ge \cov(\bbQ_\lambda),\cov(\Cohen_\lambda)$. 
\end{enumerate}
\end{claim}
But we can still hope to find a relative of $\bbQ_\lambda$ such that adding
$\lambda^{++}$ such $\lambda$-reals (e.g. as in \cite{Sh:F1580}) we get a
universe $\bold V_1$ with $2^\lambda = \lambda^{++} +$ there is a uniform
ultrafilter $D$ on $\lambda$ with $\chi(D) = \lambda^+$.

\begin{enumerate}
\item[(K)] Here we start with $\lambda$-Cohen forcing (for $\chi$
  inaccesible not limit of  inaccessibles).  We can start with
  $\bbQ_{\bar\theta \rest \lambda}$ or with other definable
  $\lambda^+$-c.c. forcing; see part II. 
\end{enumerate}

\section {On $\bbQ_\kappa$, $\kappa$--Borel sets and $\id(\bbQ_\kappa)$}    
\label{idBorel}
In this and the following sections we analyze the ideal
$\id(\bbQ_\kappa)$. A general frame including \ref{p1} is the following. 

\begin{definition}
\label{p2}
\begin{enumerate}
\item Let $\id(\Cohen_\kappa)$ be the family of all $\kappa$--meagre subsets
  of ${}^\kappa 2$, i.e., it is the collection of all $A \subseteq {}^\kappa
  2$ such that $A \subseteq \bigcup\{\lim_\kappa(\cT_i)$ for $i<\kappa\}$,
  where each $\cT_i$ is a nowhere dense subtree of ${}^{\kappa >}2$, i.e.,
  $({}^{\kappa >}2,\triangleleft)$.  
\item We say $\bold i=(\kappa,\bbQ,\name\eta)=(\kappa_{\bold i}, \bbQ_{\bold
    i}, \name\eta_{\bold i})$ is an ideal case \when \, :
\begin{enumerate}
\item[(a)]  $\kappa$ is a regular cardinal, 
\item[(b)]  $\bbQ$ is a forcing notion not adding bounded subsets of
  $\kappa$,  
\item[(c)]  $\name\eta$ is a $\bbQ$-name of a member of ${}^\kappa 2$, 
\item[(d)]  
  \begin{enumerate}
  \item[$(\alpha)$] each $p \in \bbQ$ is a subtree of $({}^{\kappa >}2,
    \trianglelefteq)$ and let $\bold B_p = \bold B_{\bold i,p}
    =\lim_\kappa(p)$, and $p\Vdash$``$\name{\eta}\in  \bold B_{\bold
      i,p}$'', \underline{or} at least 
\item[$(\beta)$] we have a mapping $p \mapsto \bold B_p =
  \bold B_{\bold i,p}$ such that 
  \begin{itemize}
\item $\bold B_{\bold i,p}$ is  a $\kappa$--Borel subset of ${}^\kappa 2$, 
\item $p\leq q \ \Rightarrow \ {\bold B}_{{\bold i},p}\supseteq {\bold
    B}_{{\bold i},q}$,  and 
\item $p \Vdash``\name\eta \in \bold B_{\bold i,p}$''; 
  \end{itemize}
so really the  function $p \mapsto\bold B_p$ is part of $\bold i$. 
  \end{enumerate}
\end{enumerate}
Below let $\bold i = (\kappa,\bbQ,\name\eta)$ be an ideal case. 
\item We let $\id^1_{\bold i} =\id_1(\bold i)$ be 
\[\big\{A \subseteq {}^\kappa 2:\mbox{ for some $\kappa$--Borel set $\bold
  B$ we have }A \subseteq \bold B\mbox{ and }\Vdash_{\bbQ} ``\name\eta
\notin \bold B\mbox{''}\};\]
we may omit the 1. 
\item For a subset $\cI$ of $\bbQ_{\bold i}$, we say that $\eta \in
  {}^\kappa 2$ fulfills $\cI$ \when \, $(\exists p \in \cI)(\eta \in \bold
  B_p)$. 
\item We define $\id^2_{\bold i} = \id_2(\bold i)$ to be the collection of
  all sets $A\subseteq {}^\kappa 2$ such that there are pre-dense subsets
  $\cI_i$ of $\bbQ_{\bold i}$ for $i<\kappa$ such that 
\[A \subseteq \big\{\eta \in {}^\kappa 2: \mbox{ for some $i<\kappa$, $\eta$ 
  does not fulfill }\cI_i\big\}.\] 
\end{enumerate}
\end{definition}

\begin{claim}
\label{p4d}
Let $\bold i$ be an ideal case.
\begin{enumerate}
\item Both $\id_1(\bold i)$ and $\id_2(\bold i)$ are $\kappa^+$--complete
  ideals on  ${}^\kappa 2$. Also $ {}^{ \kappa } 2\notin\id_1({\bold i})$
  and if $ {\bold i} $ is $ \kappa$--complete then\footnote{ Recall Prikry
    forcing} ${}^{ \kappa } 2\notin \id_1({\bold i})$.
\item In Definition \ref{p2}(5) we can replace ``pre-dense'' by ``dense 
  open'' or by ``maximal antichain''.  
\item  If $\bbQ_{\bold i}$ satisfies the $\kappa^+$-c.c. \then \,
$\id_2(\bold i) \subseteq \id_1(\bold i)$. 
\item  A sufficient condition for $\id_1(\bold i) \subseteq \id_2(\bold  
   i)$ is: 
\begin{enumerate}
\item[$(*)$]  
\begin{enumerate} 
\item[(a)] if $p,q\in\bbQ_{\bold i}$ are incompatible then ${\bold
    B}_{{\bold i},p}\cap {\bold B}_{{\bold i},q}=\emptyset$,  and
\item[(b)] if $\bold B$ is a $\kappa$--Borel set \then \, 
\[\big\{p \in \bbQ_{\bold i}:\ p\Vdash_{\bbQ_{\bold i}} ``\name\eta \in
\bold B\mbox{''  or }\bold B_p\cap \bold B \in \id_2(\bold i)\big\}\]
is a dense open subset of $\bbQ_{\bold i}$.
\end{enumerate}
\end{enumerate}
\item Let $\kappa$ be strongly inaccessible and $\bbQ_\kappa$ and
  $\name{\eta}$ be as defined in \ref{n5} and \ref{p10}(4),
  respectively. Then the triple $\bold i = \bold i^{\null}_\kappa =
  (\kappa,\bbQ_\kappa,\name\eta)$ is an ideal case and  $\id_1(\bold i) =
  \id_2(\bold i)$. 
\item The triple $\bold i = \bold i^{\Cohen}_\kappa =
   (\kappa,\Cohen_\kappa,\name\eta)$ is an ideal case and 
we have $\id_1(\bold i) = \id_2(\bold i)$ and it is closed under
translations (cf \ref{p80}). 
\end{enumerate}
\end{claim}

\begin{remark}
   If in Definition \ref{p2}(2)(d), $\bold B_p$ is just a Borel set, then
   \ref{p4d} still holds.
\end{remark}

\begin{PROOF}{\ref{p4d}}
(1), (2)\quad Obvious by the definitions.
\medskip

\noindent
(3)\quad Assume $A \subseteq {}^\kappa 2$ belongs to $\id_2(\bold i)$.  Then
by (2) we may find maximal antichains $\cI_i \subseteq \bbQ_{\bold i}$ (for
$i<\kappa$) such that 
\[\eta \in A\quad \Rightarrow\quad \mbox{ for some }i<\kappa, \ \eta\mbox{ 
does not fulfill }\cI_i.\]  
Since we are assuming that $\bbQ_{\bold i}$ satisfies the $\kappa^+$-c.c.,
$\cI_i$ has cardinality $\le \kappa$ for every $i<\kappa$. Let $\langle
p_{i,\varepsilon}:\varepsilon < \varepsilon_i\rangle$ list $\cI_i$,
$\varepsilon_i\leq \kappa$. Then 
\[A \subseteq \bold B := \bigcup\limits_{i<\kappa}  ({}^\kappa 2 \backslash
\bigcup \{{\bold B}_{\bold i,p_{i,\varepsilon}}:\varepsilon <
\varepsilon_i\}).\]
Clearly $\bold B$ is a $\kappa$--Borel set. Also, since each $\cI_i$ is a
maximal antichain, for all $i<\kappa$ we have  
\[\Vdash_{\bbQ_{\bold i}}\mbox{`` }\cI_i \cap \name{\bold G}_{\bbQ_{\bold i}}
\ne \emptyset\mbox{ and hence }\name\eta \in \bold
  B_{\bold i,p_{i,\varepsilon}}\mbox{ for some }\varepsilon < \varepsilon_i
  \mbox{'',}\]  
and hence $\Vdash_{\bbQ_{\bold i}}``\name\eta \notin \bold
B$''. Consequently $\bold B \in \id_1(\bold i)$ but $A\subseteq \bold B$
hence $A \in \id_1(\bold i)$,  so we are done.
\medskip

\noindent
(4)\quad Assume $\bold B$ is a $\kappa$--Borel set and it belongs to
   $\id_1(\bold i)$. We shall prove $\bold B \in \id_2(\bold i)$, clearly
   this suffices. 

Let $\cI = \{p:p$ forces $\name\eta \in \bold B$ or forces $\bold B_p \cap
\bold B \in \id_2(\bold i)\}$, so by the assumption $(*)$(b) the set $\cI$
is an open dense subset of $\bbQ_{\bold i}$.  Let $\cI' \subseteq \cI$ be a 
maximal antichain and let $\cI'' = \{p \in \cI':p \nVdash_{\bbQ_{\bold i}} 
``\name\eta \in \bold B"\}$. Since we assumed $\bold B \in \id_1(\bold i)$,
necessarily $\cI'' = \cI'$. So for each $p \in \cI''$, $\bold B_p \cap \bold
B \in \id_2(\bold i)$ and there is a sequence $\langle \cI_{p,i}:i
<\kappa\rangle$ witnessing it.  \Wilog \, if $i<\kappa$, $p \in \cI''$ then
$\cI_{p,i}$ is a maximal antichain of $\bbQ_{\bold i}$ and for every $q \in \cI_{p,i}$
we have $(p \le q) \vee (p,q\mbox{ are incompatible})$. For $i<\kappa$ let 
\[\cI^i=\big\{q\in\bbQ_{\bold i}: \mbox{ for some $p\in\cI''$ we have }
(p\leq q)\wedge q\in \cI_{p,i}\big\}.\] 
Clearly, each $\cI^i$ is a maximal antichain. Easily $\{\cI^i:i<\kappa\}$
witnesses $\bold B$ is included in some member of $\id_2(\bold i)$, so we
are done. 
\medskip

\noindent
(5)\quad For being an ideal case, in Definition \ref{p2}(2), clauses
(a),(b),(c) are obvious (remember Claim \ref{n11} and Observation
\ref{p10}(4)) and clause (d) is easy, too.  It suffices to prove that
$\id_2(\bold i) \subseteq \id_1(\bold i)$ and $\id_1(\bold i) \subseteq
\id_2(\bold i)$. 

Concerning ``$\id_2(\bold i) \subseteq \id_1(\bold i)$'' note that
$\bbQ_\kappa$ satisfies the $\kappa^+$-c.c., so by \ref{p4d}(3) we deduce
the inclusion.

Let us argue that $\id_1(\bold i) \subseteq \id_2(\bold i)$. Suppose that
$\bB$ is a $\kappa$--Borel subset of ${}^\kappa 2$ and
$\Vdash_{\bbQ_\kappa}$``$\name{\eta}\notin \bB$''. We may find $\cT$ and
$\bar{\bB}$ such that 
\begin{enumerate}
\item[$(\circledast)$] 
\begin{enumerate}
\item[(a)]  $\cT$ is a subtree of ${}^{\omega>}\kappa$ with no infinite
    branch, 
\item[(b)]  for every $\rho\in \cT$, either $\suc_\cT(\rho)=\emptyset$, or
  $\suc_{\cT}(\rho)=\{\rho\char 94\langle 0\rangle\}$ or $\suc_\cT(\rho)$ is infinite,
\item[(c)] $\bar{\bB}=\langle\bB_\rho:\rho\in \cT\rangle$ is a system of 
  $\kappa$--Borel subsets of ${}^\kappa2$, 
\item[(d)] $\bB_{\langle\rangle}=\bB$,
\item[(e)] if $\rho\in \cT$ and $\suc_\cT(\rho)=\emptyset$, then for some
  $i_\rho<\kappa$ and $c_\rho<2$ we have $\bB_\rho=\{\nu\in {}^\kappa 2:
  \nu(i_\rho)=c_\rho\}$,
\item[(f)] if $\rho\in\cT$ and $|\suc_\cT(\rho)|=1$, then
  $\bB_\rho={}^\kappa 2\setminus \bB_{\rho\char 94\langle0\rangle}$,
\item[(g)] if $\rho\in \cT$ and $\suc_\cT(\rho)$ is infinite, then
  $\bB_\rho=\bigcap\{\bB_\varrho:\varrho\in\suc_\cT(\rho)\}$.
\end{enumerate}
\end{enumerate}
Then by induction on $\lh(\rho)$ for each $\rho\in \cT$ we choose $\cI_\rho$
and $\bar{t}^\rho$ so that for each $\rho\in \cT$:
\begin{enumerate}
\item[$(\otimes)$]
\begin{enumerate}
\item[(a)] $\cI_\rho$ is a maximal antichain of $\bbQ_\kappa$ and
  $\bar{t}^\rho=\langle t^\rho_p:p\in\cI_\rho\rangle$, $t^\rho_p<2$ for each
  $p\in\cI_\rho$, 
\item[(b)] if $t^\rho_p=1$, then $p\Vdash$``$\name{\eta}\in \bB_\rho$'' and
  if $t^\rho_p=0$, then $p\Vdash$``$\name{\eta}\notin \bB_\rho$'', 
\item[(c)] if $\suc_\cT(\rho)=\emptyset$ and $p\in\cI_\rho$, then
  $\lh(\tr(p))> i_\rho$ (see $(\circledast)$(e) above), 
\item[(d)] if $|\suc_\cT(\rho)|=1$, then $\cI_\rho=\cI_{\rho\char 94 \langle
    0\rangle}$ and $t^\rho_p=1-t^{\rho\char 94\langle 0\rangle}_p$ for
  $p\in\cI_\rho$, 
\item[(e)] if $\suc_\cT(\rho)$ is infinite, $p\in\cI_\rho$ and $t^\rho_p=0$,
  then $p\Vdash$``$\name{\eta}\notin \bB_\varrho$'' for some $\varrho\in 
  \suc_\cT(\rho)$,
\item[(f)] if $\rho\vartriangleleft \varrho\in\cT$ and $q\in \cI_\varrho$,
  then there is unique $p\in \cI_\rho$ such that $p\leq q$. 
\end{enumerate}
\end{enumerate}
Now let $Y=\bigcap\limits_{\rho\in \cT} \set(\cI_\rho)$ (see \ref{p1}(2))
and note that ${}^\kappa2\setminus Y\in\id_2({\bold i})$. By induction on
${\rm dp}(\rho,\cT)$ we are going to argue that for $\rho\in\cT$:
\begin{enumerate}
\item[$(\heartsuit)_\rho$] for each $\nu\in Y$ we have
\[\nu\in \bB_\rho\quad \iff\quad (\exists p\in\cI_\rho)(\nu\in {\lim}_\kappa
(p)\ \wedge\ t^\rho_p=1).\]  
\end{enumerate}

\noindent {\sc Case 1:}\quad $\suc_\cT(\rho)=0$.\\
Since $\nu\in Y$ there is unique $p\in\cI_\rho$ such that
$\nu\in\lim_\kappa(p)$, recalling that for $p,q\in\bbQ_\kappa$
\[(p,q\mbox{ are incompatible })\ \Rightarrow\ (\tr(p) \notin q\vee 
\tr(q)\notin p)\ \Rightarrow\ {\lim}_\kappa(p)\cap {\lim}_\kappa(q)
=\emptyset.\]
We know that $\bB_\rho=\{\nu\in {}^\kappa 2: \nu(i_\rho)=c_\rho\}$ (see
$(\circledast)$(e)) and $\lh(\tr(p))>i_p$ (see $(\otimes)$(c)), so
\[\nu\in \bB_\rho\iff \tr(p)(i_p)=c_p\iff t^\rho_p=1.\]

\noindent {\sc Case 2:}\quad $|\suc_\cT(\rho)|=1$.\\
Let $p$ be the unique element of $\cI_\rho=\cI_{\rho\char 94\langle
  0\rangle}$ such that $\nu\in\lim_\kappa(p)$. Then 
\[\nu\in\bB_\rho\iff \nu\notin \bB_{\rho\char 94\langle
  0\rangle}\iff t^{\rho\char 94\langle 0\rangle}_p=0\iff t^\rho_p=1.\] 

\noindent {\sc Case 3:}\quad $\suc_\cT(\rho)$ is infinite.\\
Let $p$ be the unique element of $\cI_\rho$ such that
$\nu\in\lim_\kappa(p)$.\\
First, assume $t^\rho_p=1$. Thus $p\Vdash$``$\name{\eta}\in \bB_\rho= 
\bigcap\{\bB_\varrho: \varrho\in\suc_\cT(\rho)\}$''. Suppose that
$\varrho\in\suc_\cT(\rho)$ and let $q$ be the unique element of
$\cI_\varrho$ such that $\nu\in\lim_\kappa(q)$. Then, by $(\otimes)$(f),
$p\leq q$ and hence $q\Vdash$``$\name{\eta}\in \bB_\rho\subseteq
\bB_\varrho$'', so $t^\varrho_q=1$. By the inductive hypothesis we get
$\nu\in\bB_\varrho$. Since $\varrho\in\suc_\cT(\rho)$ was arbitrary we
conclude that $\nu\in \bigcap\{\bB_\varrho: \varrho\in\suc_\cT(\rho)\}=
\bB_\rho$.\\
Second, assume $t^\rho_p=0$. By $(\otimes)$(e) we know that
$p\Vdash$``$\name{\eta} \notin\bB_\varrho$'' for some
$\varrho\in\suc_\cT(\rho)$. Let $q\in \cI_\varrho$ be the unique element
such that $\nu\in\lim_\kappa(q)$. Then $p\leq q$ and hence
$t^\varrho_q=0$. By the inductive hypothesis we get $\nu\notin \bB_\varrho$
and hence also $\nu\notin\bB_\rho$. 
\smallskip

Finally note that our  assumption ``$\Vdash\name{\eta}\notin\bB$'' implies
that $t^{\langle\rangle}_p=0$ for all $p\in\cI_{\langle\rangle}$. Therefore,
$(\heartsuit)_{\langle\rangle}$ implies $Y\cap \bB=\emptyset$, so $\bB\in
\id_2({\bold i})$. 
\medskip

\noindent
(6)\quad This is similar but easier.
\end{PROOF}

\begin{definition}
\label{p3}
1) For $\bold i$ as in \ref{p2}, we define $\cov(\bold i),\add(\bold
i),\non(\bold i),\cf(\bold i)$ as those numbers for the ideal $\id({\bold
  i})$, see \ref{z26}.

\noindent
2) If $\kappa_{\bold i},\name\eta_{\bold i}$ are clear from $\bbQ_{\bold i}$
we may write $\bbQ_{\bold i}$ instead of $\bold i$ and write
$\id(\bbQ_{\bold i})$ etc. In particular we will be using this convention
for $\bbQ_\kappa$ from \S1 and for Cohen$_\kappa$. 
\end{definition}

\noindent
Recalling $S^\kappa_{\inc} = \{\partial:\partial < \kappa$ is
inaccessible$\}$, note that for low inaccessible $\kappa$'s,
$\bbQ_\kappa$ is like $\kappa$-Cohen, that is,

\begin{claim}
\label{p5}
1) If $\kappa > \sup(S^\kappa_{\inc})$ \then \, for some open dense subsets
$\cI_1,\cI_2$ of $\bbQ_\kappa,\Cohen_\kappa$ respectively, we have
$\bbQ_\kappa \rest \cI_1 \cong \Cohen_\kappa \rest \cI_2$.

\noindent
2) If $S \subseteq S^\kappa_{\inc}$ is bounded in $\kappa$ \then \,
$\bbQ_{\kappa,S}$ satisfies the conclusion of part (1), where
$\bbQ_{\kappa,S}$ is naturally defined as $\bbQ_\kappa\rest \{p:
S_p\subseteq S\}$.
\end{claim}

\begin{PROOF}{\ref{p5}}
1) Let $\mu = \sup(S^\kappa_{\inc})$, so $\mu < \kappa$.

Let $\cI_1 = \{p \in \bbQ_\kappa:\ell g(\tr(p)) \ge \mu\}$, let $\cI_2 = \{\eta
\in \text{ Cohen}_\kappa:\ell g(\eta) \ge \mu\}$ and $F:\cI_1
\rightarrow \cI_2$ be $F(p) = \tr(p)$.

\noindent
2) Similarly.
\end{PROOF}

\begin{claim}
\label{p8}
1) $\id(\bbQ_\kappa)$ is a $\kappa^+$--complete ideal on ${}^\kappa 2$
   and also $\id(\Cohen_\kappa)$ is.

\noindent
2) If $\kappa$ is weakly compact and $\cI_\alpha \subseteq
\bbQ_\kappa$ is pre-dense for $\alpha < \alpha_* < \kappa^+$ \then
 \, the sets $\cJ^*_1,\cJ^*_2$ are dense open subsets of $\bbQ_\kappa$ where

\begin{equation*}
\begin{array}{clcr}
\cJ^*_1 = \big\{p \in \bbQ_\kappa: &\text{for every } \alpha < \alpha_*
  \text{ there is } \partial < \kappa \text{ such that }\ \\
  &[\eta \in p \cap {}^\partial 2 \Rightarrow p^{[\eta]} \text{ is
  above some } q \in \cI_\alpha]\big\}.
\end{array}
\end{equation*}
and
\[\cJ^*_2 = \big\{p \in \bbQ_\kappa:{\lim}_\kappa(p)\subseteq
\bigcap_{\alpha<\alpha_*} \set(\cI_\alpha)\big\}\] 
(see \ref{p1}(2)). 

\mn 
3) Assume $\kappa$ is weakly compact.  Suppose that $p \in \bbQ_\kappa$ as
witnessed by $(\tr(p),S_p,\bar{\Lambda}_p)$, $\alpha < \kappa$ and let $\bold B
\subseteq {}^\kappa 2$ be a $\kappa$--Borel set. \Then \, there is
$q\in\bbQ_\kappa$ such that:  
\begin{enumerate}
\item[(i)] $p \le q$, $\tr(p) = \tr(q)$, 
\item[(ii)] $S_p \cap\alpha = S_q \cap \alpha$, $\bar\Lambda_p \rest \alpha =
  \bar\Lambda_q \rest \alpha$ and 
\item[(iii)] for some $\beta \in (\alpha,\kappa)$, if $\nu \in q \cap
  {}^\beta 2$ then  
\begin{enumerate}
\item[either]
  $q^{[\nu]} \Vdash ``\name\eta \in \bold B$'' and  $\lim_\kappa(q^{[\nu]})
  \subseteq \bold B$,  
\item[or] $q^{[\nu]} \Vdash ``\name\eta\notin \bold B$'' and
  $\lim_\kappa(q^{[\nu]}) \cap \bold B = \emptyset$. 
\end{enumerate}
\end{enumerate}
\end{claim}

\begin{PROOF}{\ref{p8}}
1) By \ref{p4d}(1).

\noindent
2) By \ref{n13}(2), pedantically by its proof.

\noindent
3) We prove this by the induction on the depth $\gamma$ of (the
$\kappa$--Borel representation; see the proof of \ref{p4d}(5)) of $\bold B$.   
\medskip

\noindent
\underline{\sc Case 1}:\quad  $\gamma = 0$ so $\bold B = \{\nu \in
{}^\kappa 2:\nu(i) = c\}$ for some $i<\kappa$, $c<2$.

Obvious.
\medskip

\noindent
\underline{\sc Case 2}:\quad  $\bold B$ is the complement of a
$\kappa$--Borel set $\bold B_1$ of depth $<\gamma$.

Obvious by the phrasing of (3)(iii).
\medskip

\noindent
\underline{\sc Case 3}:\quad  $\bold B = \bigcap\limits_{\alpha < \alpha(*)}
\bold B_\alpha$, where $\alpha(*) \le \kappa$ and $\bold B_\alpha$ are
$\kappa$--Borel sets of depth $<\gamma$.

\noindent Let $\cI^1_\alpha = \{q \in \bbQ_\kappa:q$ satisfies (3)(iii) for
$\bold B_\alpha$ and $\alpha$ with $\beta=\beta_{q,\alpha}<\kappa\}$. By the
induction hypothesis $\cI^1_\alpha$ is dense open in $\bbQ_\kappa$. Let 
\begin{equation*}
\begin{array}{clcr}
\cI_2 = \big\{q\in \bbQ_\kappa: &\text{ \underline{ either }}\  q\Vdash
``\name\eta \notin \bold B_\alpha^{\bold V[\bbQ_\kappa]}" \text{ for some }
\alpha = \alpha(q) < \alpha_* \\ 
  &\text{ \underline{ or }}\ q \Vdash ``\name \eta \in \bold B^{\bold
    V[\bbQ_\kappa]}" \big\}.
\end{array}
\end{equation*}
Clearly $\cI_2$ is dense open. Let
\[\cI_{3,1} = \{q\in\cI_2:q \Vdash ``\name\eta \notin \bold B_{\alpha(q)}"
\text{ and } q \in \cI^1_{\alpha(q)}\}. \]
Then for $q\in \cI_{3,1}$ we have $(\exists \beta)(\forall \nu \in q
\cap {}^\beta 2)({\lim}_\kappa(q^{[\nu]}) \cap \bold B_{\alpha(q)} =
\emptyset)$ and hence $\lim_\kappa(q) \cap \bold B_{\alpha(q)}=
\emptyset$ for $q\in\cI_{3,1}$. We let
\[\cI_{3,2} = \{q\in\bbQ_\kappa:q \Vdash ``\name\eta \in \bold B" \text{ and } 
{\lim}_\kappa(q) \subseteq {\bold B}\}\] 
and finally we set $\cI_3 = \cI_{3,1} \cup \cI_{3,2}$.

Next consider:
\begin{enumerate}
\item[$(\divideontimes)$]  for every $q_0\in \bbQ_\kappa$ there is $q \in \cI_3$
  above $q_0$.
\end{enumerate}
\mn
Why is $(\divideontimes)$ sufficient?  First note that for every $q \in
\cI_3$ the demand (3)(iii) hold for the pair $(q,\bold B)$. Indeed, by the
definition of $\cI_3$ we have to check the two possibilities: $q \in
\cI_{3,1}$ and $q \in \cI_{3,2}$.  If $q \in \cI_{3,1}$, then $\alpha(q)$ is
well defined and $\lim_\kappa(q)\cap {\bold B}_{\alpha(q)}=\emptyset$, so
$\beta=0$ is as required.  If $q \in \cI_{3,2}$ then also $\beta =0$ is as 
required. Now we may use $(\divideontimes)$ and \ref{n13}(2) to get
$q\in\bbQ_\kappa$ satisfying (i)--(iii) of (3).

\mn Why does $(\divideontimes)$ hold?  Let $q_0\in\bbQ_\kappa$ be given. We
may find $q_1$ above $q_0$ such that either $q_1\Vdash \name{\eta}\in\bB$ or
$q_1\Vdash \name{\eta}\notin\bB$. First assume that the latter is true. Then
for some $\alpha<\alpha(*)$ and $q_2\geq q_1$ we have $q_2\Vdash
\name{\eta}\notin\bB_\alpha$. By the inductive hypothesis there is $q_3\geq
q_2$ satisfying (3)(iii) for $\bB_\alpha$ and $\alpha$. Since $q_3\Vdash
\name{\eta}\notin\bB_\alpha$, this implies
$\lim_\kappa(q_3)\cap\bB_\alpha=\emptyset$ and therefore $q_3\in \cI_{3,1}
\subseteq \cI_3$.

Second, assume $q_1\Vdash \name{\eta}\in\bB$, i.e., 
$q_1\Vdash$``$\name{\eta}\in\bB_\alpha$ for every $\alpha<\alpha(*)$''. Let 
\[\cI_{3,2,\alpha}=\big\{r\in\bbQ_\kappa: r\mbox{ is incompatible with $q_1$ or
} q_1\leq r \mbox{ and }{\lim}_\kappa(r)\subseteq \bB_\alpha\big\};\]
by the inductive hypothesis it is an open dense set. By \ref{p8}(2) we may
find $q_4\geq q_1$ such that  
\[\big(\forall\alpha <\alpha(*)\big)\big(\exists\partial<\kappa\big)
\big(\eta \in q_4 \cap {}^\partial 2\ \Rightarrow\ (q_4)^{[\eta]}\in
\cI_{3,2,\alpha}\big).\] 
Since $(q_4)^{[\eta]}\in\cI_{3,2,\alpha}$ implies
${\lim}_\kappa((q_4)^{[\eta]})\subseteq \bB_\alpha$ (as $q_4\geq q_1$), we 
conclude $\lim_\kappa(q_4)\subseteq {\bold B}_\alpha$ for all
$\alpha<\alpha(*)$. Hence $q_4\in \cI_{3,2}\subseteq\cI_3$.
\end{PROOF}

\begin{claim}
\label{p80}
Considering ${}^\kappa 2$ as an Abelian Group (with addition $\oplus$ modulo
2, coordinatewise), the ideal $\id(\bbQ_\kappa)$ is closed under
translation, i.e. if $\bold B \subseteq {}^\kappa 2$ and $\eta \in {}^\kappa
2$ then $\bold B \in \id(\bbQ_\kappa) \Leftrightarrow \eta \oplus \bold B
\in \id(\bbQ_\kappa)$ where $\eta \oplus \bold B := \{\eta \oplus \nu:\nu
\in \bold B\}$.
\end{claim}

\begin{PROOF}{\ref{p80}}
Straightforward.
\end{PROOF}

\begin{claim}
\label{p24}
If $\kappa$ is an inaccessible limit of inaccessibles, \then \, ${}^\kappa
2$ can be partitioned to two sets $A_0,A_1$ such that $A_0$ is in
$\id(\Cohen_\kappa)$ and $A_1$ is in $\id(\bbQ_\kappa)$. 
\end{claim}

\begin{PROOF}{\ref{p24}}
Let $\langle \kappa_i:i < \kappa\rangle$ list the inaccessibles 
$< \kappa$ in the increasing order and let 
\[\cI_{\kappa_{i+1}} = \big\{q \in\bbQ_{\kappa_{i+1}}:\ell g(\tr(q)) > 
\kappa_i\mbox{ and } \tr(q) \rest [\kappa_i,\ell g(\tr(q))\mbox{ is not
  constantly zero }\big\}.\]   
Clearly, $\cI_{\kappa_{i+1}}$ is an open dense subset of
$\bbQ_{\kappa_{i+1}}$. Now, for $\eta \in {}^{\kappa >}2$ let $p_\eta \in
\bbQ_\kappa$ be witnessed by $(\eta,\{\kappa_{i+1}:\kappa_i > \ell
g(\eta)\},\langle \Lambda_{\kappa_{i+1}}:\kappa_i > \ell g(\eta)\rangle$)
where $\Lambda_{\kappa_{i+1}} = \{\cI_{\kappa_{i+1}}\}$. Then
\begin{enumerate}
\item[(a)] $p_\eta$ indeed belongs to $\bbQ_\kappa$,
\item[(b)] $\tr(p_\eta) = \eta$,
\item[(c)] $p_\eta$ is a nowhere-dense subtree of ${}^{\kappa >}2$. 
\end{enumerate}
Let $A_0 = \bigcup\{{\lim}_\kappa(p_\eta):\eta \in {}^{\kappa >}2\}$, $A_1  
={}^\kappa 2 \backslash A_0$. Let us argue that they are as required. 

First, why does $A_1$ belong to $\id(\bbQ_\kappa)$?   Clearly $A_1$ is 
$\kappa$--Borel and for $p \in \bbQ_\kappa$ we shall prove $p \nVdash
``\name\eta \in A_1"$, this suffices.  Let $\nu = \tr(p)$, hence $p,p_\nu$
are compatible so let $q \in \bbQ_\kappa$ be a common upper bound.  Then $q 
\Vdash$ ``$\name\eta \in \lim_\kappa(q) \subseteq \lim_\kappa(p_\nu)
\subseteq A_0 = {}^\kappa 2 \backslash A_1$''.  

Second, why does $A_0\in\id(\Cohen_\kappa)$? Because it is the union of
$|{}^{\kappa>}2|=\kappa$ nowhere dense sets (remember clause (c)).  
\end{PROOF}

\begin{claim}
\label{p19}
1) [$\kappa$ weakly compact]  Any $\kappa$--Borel set $\bold B$ is
   equal modulo $\id(\bbQ_\kappa)$ to the union of $\le \kappa$ sets, 
each is $\kappa$-closed and even $\bbQ_\kappa$--basic, see Definition
   \ref{y8}(2).

\noindent
2) $\Borel_\kappa/\id(\bbQ_\kappa)$ is a $\kappa^+$--c.c. Boolean Algebra.
\end{claim}

\begin{PROOF}{\ref{p19}}
  1) We have $\id_1(\bbQ_\kappa) = \id_2(\bbQ_\kappa)$ by \ref{p4d}(5).  As
  $\bbQ_\kappa$ satisfies the $\kappa^+$-c.c. it is enough to show that for
  a dense set of $p \in \bbQ_\kappa$, we have that $\lim_\kappa(p)\subseteq
  \bold B$ or $\lim_\kappa(p)$ is disjoint from $\bold B$. But this easily
  holds by \ref{p8}(3).

\noindent
2) Should be clear.
\end{PROOF}

\begin{claim}
\label{p22}
[$\kappa$ weakly compact]
Assume $F$ is a $\kappa$--Borel function from ${}^\kappa 2$ to ${}^\kappa
2$. 

\noindent
For a dense set of $p \in \bbQ_\kappa$, the function $F$ can be read
continuously on $\lim_\kappa(p)$, i.e. for some club $C$ of $\kappa$ and
$\bar h = \langle h_\alpha:\alpha \in C\rangle$ we have:
\begin{enumerate}
\item[(i)]  $h_\alpha:p \cap {}^\alpha 2 \longrightarrow {}^\alpha 2$, 
\item[(ii)]  if $\eta \in p \cap {}^\alpha 2$, $\nu \in p \cap {}^\beta 2$,
  $\eta\vartriangleleft \nu$ and $\{\alpha,\beta\} \subseteq C$ then
  $h_\alpha(\eta) \vartriangleleft h_\beta(\nu)$, 
\item[(iii)]  if $\eta \in \lim_\kappa(p)$ then $F(\eta) =
\bigcup\{h_\alpha(\eta \rest \alpha):\alpha \in C\}$. 
\end{enumerate}
\end{claim}

\begin{remark}
  This is parallel to ``every Borel function $F:[0,1] \longrightarrow [0,1]$
  can be approximated by step functions, that is functions such that for
  some finite partitions of $[0,1]$ to intervals, it is constant on each
  interval''.
\end{remark}

\begin{PROOF}{\ref{p22}}
By \ref{n13}(2), the set
\[\cI=\big\{q\in\bbQ_\kappa:(\forall\alpha<\kappa)(\exists\beta<\kappa)(\forall\nu\in
q\cap {}^\beta 2)(q^{[\nu]}\mbox{ forces a value to }F(\name\eta) \rest\alpha)\big\}\]
is an open dense subset of $\bbQ_\kappa$. 

Let us fix $q\in\cI$. Then by the definition of $\cI$ there are
an increasing sequence $\langle\beta(q,\alpha): \alpha<\kappa\rangle$ 
of ordinals below $\kappa$ and a sequence $\langle
g(q,\alpha):\alpha<\kappa\rangle$ of functions such that for each
$\alpha<\kappa$ we have 
\[g(q,\alpha):{}^{\beta(q,\alpha)} 2\longrightarrow {}^\alpha 2\quad\mbox{
  and }\quad \nu\in q\cap {}^{\beta(q,\alpha)} 2\ \Rightarrow\ q^{[\nu]}\Vdash\mbox{``}
F(\name\eta) \rest\alpha=g(q,\alpha)(\nu)\mbox{''.}\] 
Let $E_q=\{\delta<\kappa: \delta$ is a limit ordinal and $(\forall
\alpha<\delta) (\beta(q,\alpha)<\delta)\}$; clearly it is a club of
$\kappa$. For $\delta\in E_q$ we define a function
$h_{q,\delta}:q\cap {}^\delta2\longrightarrow q\cap {}^\delta 2$ by:
\[h_{q,\delta}(\nu)=\bigcup\big\{g(q,\alpha)(\nu\rest
\beta(q,\alpha)):\alpha<\delta\big\}\quad \mbox{ for }\nu\in q\cap {}^\delta
2.\] 
Clearly, for every $\delta\in E_q$ and $\nu\in {}^\delta 2$ we have
\begin{enumerate}
\item[$(\boxtimes)$] $q^{[\nu]}\Vdash$ `` $F(\name{\eta})\rest\delta=
  \bigcup\limits_{\alpha<\delta} (F(\name{\eta})\rest\alpha)=
  \bigcup\limits_{\alpha<\delta} g(q,\alpha)(\name{\eta}\rest
  \beta(q,\alpha))=h_{q,\delta}(\nu)$ ''. 
\end{enumerate}
For $\delta\in E_q$ and $\nu\in {}^\delta 2$ consider the set 
\[Y_{\delta,\nu}=\big\{\eta\in{\lim}_\kappa(q): \nu\vartriangleleft \eta\ \mbox{
  and } \ F(\eta)\rest\delta\neq h_{q,\delta}(\nu)\big\}.\] 
It is a $\kappa$--Borel set which (by $(\boxtimes)$) belongs to
$\id_1(\bbQ_\kappa)=\id_2(\bbQ_\kappa)$. Hence 
\[Y:=\bigcup\big\{Y_{\delta,\nu}:  \delta\in E_q\mbox{ and }\nu\in {}^\delta
2\big\} \in \id(\bbQ_\kappa).\]
Let $q^*\geq q$ be such that $\lim_\kappa(q^*)\cap Y=\emptyset$ (exists by
the proof of \ref{p19}(1)). Then $q^*, E_q, \langle h_{q,\delta}:\delta\in
E_q\rangle$ have the properties required in (i)--(iii) and the Claim
follows. 
\end{PROOF}

\begin{remark}
  \label{40C}
For $\kappa$ which is not weakly compact we may get a weaker result for
$\id_1(\bbQ_\kappa)=\id_2(\bbQ_\kappa)$. For each $\alpha<\kappa$ let
$\cI_\alpha$ be a maximal antichain of $\bbQ_\kappa$ such that 
\[q\in\cI_\alpha\quad\Rightarrow\quad q\mbox{ forces a value to }
F(\name{\eta})\rest \alpha.\] 
\Wilog \,
\begin{enumerate}
\item[$(*)_0$]  $\alpha < \beta \wedge q \in \cI_\beta \Rightarrow
  (\exists p \in \cI_\alpha)(p \le q)$
\end{enumerate}
Let $\langle q_{\alpha,i}:i<i(\alpha)\leq\kappa\rangle$ list $\cI_\alpha$ and
let $\nu_{\alpha,i}$ be such that $q_{\alpha,i}\Vdash$``$F(\name{\eta})
\rest \alpha=\nu_{\alpha,i}$''. Then clearly $\tr(q_{\alpha,j})
\trianglelefteq\tr(q_{\alpha,i})\in q_{\alpha,j}\Leftrightarrow i=j$. Let 
$Y_\alpha=\bigcup\limits_{i<i(\alpha)} \lim(q_{\alpha,i})$ and note that:
\begin{enumerate}
\item[$(*)_1$] 
\begin{enumerate}
\item[(a)] $Y_\alpha={}^\kappa 2\mod \id(\bbQ_\kappa)$ decreases with
  $\alpha$, and 
\item[(b)] $\langle \lim_\kappa(q_{\alpha,i}):i<i(\alpha)\rangle$ is a
    partition of $Y_\alpha$. 
\end{enumerate}
\end{enumerate}
Define $H_\alpha:Y_\alpha\longrightarrow {}^\alpha 2$ by
$H_\alpha(\eta)=\nu_{\alpha,i}$ if $\eta\in \lim_\kappa(q_{\alpha,i})$. Then
\begin{enumerate}
\item[$(*)_3$]
\begin{enumerate}
\item[(a)] $H_\alpha$ is continuous on $Y_\alpha$ in the sense that
$H_\alpha(\eta)$ is the value of $H_\alpha'(\eta\rest j)$ for every large
enough $j<\kappa$, where  
\item[(b)] we let $H_\alpha':{}^{\kappa>} 2\longrightarrow {}^{\kappa>} 2$
  be 
\[H_\alpha'(\nu)=\left\{\begin{array}{ll}
\nu_{\alpha,i}&\mbox{ if }\tr(q_{\alpha,i})\trianglelefteq \nu\in
q_{\alpha,i},\\
\langle (0)_\alpha\rangle & \mbox{ if there is no such }i.
\end{array}\right.\]
\end{enumerate}
\end{enumerate}
Now consider
\begin{enumerate}
\item[$(*)_4$]
\begin{enumerate}
\item[(a)] $Y=\bigcap\limits_{\alpha<\kappa} Y_\alpha$ and note $Y={}^\kappa
  2\mod\id(\bbQ_\kappa)$, and
\item[(b)] let $H:Y\longrightarrow {}^\kappa 2$ be defined by
  $H(\eta)=\lim\langle H_\alpha(\eta):\alpha<\kappa\rangle$. 
\end{enumerate}
\end{enumerate}
\end{remark}
\medskip

\noindent
Concerning Lebesgue Density Theorem:
\begin{conclusion}
\label{p23}
[$\kappa$ weakly compact]  If $X \subseteq {}^\kappa 2$ is
$\kappa$-Borel, \then \, for some $Y \in \id(\bbQ_\kappa)$ for every
$\eta \in X \backslash Y$ for every $\alpha < \kappa$ large enough
$(2^\kappa)^{[\eta \rest \alpha]} \cap X$ includes $\lim_\kappa(p)$ for some
$p \in \bbQ_\kappa$.
\end{conclusion}

\begin{remark}
So this holds also for the complement of $X$.
\end{remark}

\begin{PROOF}{\ref{p23}}
  By \ref{p8}(3) there is a maximal antichain $\langle p_i:i < i_*\rangle$
  of members of $\bbQ_\kappa$ and $S \subseteq i_*$ such that $i \in S
  \Rightarrow \lim_\kappa(p_i) \subseteq X$ and $i \in i_* \backslash S
  \Rightarrow \lim_\kappa(p_i) \cap X = \emptyset$.  Then $i_* < \kappa^+$ and
  let $Y = {}^\kappa 2 \backslash \bigcup\{\lim_\kappa(p_i):i < i_*\}$, so
  clearly $Y \in \id(\bbQ_\kappa)$.  If $\eta \in X \backslash Y$, then by
  the choice of $Y$ for some $i <i_*$, $\eta \in \lim_\kappa(p_i)$ and
  necessarily $i=i(\eta)$ is unique and $i\in S$. Let $\alpha(\eta)$ be
  ${\ell g}({\rm tr}( p_{i (\eta)}))$. Clearly we are done. 
\end{PROOF}

\begin{claim}
\label{p50}
If $\cI \subseteq \bbQ_\kappa$ is dense open and $W\subseteq \kappa =  
\sup(W)$ \then \, for some $\bar p = \langle p_\rho:\rho \in
\Omega\rangle$ we have:
\begin{enumerate}
\item[(a)]  $\Omega \subseteq {}^{\kappa >}2$, moreover $\Omega
  \subseteq \bigcup\{{}^\alpha 2:\alpha \in W\}$,
\item[(b)]  $p_\rho \in \cI \subseteq \bbQ_\kappa$ has trunk $\rho$
  for every $\rho \in \Omega$,
\item[(c)]  if $\rho \vartriangleleft \nu \in p_\rho$ then $\nu \notin
  \Omega$, 
\item[(d)]  $\{p_\rho:\rho \in \Omega\}$ is a predense subset of
  $\bbQ_\kappa$, moreover is a maximal antichain,
\item[(e)]  letting $(\rho_\rho,S_\rho,\bar\Lambda_\rho)$ witness
$p_\rho \in \bbQ_\kappa$ we have: if $\rho_1,\rho_2 \in \Omega$ and
$\ell g(\rho_1) \le \alpha = \ell g(\rho_2)$ then $S_{\rho_2} =
S_{\rho_1} \backslash (\alpha +1),\bar\Lambda_{\rho_2} =
\bar\Lambda_{\rho_1} \rest S_{\rho_2}$.
\end{enumerate}
\end{claim}

\begin{PROOF}{\ref{p50}}
Let $\Omega_1 =\{\tr(p):p \in \cI\}$ and for $\rho \in \Omega_1$
choose $p^1_\rho \in \cI$ such that $\tr(p^1_\rho) = \rho$ and let
$(\rho,S^1_\rho,\bar\Lambda^1_\rho)$ witness $p^1_\rho \in
\bbQ_\kappa$ with $\min(S^1_\rho)>\ell g(\rho)$. Note that 
\[\rho\in\Omega_1\ \wedge\ \rho \trianglelefteq\nu\in p^1_\rho\quad
\Rightarrow\quad \nu\in\Omega_1\]
because $\cI$ is open dense.  Let $S_* = \bigcup\{S^1_\rho:\rho \in
\Omega_1\}$ and note that $S_*$ is a nowhere stationary subset of
$\kappa$. Let $\bar\Lambda = \langle \Lambda_\partial:\partial \in
S_*\rangle$ where 
\[\Lambda_\partial = \bigcup\{\Lambda^1_{\rho,\partial}:\rho\mbox{ 
satisfies }\rho \in \Omega_1 \cap {}^{\partial >}2\mbox{ and }\partial
\in S^1_\rho\}.\]
Easily, if $\partial \in S_*$ then $\Lambda_\partial$ is a set of
$\le \partial$ dense subsets of $\bbQ_\partial$.  

Next, for $\rho \in \Omega_1$ let $p^2_\rho \in \bbQ_\kappa$ be
witnessed by $(\rho,S_*,\bar\Lambda)$.  Now we define
$\Omega_{2,\alpha}$ by induction on $\alpha\in W$ such that
\[\Omega_{2,\alpha} = \big\{\rho \in {}^\alpha 2:\rho \in
\Omega_1\mbox{ and if }\beta\in W\cap\alpha \wedge \varrho \in 
\Omega_{2,\beta} \wedge \varrho \triangleleft \rho\mbox{ then }\rho
\notin p^2_\varrho\big\}.\] 
Lastly, let $\Omega = \bigcup\limits_{\alpha\in W} \Omega_{2,\alpha}$ and
$p_\rho = p^2_\rho$ for $\rho \in \Omega$.  Now check.  
\end{PROOF}

\begin{claim}
\label{p59}
Assume that $\kappa$ is inaccessible limit of inaccessibles and
$W_\varepsilon\subseteq \kappa=\sup(W_\varepsilon)$ for $\varepsilon<\kappa$
are pairwise disjoint. If $A \in \id(\bbQ_\kappa)$ \then \, for some
$(S,\bar\Lambda),\bar p, \bar{\cI}$:
\begin{enumerate}
\item[(a)]  $\bar p = \langle p_\rho:\rho \in {}^{\kappa>} 2 \rangle$,
  $p_\rho \in \bbQ_\kappa$ is defined by $(\rho,S \backslash
  (\ell g(\rho) +1),\bar\Lambda \rest (S\setminus (\ell g(\rho)+1)))$,  
\item[(b)]  $\bar{\cI} = \langle \cI_\varepsilon:\varepsilon <
  \kappa\rangle$,
\item[(c)]  $\cI_\varepsilon \subseteq \{p_\rho:\rho \in {}^{\kappa
    >}2\ \wedge \ell g(\rho)\in W_\varepsilon\}$ is a predense set and even
  a maximal antichain of $\bbQ_\kappa$, 
\item[(d)]  $A \subseteq \bigcup\{{}^\kappa 2
  \backslash\set(\cI_\varepsilon): \varepsilon < \kappa\}$. 
\end{enumerate}
\end{claim}

\begin{PROOF}{\ref{p59}}
Follows by the proof of \ref{p50} but we give details. Let $A \in
\id(\bbQ_\kappa)$, hence there are a maximal antichains
$\cI_\varepsilon$ of $\bbQ_\kappa$ such that $A\subseteq
\bigcup\limits_{\varepsilon < \kappa} \big( {}^{\kappa}2\backslash
\set(\cI_\varepsilon)\big)$. As $\bbQ_\kappa$ satisfies the
$\kappa^+$--c.c. clearly $|\cI_\varepsilon| \le \kappa$. 

Recalling $\kappa = \sup(S^\kappa_{\inc})$ hence \wilog \, each $p \in
\cI_\varepsilon$ is nowhere-dense (see the proof of \ref{p24}) and hence
$|\cI_\varepsilon|=\kappa$. Let $\cI_\varepsilon = \{p_{\varepsilon,i}:i
<\kappa\}$ and suppose that each $p_{\varepsilon,i}$ is defined by 
$(\eta_{\varepsilon,i}, S_{\varepsilon,i},\bar\Lambda_{\varepsilon,i})$.
\Wilog \, $\partial \in S_{\varepsilon,i} \Rightarrow \ell
g(\eta_{\varepsilon,i}) < \partial$. Let
\begin{enumerate}
\item[$(*)_1$]  $S = \{\partial \in S^\kappa_{\inc}$: for some 
$\varepsilon,i < \partial$ we have $\partial \in
S_{\varepsilon,i}\}$. 
\end{enumerate}
Clearly,
\begin{enumerate}
\item[$(*)_2$]  $S$ is a nowhere stationary subset of
  $S^\kappa_{\inc}$.
\end{enumerate}
Let 
\begin{enumerate}
\item[$(*)_3$]  $\bar\Lambda = \langle \Lambda_\partial:\partial
  \in S\rangle$ where for $\partial \in S$ we let 
\[\Lambda_\partial=\bigcup\big\{\Lambda_{\varepsilon,i,\partial}:
\varepsilon< \partial, \  i<\partial \mbox{ and }\partial \in
S_{\varepsilon,i} \big\}.\]
\end{enumerate}
Clearly,
\begin{enumerate}
\item[$(*)_4$]  $(\langle \rangle,S,\bar\Lambda)$ defines a
  condition $p_*\in\bbQ_\kappa$, as $S \subseteq \kappa$ is nowhere
  stationary and if $\partial \in S$ then $\Lambda_\partial$ is a set
  of $\le \partial$ pre-dense subsets of $\bbQ_\partial$.
\end{enumerate}
Lastly,
\begin{enumerate}
\item[$(*)_5$]  
\begin{enumerate}
\item[(a)] for $\rho \in {}^{\kappa >}2$ let $p_\rho= (\rho,S
  \backslash (\ell g(\rho)+1), \bar\Lambda \rest (S \backslash (\ell
  g(\rho)+1)))$, 
\item[(b)] for $\varepsilon < \kappa$ let 
\[\cI'_\varepsilon =\big\{p_\rho: \mbox{for some }i < \kappa\mbox{ we
  have }i,\varepsilon<\ell g(\rho)\in W_\varepsilon\mbox{ and
}\eta_{\varepsilon,i} \trianglelefteq \rho \in p_{\varepsilon,i}\big\}.\] 
\end{enumerate}
\end{enumerate}
Then
\begin{enumerate}
\item[$(*)_6$]  for each $\varepsilon < \kappa$
\begin{enumerate}
\item[(a)]  $\cI'_\varepsilon$ is a predense subset of $\bbQ_\kappa$,
  and 
\item[(b)]  $\set(\cI'_\varepsilon) \subseteq \set(\cI_\varepsilon)$. 
\end{enumerate}
\end{enumerate}
[Why? For clause (a), if $q\in\bbQ_\kappa$ then some
$p\in\cI_\varepsilon$ is compatible with $q$ and hence there is $r\geq
q,p$. Let $i<\kappa$ be such that $p= p_{\varepsilon,i}$ and let
$\rho\in r$ be such that $\ell g(\rho)>\varepsilon, i,\ell
g(\tr(r))$ and $\ell g(\rho)\in W_\varepsilon$. Now,
$p=p_{\varepsilon,i}\leq r$ implies $\eta_{\varepsilon,i}=\tr(p)
\trianglelefteq \tr(r)\trianglelefteq \rho \in r\subseteq
p_{\varepsilon,i}$. Hence $p_\rho\in\cI_\varepsilon'$ has trunk $\rho$ and
hence it is compatible with $r$, so also with $q$. Concerning clause (b),
assume  $\eta\in \set(\cI_\varepsilon')\subseteq {}^\kappa 2$. Then for some
$\rho\in {}^{\kappa>}2$ we have $p_\rho\in\cI_\varepsilon'$ and $\eta \in
\lim_\kappa(p_\rho)$. By the definition of $\cI_\varepsilon'$, for  some
$i<\ell g(\rho)$ we have $\eta_{\varepsilon,i}\trianglelefteq \rho \in
p_{\varepsilon,i}$. Hence $\tr(p_\rho)\in p_{\varepsilon,i}$. By the choice
of $p_\rho$, clearly $\lim_\kappa(p_\rho)\subseteq \lim_\kappa
(p^{[\rho]}_{\varepsilon,i}) \subseteq\lim_\kappa(p_{\varepsilon,i})
\subseteq \set(\cI_\varepsilon)$, so we are done.]   

To get ``$\cI_\varepsilon$ a maximal antichain'' we choose
$\Omega_{\varepsilon,j} \subseteq {}^j 2$ by induction on $j\in
W_\vare\setminus (\vare+1)$ by:   
\begin{enumerate}
\item[$(*)_7$]  $\Omega_{\varepsilon,j} = \big\{\rho \in {}^j 2:
\text{ for some } i\in W_\vare\cap j\setminus (\vare+1),\
\eta_{\varepsilon,i} \trianglelefteq \rho \in p_{\varepsilon,i} \text{ but
  for no } i_1 \in W_\vare \cap j\setminus (\vare+1) \text{ and } \nu \in 
\Omega_{\varepsilon,i_1} \text{ do we have } \rho \in p_\nu\big\}$.
\end{enumerate}
Then let
\begin{enumerate}
\item[$(*)_8$]  
\begin{enumerate}
\item[(a)] $\Omega_\varepsilon = \bigcup\{\Omega_{\varepsilon,j}: j\in
  W_\vare\setminus (\vare+1)\}$,   
\item[(b)] $\cI''_\varepsilon = \{p_\rho:\rho \in
  \Omega_\varepsilon\}$. 
\end{enumerate}
\end{enumerate}
Now $(S,\Lambda),\langle p_\rho:\rho\in\Omega_\varepsilon\rangle$ and 
$\langle \cI''_\varepsilon:\varepsilon <\kappa\rangle$ are as required.
\end{PROOF}

\section{On $\add(\bbQ_\kappa)$ and $\cf(\bbQ_\kappa)$}
\label{addQ}
\begin{definition}
\label{p84}
\begin{enumerate}
\item  For $\alpha<\kappa$, $\nu \in {}^\alpha 2$, $p\in\bbQ_\kappa$, $\eta
  \in p\cap {}^\alpha 2$ we let  
\[p^{[\eta,\nu]} = \big\{\rho:\rho \trianglelefteq
   \nu\mbox{ or for some $\varrho$ we have }\eta \char 94 \varrho \in p 
   \wedge \rho = \nu \char 94 \varrho\big\}.\] 
\item For $\cI \subseteq \bbQ_\kappa$, $\alpha<\kappa$ and a permutation 
$\pi$ of ${}^\alpha 2$ let 
\[\cI^{[\alpha,\pi]} = \big\{p^{[\eta,\nu]}:p
 \in \cI,\eta \in p \cap {}^\alpha 2\mbox{ and }\nu = \pi(\eta)\big\}.\]
\item Let $\Lambda$ be a collection of subsets of $\bbQ_\kappa$ and let
  $\alpha<\kappa$. For a permutation $\pi$ of ${}^\alpha 2$ we let
\[\Lambda^{[\alpha,\pi]} = \{\cI^{[\alpha,\pi]}:\cI \in \Lambda\}.\] 
We also define 
\[\Lambda^{[\alpha]}=\{\cI^{[\alpha,\pi]}:\pi\mbox{ is a permutation of }
{}^\alpha 2\mbox{ and }\cI \in \Lambda\}\] 
and $\Lambda^{[<\alpha]}=\bigcup\{\Lambda^{[\beta]}:\beta < \alpha\}$, here
we allow $\alpha = \kappa$.  
\end{enumerate}
\end{definition}

\begin{claim}
\label{p87}
\begin{enumerate}
\item If $\alpha < \kappa$ and $\cI \subseteq \bbQ_\kappa$ is
  open/dense/predense/maximal antichain/of cardinality $\le \kappa$ \then 
  \, so is $\cI^{[\alpha,\pi]}$ in $\bbQ_\kappa$.
\item If $\alpha < \kappa$ and $\Lambda$ is a collection of subsets of
  $\bbQ_\kappa$, \then \, 
\begin{itemize}
\item $\big(\Lambda^{[\alpha]}\big)^{[\alpha]} = \Lambda^{[\alpha]}$ and
  $|\Lambda^{[\alpha]}| \le |\Lambda| + 2^{2^{|\alpha|}} + \aleph_0 \le 
  |\Lambda| + \kappa$,
\item $\big(\Lambda^{[<\alpha]}\big)^{[<\alpha]} = \Lambda^{[<\alpha]}$ and
  $|\Lambda^{[<\alpha]}| \le |\Lambda| + \Sigma\{2^{2^{|\beta|}}:\beta <
  \alpha\} \le |\Lambda| + \kappa$.   
\end{itemize}
\end{enumerate}
\end{claim}

\begin{PROOF}{\ref{p87}}
Easy.
\end{PROOF}

\begin{definition}
  \label{p36a}
  \begin{enumerate}
  \item For an inaccessible cardinal $\kappa$ let $\Pr(\kappa)$ mean:

{\em there are predense sets $\cI_\vare\subseteq \bbQ_\kappa$ for 
  $\vare<\kappa$ such that 

if $p\in\bbQ_\kappa$ \then\; $\lim_\kappa(p)\nsubseteq
\bigcap\limits_{\vare<\kappa} \set(\cI_\vare)$.}
\item Let $S^\kappa_\pr=\{\partial<\kappa: \partial\in S^\kappa_\inc\
  \wedge\ \Pr(\partial)\}$ and 
\[\nst^\pr_\kappa=\nst_{\kappa,\pr}=\{S\subseteq
  S^\kappa_\inc: S\mbox{ is nowhere stationary and }S\subseteq
  S^\kappa_\pr\}.\] 
  \end{enumerate}
\end{definition}

\begin{observation}
\label{obs30}
  \begin{enumerate}
\item  If $\kappa$ is inaccessible but it is not a Mahlo cardinal, then 
  $\Pr(\kappa)$.  
\item If $\kappa$ is weakly compact, then $\neg\Pr(\kappa)$. 
\item If $\kappa=\sup(S^\kappa_\inc)$,  then $\kappa=\sup(S^\kappa_\pr)$. 
\item If $\kappa$ is Mahlo, i.e., $S^\kappa_\inc$ is a stationary subset of
  $\kappa$, then $S^\kappa_\pr$ is a stationary subset of $\kappa$. 
  \end{enumerate}
\end{observation}

\begin{PROOF}{\ref{obs30}}
(1)\quad First assume $\theta=\sup(S^\kappa_\inc)<\kappa$. For
$\vare<\kappa$ define
\[\cI_\vare=\big\{\big({}^{\kappa>}2\big)^{[\nu\char 94\langle 0\rangle]}:
\nu\in {}^{\kappa>}2\ \wedge\ \ell g(\nu)>\vare\big\}.\]
It should be clear that each $\cI_\vare$ is a predense subset of
$\bbQ_\kappa$ and we claim that they witness $\Pr(\kappa)$. So suppose
that $p\in\bbQ_\kappa$ and pick $\nu\in p$ of length 
greater than $\theta$ and than $\ell g(\tr(p))$; note that then 
$p^{[\nu]}=({}^{\kappa>}2)^{[\nu]}$. Let $\eta\in {}^\kappa2$ be such that 
  $\nu\vartriangleleft \eta$ and $\eta(i)=1$ for $i\in [\ell
  g(\nu),\kappa)$. Clearly, $\eta\in\lim_\kappa(p)$ but $\eta\notin
  \set(\cI_\vare)$ for $\vare>\ell g(\nu)$. 

Second, assume $\kappa=\sup(S^\kappa_\inc)$ but it is not Mahlo. Let $E$ be
a club of $\kappa$ disjoint from $S^\kappa_\inc$ and let $\langle
\alpha_i:i<\kappa\rangle$ be the increasing enumeration of $E$. For
$\vare<\kappa$ let
\[\cI_\vare=\big\{\big({}^{\kappa>}2\big)^{[\nu\char 94\langle 0\rangle]}:
\nu\in {}^{\alpha_i}2\ \wedge\ i>\vare\big\}.\]
Clearly, each $\cI_\vare$ is a predense subset of $\bbQ_\kappa$. We will
argue that they witness $\Pr(\kappa)$. Let $p\in\bbQ_\kappa$ and fix $\vare$
such that $\alpha_\vare>\ell g(\tr(p))$. By induction on $i\in
[\vare,\kappa)$ choose $\nu_i\in {}^{\alpha_i}2\cap p$ so that 
\begin{itemize}
\item if $\vare\leq j<i<\kappa$ then $\nu_j\char 94 \langle 1\rangle
  \trianglelefteq \nu_i$.
\end{itemize}
(It is clearly possible; at successor stages remember \ref{n9}(1) and at
limit stages remember the choice of $E$.) Then $\eta:=\bigcup\{\nu_i:
\vare\leq i<\kappa\}\in\lim_\kappa(p)$ does not belong to
$\set(\cI_\vare)$. 
\medskip

\noindent (2)\quad Remember Claim \ref{p8}(2).   
\medskip

\noindent (3,4)\quad Follow from part (1).
\end{PROOF}

\begin{question}
For which inaccessible cardinals $\kappa$ do we have $\Pr(\kappa)$? See
\cite{Sh:F1580}. 
\end{question}

\begin{claim}
\label{p36cSUB}
The following are equivalent for $\kappa$:
\begin{enumerate}
\item[(a)] $\neg\Pr(\kappa)$.
\item[(b)] If $\Lambda$ is a set of $\leq\kappa$ maximal antichains of
  $\bbQ_\kappa$ and $\alpha<\kappa$, \then\, there is $p\in\bbQ_\kappa$ such
  that $\tr(p)=\langle\rangle$, $S_p\cap \alpha=\emptyset$ and
  $\lim_\kappa(p) \subseteq \set(\Lambda)$. 
\end{enumerate}
\end{claim}

\begin{PROOF}{\ref{p36cSUB}}
(b) $\Rightarrow$ (a)\quad Straightforward by Definition \ref{p36a}(1). 

\noindent (a) $\Rightarrow$ (b)\quad Suppose that $\Pr(\kappa)$ does
not hold.
 
Assume $\Lambda$ is a set of $\leq\kappa$ maximal antichains of
$\bbQ_\kappa$. Let $\Lambda_1=\Lambda^{[<\kappa]}$ (see \ref{p84}). Then
$\Lambda_1=(\Lambda_1)^{[<\kappa]}$ and $|\Lambda_1|\leq \kappa$ (remember
\ref{p87}). Since $\Pr(\kappa)$ fails, there is a condition
$q\in\bbQ_\kappa$ such that $\lim(q)\subseteq \set(\Lambda_1)$ and ${\ell
  g}(\tr(q))>\alpha$, $S_q\cap\alpha=\emptyset$.

Let $S_p=S_q$ and for $\partial\in S_p$ let $\Lambda_\partial=
\Lambda_{q,\partial}$. Put $\bar{\Lambda} =\langle
\Lambda_\partial: \partial\in S_p\rangle$ and let $p$ be the condition
determined by $(\langle\rangle, S_p,\bar{\Lambda})$. 

Note that if $\eta\in q\cap {}^\beta 2$, $\beta<\kappa$, then for every
$\nu\in {}^\beta 2$ also   $q^{[\eta,\nu]}$ satisfies $\lim(q^{[\eta,\nu]})
\subseteq \set(\Lambda_1)$ by the choice of $\Lambda_1$. Therefore we also
get $\lim(p) \subseteq \set(\Lambda_1) \subseteq \set(\Lambda)$, so $p$ is
as required.   
\end{PROOF}

\begin{claim}
\label{p36cTWO}
Suppose that $p\in\bbQ_\kappa$, $\ell g(\tr(p))<\alpha_*\leq
\beta_*\leq \kappa$. \Then\, there is $q\in\bbQ_\kappa$ such that 
\begin{enumerate}
\item[(a)] $p\leq q$, $\tr(q)=\tr(p)$ and
\item[(b)] $S_q\setminus (\alpha_*,\beta_*)=S_p\setminus
  (\alpha_*,\beta_*)$ and $\gamma \in S_q \setminus
  (\alpha_*,\beta_*) \Rightarrow \Lambda_{q,\gamma} =
  \Lambda_{p,\gamma}$,  
\item[(c)] $S_q\cap (\alpha_*,\beta_*)\subseteq S^\kappa_\pr$. 
\end{enumerate}
\end{claim}

\begin{PROOF}{\ref{p36cTWO}}
We prove this by induction on $\beta_*$.
\medskip

\noindent\underline{Case 0}:\quad $\alpha_*=\beta_*$ or $\alpha_*+1=
\beta_*$\\ 
Trivial, as then $(\alpha_*,\beta_*)=\emptyset$.
\medskip

\noindent\underline{Case 1}:\quad $\beta_*=\sup(\beta_*\cap S_p)+1$
but $\sup(\beta_*\cap S_p)\notin S_p\setminus S^\kappa_\pr$.\\ 
Let $\gamma_*=\sup(\beta_*\cap S_p)$. Use the inductive hypothesis for
$p$ and $(\alpha_*,\gamma_*)$ to get a condition $q$. It will satisfy
the demands for $(\alpha_*,\beta_*)$ as well as either $\gamma_*\notin
S_p$ or else $\gamma_*\in S^\kappa_\pr$. 
\medskip

\noindent\underline{Case 2}:\quad $\beta_*>\sup(\beta_*\cap S_p)+1$\\
Use the inductive hypothesis for $\gamma_*=\sup(\beta_*\cap S_p)+1$,
proceeding like in Case 1. 
\medskip

\noindent\underline{Case 3}:\quad $\beta_*=\sup(\beta_*\cap S_p)$, so
$\beta_*$ is limit\\
Pick an increasing continous sequence $\bar{\alpha}=\langle\alpha_i:
i\leq \cf(\beta_*) \rangle$ such that $\alpha_0=\alpha_*$,
$\alpha_{\cf(\beta_*)}=\beta_*$ and $\alpha_i\notin S_p$ for all
$0<i<\cf(\beta_*)$. By induction on $i\leq\cf(\beta_*)$ choose $q_i$
such that  
\begin{enumerate}
\item[(a)] $q_0=p$, $\tr(q_i)=\tr(p)$,
\item[(b)] $S_{q_i}\setminus (\alpha_0,\alpha_i)=S_p\setminus
  (\alpha_0,\alpha_i)$ and $\gamma \in S_{q_i} \setminus
  (\alpha_0,\alpha_i)\Rightarrow \Lambda_{{q_i},\gamma} =
  \Lambda_{p,\gamma}$,  
\item[(c)] if $j<i$, then $q_j\leq q_i$, $S_{q_i}\setminus
  (\alpha_j,\alpha_i)=S_{q_j}\setminus (\alpha_j,\alpha_i)$ and $\gamma \in S_{q_i} \setminus
  (\alpha_j,\alpha_i)\Rightarrow \Lambda_{{q_i},\gamma} =
  \Lambda_{q_j,\gamma}$,  
\item[(d)] if $i=j+1$, then $S_{q_i}\cap (\alpha_j,\alpha_i)\subseteq
  S^\kappa_\pr$. 
\end{enumerate}
There are no problems in carrying out the inductive construction. Then
$q_{\cf(\beta_*)}$ is as required. 
\medskip

\noindent\underline{Case 4}:\quad $\beta_*=\partial+1$, $\partial\in
S_p \setminus S^\kappa_\pr$ and $\partial>\alpha_*$\\ 
Here we use \ref{p36cSUB} for $\bbQ_\partial$, $\Lambda_{p,\partial}$ and 
the ordinal $\alpha_*$. So there is $p_*\in \bbQ_\partial$ such that 
\begin{itemize}
\item $\tr(p_*)=\langle\rangle$,
\item $S_{p_*}\subseteq (\alpha_*,\partial)$, and 
\item $\lim(p_*)\subseteq \set(\Lambda_{p,\partial})$.
\end{itemize}
Now we define a condition $q_1$ by letting:
\begin{itemize}
\item $\tr(q_1)=\tr(p)$,
\item $S_{q_1}=(S_p\setminus \{\partial\})\cup S_{p_*}$,
\item $\Lambda_{q_1,\theta}$ is 
  \begin{itemize}
\item $\Lambda_{p,\theta}$ if $\theta\in S_p\setminus S_{p_*}$,
\item $\Lambda_{p_*,\theta}$ if $\theta\in S_{p_*}\setminus S_p$,
\item $\Lambda_{p,\theta}\cup \Lambda_{p_*,\theta}$ if $\theta\in
  S_{p_*}\cap S_p$.
  \end{itemize}
\end{itemize}
Then we continue as in Case 1 with $q_1,\alpha_*,\beta_*$ (as 
$\partial\notin S_{q_1}$). 
\end{PROOF}

\begin{conclusion}  
\label{q35}
For any $\alpha < \kappa$, the set 
$\{p \in \bbQ_\kappa:S_p \subseteq S^\kappa_{\pr} \setminus \alpha\}$ is
dense in $\bbQ_\kappa$.
\end{conclusion}

Note that if $\kappa>\sup(S^\kappa_\inc)$, then $\id(\bbQ_\kappa)=
\id(\Cohen_\kappa)$. Therefore:

\begin{hypothesis}
For the rest of this section we assume that $\kappa=\sup(S^\kappa_\inc)$ (so
also $\kappa=\sup(S^\kappa_\pr)$, remember \ref{obs30}(3)). 
\end{hypothesis}

\begin{definition}
  \label{p36e}
\begin{enumerate}
\item Let $\add(\nst^\pr_\kappa)$ be the minimal cardinal $\mu$ such that
  there are $S_\zeta\in\nst^\pr_\kappa$ for $\zeta<\mu$ with the property
  that there is no $S\in \nst^\pr_\kappa$ satisfying 
\[\zeta<\mu\quad\Rightarrow\quad S_\zeta\subseteq S\mod \mbox{ bounded}.\] 
Dually, $\cf(\nst^\pr_\kappa)$ is the minimal cardinal $\mu$ such that
  there are $S_\zeta\in\nst^\pr_\kappa$ for $\zeta<\mu$ with the property
  that for every $S\in \nst^\pr_\kappa$ there is $\zeta<\mu$ satisfying 
 $S\subseteq S_\zeta\mod \mbox{ bounded}$. 
\item For $S\subseteq S^\kappa_\inc$ we define:
\begin{enumerate}
\item[(a)] $\bbQ^*_{\kappa,S}$ is the subforcing of $\bbQ_\kappa$
  consisting of all conditions $p\in\bbQ_\kappa$  satisfying $S_p\subseteq S$. 
\item[(b)] $\id[\bbQ^*_{\kappa,S}]$ is the collection of all $A
  \subseteq {}^\kappa 2$ such that for some $\bar{\cJ}=\langle
  \cJ_\zeta: \zeta<\kappa\rangle$ we have 
  \begin{enumerate}
\item each $\cJ_\zeta$ is predense subset (or maximal antichain)
    of $\bbQ_\kappa$,
\item $\cJ_\zeta\subseteq\bbQ^*_{\kappa,S}$ for each
  $\zeta<\kappa$, and 
\item $A\subseteq \bigcup\limits_{\zeta<\kappa} \big( {}^\kappa2 \setminus 
  \set(\cJ_\zeta)\big)$,  
  \end{enumerate}
\item[(c)] $\add(\id[\bbQ^*_{\kappa,S}],\id(\bbQ_\kappa))=\min\{|\cA|:
  \cA\subseteq\id[\bbQ^*_{\kappa,S}]\ \wedge\ \bigcup\cA\notin
  \id(\bbQ_\kappa)\}$. 
\end{enumerate}
\item $\Add^*_{\pr,\kappa}=\min\big\{\add(\id[\bbQ^*_{\kappa,S}],
  \id(\bbQ_\kappa)): S\in\nst^\pr_\kappa\big\}$.
\end{enumerate}
\end{definition}

\begin{claim}
\label{p36hx}
\begin{enumerate}
\item $\add(\bbQ_\kappa)=\min\big\{\add(\nst^\pr_\kappa),
  \Add^*_{\pr,\kappa} \big\}$. 
\item $\cf(\bbQ_\kappa) \ge \cf(\nst^{\pr}_\kappa)$. 
\end{enumerate}

\end{claim}

\begin{PROOF}{\ref{p36hx}}
(1)\quad (Step 1)\quad \underline{$\add(\bbQ_\kappa)\leq
  \add(\nst^\pr_\kappa)$}. 
\smallskip

\noindent
Let $S_\zeta\in\nst^\pr_\kappa$ for $\zeta<\add(\nst^\pr_\kappa)$ be such that
\[S\in\nst_\kappa^\pr\quad \Rightarrow \quad \bigvee_\zeta
\kappa=\sup(S_\zeta\setminus S).\] 
For $\partial\in S^\kappa_\pr$ let $\Lambda_\partial^*=\{\cI_\vare^\partial: 
\vare<\partial\}$ witness $\partial\in  S^\kappa_\pr$ (see Definition
\ref{p36a}(1)). For $\zeta<\add(\nst^\pr_\kappa)$ let\footnote{Recall that
  ``$\forall^\infty\partial \in S$'' means ``for all but boundedly many
  $\partial \in S$''.}  
\[\bB_\zeta={}^\kappa2\setminus \big\{\eta\in {}^\kappa 2:
\big(\forall^\infty \partial\in S_\zeta\big)\big(\eta\rest \partial \in
\set(\Lambda^*_\partial)\big)\big\}.\]
Clearly $\bB_\zeta\in\id(\bbQ_\kappa)$. Now it suffices to prove that
$\bB:=\bigcup\big\{\bB_\zeta: \zeta<\add(\nst^\pr_\kappa)
\big\}\notin\id(\bbQ_\kappa)$. So suppose towards a contradiction that
$\bB\in\id(\bbQ_\kappa)$ and let $(S,\bar{\Lambda},\bar{p},\bar{\cI})$ be
given by Claim \ref{p59} for $\bB$. Next,
\begin{enumerate}
\item[$(*)_1$]  if $\vare<\kappa$, $\alpha<\kappa$ and $\eta\in {}^\alpha
  2$, \then \, there are $\beta,\nu,\rho$ such that 
\begin{enumerate}
\item[(a)]  $\alpha<\beta<\kappa$,
\item[(b)]  $\eta\vartriangleleft \nu \in {}^\beta 2$, 
\item[(c)]  $p_\rho \in \cI_\vare$ and $\rho \trianglelefteq \nu$  and $\nu
  \in p_\rho$, 
\item[(d)]  if $\partial\in S\cap (\alpha,\beta]$ then $\nu\rest \partial
  \in \set(\Lambda_\partial)$. 
\end{enumerate}
\end{enumerate} 
[Why?  Consider the triple $(\eta,S \backslash (\alpha +1),\langle
\Lambda_\partial:\partial \in S \backslash (\alpha +1)\rangle)$.  It defines
the condition $p_\eta \in \bbQ_\kappa$ and we know that $\cI_\vare$ is a
predense subset of $\bbQ_\kappa$.  Hence for some $\rho \in {}^{\kappa>}2$,
$p_\rho\in\cI_\vare$ and the conditions $p_\rho,p_\eta$ are compatible in 
$\bbQ_\kappa$. Then there is $\nu \in {}^{\kappa >}2$ such that $\tr(p_\rho)
\vartriangleleft \nu \in p_\rho$, $\tr(p_\eta)\vartriangleleft \nu \in
p_\eta$.  By the definition of $p_\eta$ above, $\ell g(\nu),\nu,\rho$
satisfy all the requirements.]  

Now, 
\begin{enumerate}
\item[$(*)_2$]  For $\vare<\kappa$ let $F^\vare_1,F^\vare_2:{}^{\kappa >}2
  \longrightarrow {}^{\kappa >}2$ be such that for each $\eta \in {}^{\kappa 
    >}2$, the triple $(\beta,\nu,\rho)$ given by $\beta = \ell
  g(F^\vare_1(\eta))$, $\nu = F^\vare_1(\eta)$ and $\rho = F^\vare_2(\eta)$,
  is as required above in $(*)_1$ for $\vare$ and $\eta$.
\item[$(*)_3$]  Let $E_1 = \{\delta < \kappa:\delta$ a limit
  ordinal and $(\vare<\delta\ \wedge\ \eta \in {}^{\delta >}2) \Rightarrow
  F^\vare_1(\eta) \in {}^{\delta >}2\}$.
\end{enumerate}
By the choice of $\langle S_\zeta:\zeta<\add(\nst^\pr_\kappa)\rangle$ there
is $\zeta< \add(\nst^\pr_\kappa)$ such that $S_\zeta\setminus S$ is
unbounded in $\kappa$. Easily we may choose an unbounded set  $S'\subseteq
S_\zeta\setminus S$ such that 
\begin{itemize}
\item the closure $E$ of $S'$ is disjoint from $S$, and
\item if $\gamma_0\in E$, $\gamma_1=\min(E\setminus(\gamma_0+1))$, then 
  $(\gamma_0,\gamma_1)\cap E_1\neq \emptyset$.
\end{itemize}
Let $\langle\gamma_i:i<\kappa\rangle$ list $E\cup\{0\}$ in the increasing
order (so $\gamma_{i+1}\in S_\zeta\setminus S$ and $\gamma_i\notin S$;
remember $\gamma_0=0\notin S\subseteq S^\kappa_\inc$).  By induction on
$i<\kappa$ we choose $\eta_i\in {}^{\gamma_i}2$ such that 
\begin{enumerate}
\item[(a)] $j<i<\kappa\ \Rightarrow\ \eta_j\vartriangleleft \eta_i$,
\item[(b)] if $i=j+1$ then $\eta_i\notin \set(\Lambda^*_{\gamma_i})$ and
  $F^j_1(\eta_j)\trianglelefteq \eta_i$, 
\item[(c)] if $\partial\in S\cap (\gamma_i+1)$, $i>0$, then
  $\eta_i\rest\partial \in \set(\Lambda_\partial)$,
\item[(d)] if $j<i$ then $F^j_2(\eta_j)\trianglelefteq \eta_i\in
  p_{F^j_2(\eta_j)}$ (follows from (b)+(c) and $(*)_2$).   
\end{enumerate}
If we succeed in carrying out the induction, then we may let
$\eta=\bigcup\limits_{i<\kappa} \eta_i$ and note that  
\begin{itemize}
\item $\eta$ belongs to $\bB_\zeta$ because $\eta\rest\gamma_i \notin 
  \set(\Lambda^*_{\gamma_i})$ for all successor $i<\kappa$ by clause (b),
\item $\eta$ does not belong to $\bB$ by clauses (c)+(d).
\end{itemize}
Consequently, $\eta$ witnesses $\bB_\zeta\nsubseteq\bB$, a contradiction. 
\medskip

\noindent
\underline{Why can we carry out the induction?}\\
For $i=0$ it is trivial.

\noindent
For a limit $i<\kappa$ we let $\eta_i=\bigcup\limits_{j<i}\eta_j$.

\noindent
Let $i=j+1$. First, $F^j_1(\eta_j)$ satisfies the requirements on $\eta_i$
except that $\ell g(F^j_1(\eta_j))$ is not $\gamma_i$ (and so ``$ \eta _i
\notin\set(\Lambda^*_{\gamma _i})$'' from (b) is meaningless): it is
$<\gamma_i$ by the choices of $E_1$ and $E$. 

Second, we use the definition of $S_\zeta\subseteq S^\kappa_\pr$ and
$\gamma_{j+1}\in S_\zeta\setminus S$ for the condition with trunk
$F^j_1(\eta_j)$ and $\langle \Lambda_\partial:\partial\in
(\gamma_j,\gamma_{j+1})\cap S\rangle$ and the choice of
$\Lambda^*_{\gamma_i}$. 
\medskip

This completes the proof of ``$\add(\bbQ_\kappa)\leq
\add(\nst^\pr_\kappa)$''. 
\medskip

\noindent (Step 2)\quad \underline{$\add(\bbQ_\kappa)\leq 
  \Add^*_{\pr,\kappa}$}.
\smallskip

\noindent It should be obvious that if $S\subseteq  
S^\kappa_\pr$ then $\add(\bbQ_\kappa)\leq
\add(\id[\bbQ^*_{\kappa,S}],\id(\bbQ_\kappa))$. 
\medskip

\noindent (Step 3)\quad\underline{$\min\big\{ \add(\nst^\pr_\kappa),
  \Add^*_{\pr,\kappa}\big\}\leq \add(\bbQ_\kappa)$}.  
\smallskip

\noindent
Why?  Assume $A_i \in \id(\bbQ_\kappa)$ for $i<i_* <\min 
\big\{\add(\nst^\pr_\kappa),\Add^*_{\pr,\kappa}\big\}$.  For each $i$ let 
$(S_i,\bar\Lambda_i,\bar{\cI}_i,\bar p_i)$ be given by Claim \ref{p59} for
$A_i$.  By Conclusion \ref{q35} (and the proof of \ref{p59}) we may also
require that  $S_i \in\nst^{\pr}_\kappa$ for all $i < i_*$.  As $i_* <
\add(\nst^\pr_\kappa)$ there is $S \in \nst^{\pr}_\kappa$ such that 
\[i < i_* \quad\Rightarrow\quad S_i \subseteq S \mod J^{\bd}_\kappa.\] 
Then easily $A_i \in \id[\bbQ^*_{\kappa,S}]$ for every $i < i_*$.  Since
$i_* < \Add^*_{\pr,\kappa}$ we also have $i_* < \add(\id[\bbQ^*_{\kappa,S}], 
\id(\bbQ_\kappa))$ and hence $\bigcup\limits_{i<i^*}A_i \in
\id(\bbQ_\kappa)$ and we are done.
\medskip

\noindent (2)\quad  In order to show $\cf(\bbQ_\kappa) \ge
\cf(\nst^{\pr}_\kappa)$ let us assume towards contradiction that $\mu := 
\cf(\bbQ_\kappa) < \cf(\nst^{\pr}_\kappa)$. Let $\langle \bold
B_\zeta:\zeta < \mu\rangle$ witness $\mu = \cf(\bbQ_\kappa)$ and let  
$S_\zeta,\bar\Lambda_\zeta$, $\bar{p}_\zeta=\langle p_{\zeta,\rho}:\rho\in
{}^{\kappa>} 2\rangle$ and $\bar{\cI}_\zeta
=\{\cI_{\zeta,i}:i<\kappa\}$ be given by \ref{p59} for $\bold 
B_\zeta$. Let $S \in \nst^{\pr}_\kappa$ be such that 
\[\zeta < \mu\quad \Rightarrow\quad \kappa = \sup(S \backslash
S_\zeta).\] 
For each $\partial \in S$ let
$\Lambda^*_\partial= \{\cI^\partial_\vare: \vare<\partial\}$ witness
$\partial\in S^\kappa_\pr$ (see Definition \ref{p36a}(1)) and let 
\[{\bold B}:=\{\eta\in {}^\kappa 2:(\exists^\infty \partial \in
S)(\exists \vare<\partial)(\eta\rest\partial \notin
\set(\cI^\partial_\vare))\}.\]
Clearly ${\bold B}\in\id(\bbQ_\kappa)$, so for some $\zeta < \mu$ we
have $\bold B \subseteq \bold B_\zeta$. Let $E\subseteq
\kappa\setminus S_\zeta$ be a club and let $p\in\bbQ_\kappa$ be a
condition determined by $(\langle\rangle,
S_\zeta,\bar\Lambda_\zeta)$. By induction on $i<\kappa$ we choose
$\alpha_i\in E$ and $\eta_i\in {}^{\alpha_i}2\cap p$ so that 
\begin{enumerate}
\item[(i)] $\langle \alpha_i:i<\kappa\rangle\subseteq E$ is increasing
  continuous, 
\item[(ii)] $\langle \eta_i:i<\kappa\rangle$ is
  $\vartriangleleft$--increasing continuous, 
\item[(iii)] for each $i<\kappa$, for some $\rho\in {}^{\kappa>} 2$ we
  have $\rho\vartriangleleft \eta_{i+1}$ and $p_{\zeta,\rho}\in
  \cI_{\zeta,i}$, 
\item[(iv)]  for each $i<\kappa$ there is $\partial\in
  (\alpha_i,\alpha_{i+1})\cap S$ such that $\eta_{i+1}\rest \partial
  \notin \bigcap\limits_{\vare<\partial} \set(\cI^\partial_\vare)$. 
\end{enumerate}
It should be clear how to carry out the construction. At the end, the
sequence $\eta:=\bigcup\limits_{i<\kappa} \eta_i\in {}^\kappa 2$
belongs to ${\bold B}$ (by (iv)) but it does not belong to ${\bold
  B}_\zeta$ (by (iii)), contradicting the choice of $\zeta<\mu$. 
\end{PROOF}

\begin{claim}
\label{q7-item2}
If $\kappa$ is Mahlo and there is a non-reflecting stationary set
$S\subseteq S^\kappa_\pr$, then
\begin{enumerate}
\item $\add(\nst^\pr_\kappa) \le \gb_\kappa$, 
\item above we actually have $\add(\nst_{\kappa,S})=\gb_\kappa$,
\item $\gd_\kappa\leq\cf(\nst^{\pr}_\kappa)$.
\end{enumerate}
\end{claim}

\begin{PROOF}{\ref{q7-item2}}
Straightforward, as for $S'\subseteq S$ we have:

\qquad $S'\in\nst^\pr_\kappa$ if and only if $S'$ is non-stationary. 
\end{PROOF}

\section {The parallel of the Cicho\'n Diagram}
\label{CichKappa}
As before, $\lambda,\partial,\kappa$ vary on inaccessibles. 
\bigskip

\noindent
We have a characterization of $\kappa$--meagre sets similar to the one for
the case of $\kappa = \aleph_0$. (Note: here $\kappa$ inaccessible is used.)
\medskip

\begin{observation}
\label{p31}
1)\quad If $X \subseteq {}^\kappa 2$ is $\kappa$--meagre and $A \subseteq
\kappa$ is unbounded \then \, there is an increasing sequence $\bar\alpha$
of members of $A$ of length $\kappa$ and $\eta \in {}^\kappa 2$ such that
\[X \subseteq X_{\eta,\bar\alpha} := \{\nu \in {}^\kappa 2:\mbox{ for every }i <
\kappa\mbox{ large enough, }\eta \rest [\alpha_i,\alpha_{i+1}) \nsubseteq
\nu\}.\]
Moreover, if $A$ contains a club of $\kappa$ then the sequence
$\bar{\alpha}$ above can be increasing continuous.

\noindent 2)\quad If $\eta \in {}^\kappa 2$ and $\bar\alpha$ is an
increasing sequence of ordinals $< \kappa$ of length $\kappa$ \then \, the
set $X_{\eta,\bar\alpha}$ defined above is a $\kappa$-meagre subset of
${}^\kappa 2$.
\end{observation}

\begin{PROOF}{\ref{p31}}
1)\quad Let $X \subseteq \bigcup\{\lim_\kappa(\cT_i):i < \kappa\}$ where $\cT_i$ 
is a nowhere dense subtree of ${}^{\kappa >}2$.  For every infinite $\alpha
\in A$ let $\langle(\eta_{\alpha,\varepsilon},i_{\alpha,\varepsilon}):
\varepsilon<2^{|\alpha|}\rangle$ list ${}^\alpha 2 \times \alpha$, and then
we choose $\nu_{\alpha,\varepsilon},\beta_{\alpha,\varepsilon}$ by induction
on $\varepsilon \le 2^{|\alpha|}$ such that:
\begin{enumerate} 
\item[(a)] $\beta_{\alpha,\varepsilon}=\beta(\alpha,\varepsilon) < \kappa$
  is increasing continuous with $\varepsilon$, 
\item[(b)] $\nu_{\alpha,\varepsilon} \in {}^{\beta(\alpha,\varepsilon)}2$, 
\item[(c)] $\zeta < \varepsilon\quad \Rightarrow\quad   \nu_{\alpha,\zeta}
  \trianglelefteq \nu_{\alpha,\varepsilon}$, 
\item[(d)] $\eta_{\alpha,\varepsilon} \char 94 \nu_{\alpha,\varepsilon +1}
  \notin \cT_{i_{\alpha,\varepsilon}}$. 
\end{enumerate}
Why we can?  For $\varepsilon = 0$, let $\nu_{\alpha,\varepsilon} =
\langle \rangle$, for limit $\varepsilon$ let $\nu_{\alpha,\varepsilon} =
\bigcup\{\nu_{\alpha,\zeta}:\zeta < \varepsilon\}$ recalling (by \ref{pp8})
that $\cf(\kappa) = \kappa > 2^{|\alpha|}\geq\varepsilon$ and for
$\varepsilon = \zeta +1$ use  ``$\cT_{i_{\alpha,\varepsilon}}$ is nowhere
dense subtree of ${}^{\kappa >}2$''. 

Now by induction on $i < \kappa$ we choose $(\alpha_i,\nu_i)$ such that: 
\begin{enumerate}
\item[(e)]  $\alpha_i \in A$ is infinite increasing with $i$, $\alpha_i$
  minimal under these restrictions, 
\item[(f)]  $\nu_i \in {}^{\alpha_i}2$ is $\vartriangleleft$--increasing, 
\item[(g)]  if $i = j+1$ and $\gamma = 2^{|\alpha_j|}$ then $\alpha_i =
  \min\{\alpha \in A:\alpha > \alpha_j + \ell g(\nu_{\alpha_j,\gamma})\}$
  and $\nu_i$ is a member of ${}^{\alpha_i}2$ such that $\nu_j \char 94
  \nu_{\alpha_j,\gamma}\vartriangleleft \nu_i$.
\end{enumerate}
There is no problem to carry out the induction and $\langle \alpha_i:i <
\kappa\rangle$, $\eta := \bigcup\{\nu_i:i < \kappa\}$ are as required. 
\medskip

\noindent 2)\quad Should be clear. 
\end{PROOF}

\begin{remark}
\label{rem} 
The ideal $\id(\Cohen_\kappa)$ is an ideal of subsets of ${}^\kappa 2$. It
has a natural relative on ${}^\kappa\kappa$ --- the ideal of meagre subsets
of ${}^\kappa\kappa$. The two ideals are isomorphic in a suitable sense and
they have the same cardinal coefficients, cf \cite[Section 4]{MRSh:799}.  
\end{remark}

\begin{claim}
\label{p33}
\begin{enumerate}
\item $\add(\Cohen_\kappa) \le \gb_\kappa \le \non(\Cohen_\kappa)$. 
\item $\cov(\Cohen_\kappa) \le \gd_\kappa \le \cf(\Cohen_\kappa)$. 
\item $\cf(\Cohen_\kappa) = \max\{\gd_\kappa,\non(\Cohen_\kappa)\}$. 
\item $\add(\Cohen_\kappa) = \min\{\gb_\kappa,\cov(\Cohen_\kappa)\}$. 
\end{enumerate}
\end{claim}

\begin{PROOF}{\ref{p33}}
Our arguments are similar to those for $\kappa = \aleph_0$. 
\medskip

\noindent (1)\quad We will show that $\add(\Cohen_\kappa) \le \gb_\kappa$
(the inequality $\gb_\kappa \le \non(\Cohen_\kappa)$ should be clear;
remember \ref{rem}). Let $\mu=\gb_\kappa$ and let
$\{g_\alpha:\alpha<\mu\}\subseteq {}^\kappa\kappa$ exemplify this. For each
$\alpha<\mu$ let  
\[E_\alpha=\big\{\delta<\kappa:\delta\mbox{ is a limit ordinal and }
\big(\forall i<\delta \big) \big(g_\alpha(i)<\delta \big) \big\}.\] 
Let $\bar{\beta}_\alpha=\langle \beta_{\alpha,i}:i<\kappa\rangle$ list
$E_\alpha$ in the increasing order and let $\eta_\iota\in {}^\kappa 2$ be
constantly $\iota$ for $\iota=0,1$. Then $\{X_{\eta_\iota,
  \bar{\beta}_\alpha}: \iota<2\mbox{ and }\alpha<\mu\}$ is a collection of
$\mu$ many $\kappa$--meagre sets. Assume towards contradiction that their
union $A=\bigcup\{X_{\eta_\iota,\bar{\beta}_\alpha}: \iota<2\mbox{ and }
  \alpha<\mu\}$ is meagre. Hence, by \ref{p31}, there are $\eta\in {}^\kappa
  2$ and an increasing continous $\bar{\beta}\in {}^\kappa\kappa$ such that
  $A\subseteq X_{\eta,\bar{\beta}}$. Let  $g\in {}^\kappa\kappa$ be defined
  by $g(j)=\beta_{j+1}$. Then for some $\alpha<\mu$ we have $\neg(
g_\alpha\leq_{J^{\rm bd}_\kappa} g)$. If $\beta_j<\beta_{\alpha,i}\leq
\beta_{j+1}$, then $j \le \beta_j < \beta_{\alpha,i}$ and hence
$g_\alpha(j)<\beta_{\alpha,i}\leq\beta_{j+1}= g(j)$, so the set 
\[S=\{j<\kappa: (\beta_j,\beta_{j+1}]\cap \{\beta_{\alpha,i}:
i<\kappa\}=\emptyset\}\]
is of size $\kappa$. Choose a subset $S_0\subseteq S$ of size $\kappa$ such
that $j\in S_0\ \Rightarrow\ j+1\notin S_0$. Let $\nu\in
{}^\kappa 2$ be such that $\nu\rest [\beta_j, \beta_{j+1})=\eta\rest
[\beta_j,\beta_{j+1})$ for $j\in S_0$ and $\nu(i)=1$ whenever $i\notin
\bigcup\{[\beta_j,\beta_{j+1}): j\in S_0\}$. Then $\nu\in
X_{\eta_0,\bar{\beta}_\alpha}\setminus X_{\eta,\bar{\beta}}$, contradicting
$A\subseteq X_{\eta,\bar{\beta}}$.  
\medskip

\noindent
(2)\quad We will show that $\gd_\kappa \le \cf(\Cohen_\kappa)$. So let $\mu
= \cf(\Cohen_\kappa)$ and let $\langle A_\alpha:\alpha<\mu\rangle$ list a
cofinal subset of $\id(\Cohen_\kappa)$.  For each $\alpha < \mu$ we can find
$(\nu_\alpha,\bar\beta_\alpha)$ as in \ref{p31} such that $A_\alpha
\subseteq X_{\nu_\alpha,\bar\beta_\alpha}$. Let 
\[E_\alpha=\big\{\delta < \kappa:\delta\mbox{ is a limit ordinal such that 
}(\forall i)(\beta_{\alpha,i} < \delta \Leftrightarrow i < \delta)\big\},\]
it is a club of $\kappa$.  Towards contradiction assume $\gd_\kappa >
\mu$. Then there is a club $E$ of $\kappa$ such that $\sup(E_\alpha
\backslash E)=\kappa$ for all $\alpha < \mu$.  Let $\nu \in
{}^\kappa 2$ and the sequence $\bar\beta$ list $E$ in increasing
order and consider the $\kappa$--meagre set $X_{\nu,\bar\beta}$.  For some
$\alpha < \mu$ we have $X_{\nu,\bar\beta} \subseteq A_\alpha\subseteq
X_{\nu_\alpha,\bar{\beta}_\alpha}$. Easy contradiction to $\kappa = 
\sup(E_\alpha \backslash E)$.
\medskip

The inequality $\cov(\Cohen_\kappa)\le \gd_\kappa$ should be clear (remember
\ref{rem}). 
\medskip

\noindent
(3)\quad Recall that $\non(\Cohen_\kappa) \le \cf(\Cohen_\kappa)$ by
\ref{z27} and $\gd_\kappa \le \cf(\Cohen_\kappa)$ is proved in (2) above. So
we are left with: 
\[\cf(\Cohen_\kappa) \le \gd_\kappa + \non(\Cohen_\kappa).\]
Let $\mu = \non(\Cohen_\kappa)$; now
\begin{enumerate}
\item[$(\boxplus)$] there is $\{\varrho_\beta:\beta < \mu\} \subseteq
  {}^\kappa \kappa$ such that for every $\nu \in {}^\kappa \kappa$ for some
  $\beta<\mu$ we have $\sup\{i<\kappa:\varrho_\beta(i)=\nu(i)\}=\kappa$. 
\end{enumerate}
[Why?  For $\rho \in {}^\kappa 2$ let $\nu_\rho \in {}^\kappa \kappa$ be
such that for $i < \kappa$, $\nu_\rho(i)$ is $\gamma_{\rho,i}$ when
$\gamma_{\rho,i} < \kappa$ is the minimal $\gamma < \kappa$ such that,
if possible, $\rho(i + \gamma)=1$ (and if there is no such $\gamma$ then
it is 0).   Let $\eta_0\in {}^\kappa\kappa$ be constantly 0. Now if $\Lambda
\subseteq {}^\kappa 2$ is non-meagre of cardinality $\mu$ then  recalling
\ref{p31} the set $\{\nu_\rho:\rho \in \Lambda\}\cup\{\eta_0\} \subseteq
{}^\kappa \kappa$ is as required.] 

Let $\langle E_\gamma:\gamma < \gd_\kappa \rangle$ be a sequence of clubs of
$\kappa$ such that for any club $E$ of $\kappa$, for some
$\gamma$, $E_\gamma \subseteq E$, this is a variant of the definition of
$\gd_\kappa$.  For $\gamma < \gd_\kappa$ let $\bar\alpha_\gamma = \langle
\alpha_{\gamma,i}:i < \kappa\rangle$ list $E_\gamma \cup \{0\}$ in increasing
order. 

Let $\langle \rho_j:j < \kappa\rangle$ list $\bigcup\{{}^{[i,j)}2:i < j <
\kappa\}$ and for $(\beta,\gamma,\xi) \in \mu \times \gd_\kappa \times
\gd_\kappa$ let $A_{\beta,\gamma,\xi} =
X_{\varrho_{\beta,\gamma},\bar\alpha_\xi}$ from \ref{p31} where:
\begin{enumerate}
\item[$(\circledcirc)$]  for $\beta < \mu$ and $\gamma < \gd_\kappa$ let
$\varrho_{\beta,\gamma} \in {}^\kappa 2$ be such that
$\varrho_{\beta,\gamma} \rest [\alpha_{\gamma,i},\alpha_{\gamma,i+1})$
is equal to $\rho_{\varrho_\beta(i)}$ if $\rho_{\varrho_\beta(i)} \in
  {}^{[\alpha_{\gamma,i},\alpha_{\gamma,i+1})}2$ and is constantly
  zero otherwise. 
\end{enumerate}
So $\cA = \{A_{\beta,\gamma_1,\gamma_2}:\beta < \mu,\ \gamma_1 <
\gd_\kappa,\ \gamma_2 < \gd_\kappa\}$ is a subset of $\id(\Cohen_\kappa)$ 
and has cardinality $\le \mu + \gd + \gd = \max\{\mu,\gd\}$. Hence it
suffices to prove that $\cA$ is cofinal in $\id(\Cohen_\kappa)$.  To this
end let $A \in \id(\Cohen_\kappa)$, and let $\eta\in {}^\kappa 2$ and
increasing $\bar\alpha \in {}^\kappa \kappa$ be such that $A \subseteq
X_{\eta,\bar\alpha}$ (remember \ref{p31}).

Now, $E := \{\alpha < \kappa:\alpha\mbox{ is limit and }(\forall i <
\alpha)(\alpha_i < \alpha)\}$ is a club of $\kappa$, hence there is
$\gamma(1)<\gd_\kappa$ such that $E \supseteq E_{\gamma(1)}$. Then $A
\subseteq X_{\eta,\bar\alpha}\subseteq
X_{\eta,\bar\alpha_{\gamma(1)}}$.  Let $\varrho \in {}^\kappa \kappa$
be such that $i < \kappa \Rightarrow \eta \rest
[\alpha_{\gamma(1),i},\alpha_{\gamma(1),i+1}) = \rho_{\varrho(i)}$ and
let $\beta < \mu$ be such that $B =\{i < \kappa:\varrho(i) =
\varrho_\beta(i)\}$ is an unbounded subset of $\kappa$. Pick
$\gamma(2) < \gd$ such that 
\[E_{\gamma(2)}\subseteq \big\{\alpha\in E_{\gamma(1)}: \alpha\mbox{ is
  limit and }(\forall i<\alpha)(\alpha_{\gamma(1),i}<\alpha)
\big \}\] 
and $[\alpha_{\gamma(2),i}, \alpha_{\gamma(2),i+1}) \cap B
\ne \emptyset$ for every $i$.  Now clearly it suffices to prove:
\begin{enumerate}
\item[$(*)$]  $A \subseteq A_{\beta,\gamma(1),\gamma(2)}$.
\end{enumerate}
Why does $(*)$ hold?   Fix $\nu \in A$ and we shall prove that $\nu \in
A_{\beta,\gamma(1),\gamma(2)}$. By the choice of $(\eta,\bar\alpha)$ we know 
$\nu \in X_{\eta,\bar\alpha}$, so for $i < \kappa$ large enough $\nu \rest
[\alpha_i,\alpha_{i+1}) \nsubseteq \eta$. Let $i^*<\kappa$ be such that $\nu
\rest [\alpha_i,\alpha_{i+1}) \nsubseteq \eta$ for all $i\geq i^*$.

Let $i\in [i^*,\kappa)$. By the choice of $\gamma(2)$ we can fix $i_1 \in B$
such that $\alpha_{\gamma(2),i} \le i_1 < \alpha_{\gamma(2),i+1}$. Then, by
the definition of $B$, we have $\varrho(i_1)=\varrho_\beta(i_1)$ and by the
choice of $\varrho$ we have $\rho_{\varrho(i_1)} = \rho_{\varrho_\beta(i_1)}
= \eta \rest [\alpha_{\gamma(1),i_1},\alpha_{\gamma(1),i_1+1}) \in
{}^{[\alpha_{\gamma(1),i_1},\alpha_{\gamma(1),i_1+1})}2$. By the choice of
$\varrho_{\beta,\gamma(1)}$  in $(\circledcirc)$ we have 
\begin{enumerate}
\item[$(\boxdot)$]  $\varrho_{\beta,\gamma(1)} \rest
  [\alpha_{\gamma(1),i_1},\alpha_{\gamma(1),i_1+1}) = \eta \rest
  [\alpha_{\gamma(1),i_1},\alpha_{\gamma(1),i_1+1})$.
\end{enumerate}
Since $E_{\gamma(1)}\subseteq E$, we may find $i_2<\kappa$ such that
$[\alpha_{i_2},\alpha_{i_2+1}) \subseteq
[\alpha_{\gamma(1),i_1},\alpha_{\gamma(1),i_1+1})$. Then necessarily
$i_2\geq i_1\geq i^*$ and hence we have 
\[\nu\rest [\alpha_{i_2},\alpha_{i_2+1})\neq \eta\rest [\alpha_{i_2},
\alpha_{i_2+1})= \varrho_{\beta,\gamma(1)} \rest [\alpha_{i_2},
\alpha_{i_2+1}),\]  
and consequently $\nu\rest [\alpha_{\gamma(1),i_1},
\alpha_{\gamma(1),i_1+1}) \neq\varrho_{\beta,\gamma(1)}\rest
[\alpha_{\gamma(1),i_1},\alpha_{\gamma(1),i_1+1})$. Since
$E_{\gamma(2)}\subseteq \big\{\alpha<\kappa: \alpha\mbox{ is limit and
}(\forall j<\alpha)(\alpha_{\gamma(1),j}<\alpha)\big\}$, we know that  
\begin{enumerate}
\item[$(\boxtimes)$]  $[\alpha_{\gamma(1),i_1},
  \alpha_{\gamma(1),i_1+1}) \subseteq [\alpha_{\gamma(2),i},
  \alpha_{\gamma(2),i+1})$ and thus $\nu\rest [\alpha_{\gamma(2),i},
\alpha_{\gamma(2),i+1}) \neq\varrho_{\beta,\gamma(1)}\rest
[\alpha_{\gamma(2),i},\alpha_{\gamma(2),i+1})$.  
\end{enumerate}
Now we easily finish concluding that $\nu\in
X_{\varrho_{\beta,\gamma(1)},\bar\alpha_{\gamma(2)}}=  
A_{\beta,\gamma(1),\gamma(2)}$, as desired. 
\medskip

\noindent (4)\quad It follows from \ref{z27} and \ref{p33}(1) that 
$\mu:=\add(\Cohen_\kappa) \leq
\min\{\gb_\kappa,\cov(\Cohen_\kappa)\}$. In order to show the converse
inequality assume towards contradiction that $\mu<\min\{\gb_\kappa,
\cov(\Cohen_\kappa)\}$.  Suppose that $\cA=\{ A_\gamma: \gamma<\mu\}$
is a family of members of $\id(\Cohen_\kappa)$ (and we will argue that
$\bigcup\cA\in \id(\Cohen_\kappa)$). For $\gamma<\mu$ let
$(\eta_\gamma, \bar{\beta}_\gamma)$ be as in \ref{p31} and such that
$A_\gamma\subseteq X_{\eta_\gamma,\bar{\beta}_\gamma}$ and let  
\[E_\gamma=\{\alpha < \kappa:\alpha\mbox{ is limit and }(\forall i <
\alpha)(\beta_{\gamma,i} < \alpha)\}\]
(it is a club of $\kappa$). As $\mu<\gb_\kappa$ we may find an
increasing continuous sequence $\bar{\beta}=\langle
\beta_j:j<\kappa\rangle$ of ordinals below $\kappa$ such that for each
$\gamma$ and every sufficiently large $j$ we have $\beta_j\in
E_\gamma$.   Then $X_{\eta_\gamma,\bar{\beta}_\gamma}\subseteq
X_{\eta_\gamma, \bar{\beta}}$. Since $\mu<\cov(\Cohen_\kappa)$, by an
easy dualization of $(\boxplus)$ of (3), we have:  
\begin{enumerate}
\item[$(\boxplus)^*_\bot$] there is $\nu\in {}^\kappa 2$ such that 
for every $\gamma<\mu$ the set
\[Z_\gamma:=\big\{j<\kappa:\eta_\gamma\rest [\beta_j,
\beta_{j+1})=\nu\rest [\beta_j, \beta_{j+1})\big\}\]
is of size $\kappa$.
\end{enumerate}
Using $\mu<\gb_\kappa$ again, we may find an increasing sequence
$\bar{\alpha}$ such that
\[(\forall \gamma<\mu)(\exists i_0<\kappa)(\forall i>i_0)(Z_\gamma\cap
[\alpha_i, \alpha_{i+1})\neq \emptyset).\]
Then letting $\delta_i=\beta_{\alpha_i}$ (for $i<\kappa$) we will have
$X_{\eta_\gamma,\bar{\beta}}\subseteq X_{\nu,\bar{\delta}}$ for each
$\gamma$ and the desired conclusion easily follows.  
\end{PROOF}

\begin{claim}
\label{p39}
\begin{enumerate}
\item  If $\kappa = \sup(S_{\inc}^\kappa)$ \then \,
  $\cov(\Cohen_\kappa) \le \non(\bbQ_\kappa)$. 
\item  If $\kappa = \sup(S_{\inc}^\kappa)$ \then \,
  $\cov(\bbQ_\kappa) \le \non(\Cohen_\kappa)$. 
\end{enumerate}
\end{claim}

\begin{PROOF}{\ref{p39}}
Both follow by \ref{p80} and \ref{p24}.
\medskip

\noindent (1)\quad Let $A_0\in \id(\Cohen_\kappa)$,
$A_1 \in \id(\bbQ_\kappa)$ be a partition of ${}^\kappa 2$ (see
\ref{p24}). There is $X = \{\eta_\vare:\vare < \mu\} \subseteq {}^\kappa 2$
where $\mu = \non(\bbQ_\kappa)$ such that $X \notin \id(\bbQ_\kappa)$.  Now,
${}^\kappa 2$ with addition $\oplus$ modulo 2, coordinatewise, is an Abelian
Group and both ideals $\id(\Cohen_\kappa)$ and $\id(\bbQ_\kappa)$ are closed
under translations (see \ref{p80}). Thus $\{\eta_\vare\oplus A_0:\vare <
\mu\}$ is a family of $\le \mu$ members of $\id(\Cohen_\kappa)$ and it 
suffices to prove that $\bigcup\{\eta_\vare \oplus A_0:\vare < \mu\} =
{}^\kappa 2$.  So let $\nu \in {}^\kappa 2$. Since $\{\eta_\vare:\vare <\mu\} \notin
\id(\bbQ_\kappa)$, also $\{\eta_\vare\oplus\nu:\vare < \mu\} \notin
\id(\bbQ_\kappa)$ and hence it is not included in $A_1$. Thus for some
$\vare < \mu$, $\eta_\vare\oplus\nu \in A_0$, hence $\nu \in
\eta_\vare\oplus A_0$ as required.

\noindent
2) Same proof, just interchanging $A_0$ and $A_1$. 
\end{PROOF}

\begin{claim}
\label{u2}
If $\gb_\kappa >\cov(\Cohen_\kappa)$, \then \, $\cov(\bbQ_\kappa) \le
\cov(\Cohen_\kappa)$.    
\end{claim}

\begin{PROOF}{\ref{u2}}
If $\kappa>\sup(S^\kappa_\inc)$,  then $\cov(\bbQ_\kappa)=\cov(
\Cohen_\kappa)$. 

So suppose $\kappa$ is an inaccessible limit of inaccessibles and
$\gb_\kappa >\cov(\Cohen_\kappa)$. Assume towards contradiction that 
$\cov(\bbQ_\kappa) >\cov(\Cohen_\kappa):=\mu$. 

Using the assumption  $\gb_\kappa>\mu=\cov(\Cohen_\kappa)$ and Observation 
\ref{p31}  we can easily find an increasing sequence $\bar\theta = \langle 
\theta_\vare:\vare < \kappa\rangle$ and a family $\Upsilon \subseteq
\prod\limits_{\vare<\kappa} \theta_\vare$ such that   
\begin{enumerate}
\item[$(*)_1$]  $0 < \theta_\vare<\kappa$ for each $\vare<\kappa$,
  $|\Upsilon|=\mu$ and 
\item[$(*)_2$] $(\forall\nu \in\prod\limits_{\vare<\kappa}
  \theta_\vare)(\exists\rho \in \Upsilon)(\forall^\infty\vare<\kappa)(
  \rho(\vare)\neq\nu(\vare))$. 
\end{enumerate}
Next, by induction on $\vare < \kappa$, we choose inaccessible cardinals 
$\partial_\vare$ such that: 
\begin{enumerate}
\item[$(*)_3$] $\partial_\vare > \theta_\vare+\sum\limits_{\zeta <
    \vare} \partial_\zeta$\quad and \quad $\partial_\vare >
  \sup(\partial_\vare \cap  S^\kappa_\inc)$. 
\end{enumerate}
For each $\vare < \kappa$ fix a partition $\langle
S_{\vare,i}:i<\theta_\vare \rangle$ of $\partial_\vare$ into stationary sets   
and 
\begin{itemize}
\item for $0<i<\theta_\vare$ define $A_{\vare,i} = \big\{\eta \in
  {}^{\partial_\vare}2:$ the set $\{\alpha \in
  S_{\vare,i}:\eta(\alpha)=1\}$ is stationary but for each $j<i$ the set
  $\{\alpha\in S_{\vare,j}:\eta(\alpha)=1\}$ is not stationary$\big\}$, and  
\item let $A_{\vare,0} = {}^{\partial_\vare}2\setminus \bigcup\limits_{i \in
    [1,\theta_\vare)} A_{\vare,i}$. 
\end{itemize}
Note that $\langle A_{\vare,i}:i<\theta_\vare\rangle$
  is a partition of ${}^{\partial_\vare}2$ such that  
\begin{enumerate}
\item[$(*)_4$] $\nu \in   {}^{\partial_\vare >}2\ \Rightarrow\ \{\eta \in 
  A_{\vare,i}:\nu\vartriangleleft \eta\} \notin 
  \id(\Cohen_{\partial_\vare})$.
\end{enumerate}
Now, for $\rho \in \Upsilon$ and $\alpha < \kappa$ let 
\[\begin{array}{ll}
\cI_{\rho,\alpha} = \big\{p\in\bbQ_\kappa:&\ell g(\tr(p)) > \alpha\mbox{ and
for some } \vare<\kappa\\
&\alpha < \partial_\vare < \ell g(\tr(p))\ \wedge\
\tr(p) \rest \partial_\vare \in A_{\vare,\rho(\vare)}\big\}.
\end{array}\]   
It should be clear that each $\cI_{\rho,\alpha}$ is an open dense subset of
$\bbQ_\kappa$ (remember that $\partial_\vare > \sup(\partial_\vare \cap 
  S^\kappa_\inc)$ and use $(*)_4$). 

As we are assuming towards contradiction that $\cov(\bbQ_\kappa) > \mu$, the
set  $\bigcap\limits_{\rho \in \Upsilon} \, \bigcap\limits_{\alpha < \kappa}
\set(\cI_{\rho,\alpha})$ is not empty. Let $\eta \in \bigcap\limits_{\rho \in
  \Upsilon} \, \bigcap\limits_{\alpha < \kappa} \set(\cI_{\rho,\alpha})$ and
let $\nu \in \prod\limits_{\vare<\kappa} \theta_\vare$ be such that 
\[\vare < \kappa\quad \Rightarrow\quad \eta\rest \partial_\vare \in
A_{\vare,\nu(\vare)}.\]  
By the choice of $\eta$, for every $\rho\in\Upsilon$ we have
$\sup(\{\vare<\kappa: \eta\rest \partial_\vare \in
A_{\vare,\rho(\vare)} \})=\kappa$. Hence
\[(\forall \rho \in \Upsilon)(\exists^\infty \vare < \kappa)(\nu(\vare) =
\rho(\vare)),\] 
a clear contradiction with $(*)_2$.
\end{PROOF}

\begin{conclusion}
\label{u5}
Assume that either
\begin{enumerate}
\item[(a)] $\kappa>\sup(S^\kappa_\inc)$, or
\item[(b)] $\gb_\kappa>\cov(\Cohen_\kappa)$, or 
\item[(c)] there is a stationary non-reflecting set $S\subseteq
  S^\kappa_\pr$. 
\end{enumerate}
Then $\add(\bbQ_\kappa)\le \add(\Cohen_\kappa)$.
\end{conclusion}

\begin{PROOF}{\ref{u5}}
If $\kappa > \sup(\kappa \cap S^\kappa_\inc)$ then $\bbQ_\kappa$ is
equivalent to $\Cohen_\kappa$, and moreover $\id(\bbQ_\kappa) = 
\id(\Cohen_\kappa)$ (see \ref{p5}(1)) and so $\add(\bbQ_\kappa) =
\add(\Cohen_\kappa)$.  

Let us assume $\gb_\kappa > \cov(\Cohen_\kappa)$. Then, by
\ref{p33}(4), 
\begin{enumerate}
\item[$(\bullet)_1$]   $\add(\Cohen_\kappa)
= \cov(\Cohen_\kappa)$ 
\end{enumerate}
and by the Claim \ref{u2}
\begin{enumerate}
\item[$(\bullet)_2$]  $\cov(\bbQ_\kappa) \le \cov(\Cohen_\kappa)$.
\end{enumerate}
Hence (first inequality trivial, holds for any ideal, e.g. 
see \ref{z27}, the other two by $(\bullet)_2$ and $(\bullet)_1$)
\begin{enumerate}
\item[$(\bullet)$]  $\add(\bbQ_\kappa) \le \cov(\bbQ_\kappa) \le
  \cov(\Cohen_\kappa) = \add(\Cohen_\kappa)$. 
\end{enumerate}

Finally, if $\gb_\kappa\leq\cov(\Cohen_\kappa)$ but there is a
stationary non-reflecting set $S\subseteq S^\kappa_\pr$, then by
\ref{p33}(4) we have $\add(\Cohen_\kappa)=\gb_\kappa$ and by
\ref{q7-item2}(1)+\ref{p36hx}(1) we get 
\[\add(\bbQ_\kappa)\leq \add(\nst^\pr_\kappa)\leq \gb_\kappa
=\add(\Cohen_\kappa).\] 
So we are done.
\end{PROOF}

The following result is dual to \ref{u2}. 

\begin{claim}
\label{u8}
If $\gd_\kappa<\non(\Cohen_\kappa)$, \then \, $\non(\Cohen_\kappa) \le
\non(\bbQ_\kappa)$.   
\end{claim}

\begin{PROOF}{\ref{u8}}
If $\kappa > \sup(S_\inc^\kappa \cap \kappa)$ this holds trivially as in the 
proof of \ref{u5}, so from now on assume $\kappa = \sup(S_\inc^\kappa\cap
\kappa)$. For every $\bar\theta = \langle \theta_\vare:\vare <
\kappa\rangle$ with $1< \theta_\vare < \kappa$ we choose
$\bar\partial_{\bar\theta} = \langle\partial_{\bar{\theta},\vare}:\vare < 
\kappa\rangle$, $\bar{S}_{{\bar\theta},\vare} = \langle
S_{\bar{\theta},\vare,i}: i<\theta_\vare\rangle$, $\bar A_{\bar\theta,\vare} 
=  \langle A_{\bar\theta,\vare,i}:i <\theta_\vare\rangle$ as in the proof of  
Claim \ref{u2}. That is, $\bar\partial_{\bar\theta},
\bar{S}_{{\bar\theta},\vare}, \bar A_{\bar\theta,\vare}$ satisfy for
$\vare<\kappa$: 
\begin{enumerate}
\item[$(\oplus)_1$] $\partial_{\bar{\theta},\vare}<\kappa$ is an
  inaccessible cardinal such that $\partial_{\bar{\theta},\vare}>
  \theta_\vare+\sum\limits_{\zeta<\vare} \partial_{\bar{\theta},\zeta}$ and  
  $\partial_{\bar{\theta},\vare} >\sup(\partial_{\bar{\theta},\vare} \cap 
  S^\kappa_\inc)$,  
\item[$(\oplus)_2$] $\langle S_{\bar{\theta},\vare,i}:i<\theta_\vare
  \rangle$ is a partition of $\partial_\vare$ into stationary sets, and 
\item[$(\oplus)_3$] for $0<i<\theta_\vare$, $A_{\bar{\theta},\vare,i} =
  \big\{\eta \in {}^{\partial_\vare}2:$ the set $\{\alpha \in
  S_{\bar{\theta},\vare,i}:\eta(\alpha)=1\}$ is stationary but for each
  $j<i$ the set $\{\alpha\in S_{\bar{\theta},\vare,j}:\eta(\alpha)=1\}$ is
  not stationary$\big\}$, and   
\item[$(\oplus)_4$]  $A_{\bar{\theta},\vare,0} =
  {}^{\partial_\vare}2\setminus \bigcup\limits_{i \in [1,\theta_\vare)}
  A_{\bar{\theta},\vare,i}$.  
\end{enumerate}
A mapping ${}^\kappa 2\ni\eta \mapsto \nu_{\bar\theta,\eta} \in 
\prod\limits_{\vare<\kappa} \theta_\vare$ is defined by the condition 
$\eta \rest \partial_{\bar\theta,\vare} \in
A_{\bar\theta,\vare,\nu_{\bar\theta,\eta}(\vare)}$ for each $\vare<\kappa$. 

Choose $\Upsilon \subseteq {}^\kappa 2$, $\Upsilon \notin \id(\bbQ_\kappa)$,  
of cardinality $\non(\bbQ_\kappa)$.  For any $\bar\theta$ as above let 
$\Upsilon_{\bar\theta} =\{\nu_{\bar\theta,\eta}:\eta \in \Upsilon\}$. Then
clearly 
\begin{enumerate}
\item[$(\oplus)_5$]  $\Upsilon_{\bar\theta} \subseteq
  \prod\limits_{\vare<\kappa}\theta_\vare$ and $\Upsilon_{\bar\theta}$ has
  cardinality $\le \non(\bbQ_\kappa)$.
\end{enumerate}

Dually to arguments in \ref{u2} we will argue now that 
\begin{enumerate}
\item[$(\oplus)_6$] for every $\rho \in \prod\limits_{\vare<\kappa} \theta_\vare$,
  there is $\nu \in \Upsilon_{\bar\theta}$ such that $(\exists^\infty \vare
  < \kappa)(\rho(\vare) = \nu(\vare))$. 
\end{enumerate}
Why? Suppose $\rho \in \prod\limits_{\vare<\kappa} \theta_\vare$. For
$\alpha < \kappa$ let  
\[\begin{array}{ll}
\cI_\alpha = \big\{p\in\bbQ_\kappa:&\ell g(\tr(p)) > \alpha\mbox{ and
for some } \vare<\kappa\\
&\alpha < \partial_{\bar{\theta},\vare} < \ell g(\tr(p))\ \wedge\
\tr(p) \rest \partial_{\bar{\theta},\vare} \in
  A_{\bar{\theta},\vare,\rho(\vare)}\big\}. 
\end{array}\]   
Clearly, each $\cI_\alpha$ is an open dense subset of $\bbQ_\kappa$
(remember $\partial_{\bar{\theta},\vare}>\sup(\partial_{\bar{\theta},\vare} 
\cap S^\kappa_\inc)$).   Since $\Upsilon \notin \id(\bbQ_\kappa)$ we know
that $\Upsilon\cap\bigcap\limits_{\alpha < \kappa}
\set(\cI_\alpha)\neq\emptyset$. Let $\eta \in
\Upsilon\cap\bigcap\limits_{\alpha < \kappa} \set(\cI_\alpha)$. Then 
$(\exists^\infty \vare < \kappa)(\nu_{\bar{\theta},\eta}(\vare) 
=\rho(\vare))$. Thus $(\oplus)_6$ is justified. 
\medskip

Easily by definition of $\gd_\kappa$ we may choose a family
$\{\bar{\alpha}_\xi: \xi<\gd_\kappa\}$ such that 
\begin{enumerate}
\item[$(\oplus)_7$]
    \begin{enumerate}
\item[(a)] $\bar{\alpha}_\xi=\langle \alpha_{\xi,\vare}:
      \vare<\kappa\rangle$ is an increasing continuous sequence in $\kappa$
      (for each $\xi<\gd_\kappa$), and
\item[(b)] if $\langle\alpha_i:i<\kappa\rangle$ is an increasing sequence of
  ordinals below $\kappa$, then for some $\xi<\gd_\kappa$ we have 
\[(\forall^\infty \vare<\kappa)(\exists i<\kappa)(\alpha_{\xi,\vare} <
\alpha_i<\alpha_{i+1}<\alpha_{\xi,\vare+1}).\]  
    \end{enumerate}
\end{enumerate}
Now, for each $\xi<\gd_\kappa$ let $\bar{\theta}_\xi=\langle \theta_{\xi,\vare}:
\vare<\kappa\rangle$, where
$\theta_{\xi,\vare}=|{}^{[\alpha_{\xi,\vare},\alpha_{\xi,\vare+1})}
2|$. Also, for each $\xi, \vare$ fix a bijection
$\pi_{\xi,\vare}:\theta_{\xi,\vare} \longrightarrow
{}^{[\alpha_{\xi,\vare},\alpha_{\xi,\vare+1})} 2$ and for $\nu\in
\prod\limits_{\vare<\kappa} \theta_{\xi,\vare}$ (for $\xi<\gd_\kappa$) set 
$x_{\xi,\nu}=\bigcup\limits_{\vare<\kappa} \pi_{\xi,\vare}(\nu(\vare))\in
{}^\kappa 2$. Consider the set 
\[\cX=\big\{x_{\xi,\nu}: \xi<\gd_\kappa \ \wedge\  \nu\in
\Upsilon_{\bar{\theta}_\xi}\big\}.\] 
We claim that 
\begin{enumerate}
\item[$(\oplus)_8$] $\cX\notin\id(\Cohen_\kappa)$.  
\end{enumerate}
If not, then for some $\eta\in {}^\kappa 2$ and an increasing continuous
sequence $\bar{\alpha}=\langle\alpha_i:i<\kappa\rangle\subseteq \kappa$ we
have $\cX\subseteq X_{\eta,\bar{\alpha}}$.  Let $\xi<\gd_\kappa$ be given by
$(\oplus)_7({\rm b})$ for $\bar{\alpha}$ and let $\rho^*\in
\prod\limits_{\vare<\kappa} \theta_{\xi,\vare}$ be such that
$\pi_{\xi,\vare}(\rho^*(\vare))= \eta\rest [\alpha_{\xi,\vare},
\alpha_{\xi,\vare+1})$ for each $\vare<\kappa$. It follows from $(\oplus)_6$
that for some $\nu\in\Upsilon_{\bar{\theta}_\xi}$ we have $(\exists^\infty \vare
  < \kappa)(\rho^*(\vare) = \nu(\vare))$.  This implies that
  $\big(\exists^\infty \vare<\kappa\big)\big(x_{\xi,\nu}\rest
  [\alpha_{\xi,\vare},\alpha_{\xi,\vare+1})=\eta\rest [\alpha_{\xi,\vare},
  \alpha_{\xi,\vare+1})\big)$ and hence (remembering the choice of 
  $\xi$) we get $\big(\exists^\infty i<\kappa\big)\big(x_{\xi,\nu}\rest
  [\alpha_i,\alpha_{i+1})=\eta\rest
  [\alpha_i,\alpha_{i+1})\big)$. Consequently $x_{\xi,\nu}\notin
  X_{\eta,\bar{\alpha}}$, a contradiction.  
\medskip

It follows from $(\oplus)_8$ that $\gd_\kappa< \non(\Cohen_\kappa) \le
|\cX|\leq \non(\bbQ_\kappa)+\gd_\kappa$ and therefore $\non(\Cohen_\kappa)
\leq \non(\bbQ_\kappa)$. 
\end{PROOF}

\begin{conclusion}
\label{u11}
Assume that either
\begin{enumerate}
\item[(a)] $\kappa>\sup(S^\kappa_\inc)$, or
\item[(b)] $\gd_\kappa<\non(\Cohen_\kappa)$, or 
\item[(c)] there is a stationary non-reflecting set $S\subseteq
  S^\kappa_\pr$. 
\end{enumerate}
Then $\cf(\Cohen_\kappa) \le \cf(\bbQ_\kappa)$.  
\end{conclusion}

\begin{PROOF}{\ref{u11}}
The proof is similar to the proof of \ref{u5}.
\medskip

If $\kappa>\sup(S^\kappa_\inc)$ then $\id(\bbQ_\kappa)=\id( \Cohen_\kappa)$
and $\cf(\bbQ_\kappa)=\cf(\Cohen_\kappa)$.
\medskip

If $\gd_\kappa<\non(\Cohen_\kappa)$, then it follows from \ref{p33}(3) that
$\cf(\Cohen_\kappa)= \non(\Cohen_\kappa)$. Also, by \ref{u8} and \ref{z27}(b), we have
$\non(\Cohen_\kappa)\leq \non(\bbQ_\kappa) \leq \cf(\bbQ_\kappa)$. Together
$\cf(\Cohen_\kappa) \le \cf(\bbQ_\kappa)$ (under present assumptions).
\medskip

If $\gd_\kappa\geq\non(\Cohen_\kappa)$, but there is a non-reflecting
stationary subset of $S^\kappa_\pr$, then we use \ref{q7-item2}(3) to get
$\cf(\nst^\pr_\kappa)\geq \gd_\kappa$. Now. \ref{p33}(3) implies
$\cf(\Cohen_\kappa)=\gd_\kappa$ and \ref{p36hx}(2) gives
$\cf(\bbQ_\kappa)\geq \cf(\nst^\pr_\nst)$. Together we conclude
$\cf(\bbQ_\kappa)\geq \cf(\Cohen_\kappa)$, as desired. 
\end{PROOF}

Now we may summarize the results of this section in the form of
diagrams. 

\begin{theorem}
\label{u14}
Assume that $\kappa$ is an inaccessible cardinal and
$\kappa=\sup(S^\kappa_\inc)$. Then the inequalities represented by arrows 
in the following diagram hold true:
\[\hspace{-1cm}
\begin{array}{ccccccccccc}
&&{\rm cov}(\bbQ_\kappa)&\rightarrow &{\rm non}(\Cohen_\kappa)
    &\rightarrow&{\rm cf}(\Cohen_\kappa) & &
{\rm cf}(\bbQ_\kappa) &\rightarrow&2^\kappa\\
&&\Bigg\uparrow& &\uparrow &&\uparrow&&\Bigg\uparrow \\ 
&&\big|& & {\mathfrak b}_\kappa &\rightarrow &{\mathfrak d}_\kappa&&\big|
    \\ 
&&\Bigg |& &\uparrow &&\uparrow&&\Bigg | \\ 
\kappa^+&\rightarrow&{\rm add}(\bbQ_\kappa)&  &{\rm add}(\Cohen_\kappa)
&\rightarrow&{\rm cov}(\Cohen_\kappa)&\rightarrow&
{\rm non}(\bbQ_\kappa)&&
\end{array}\] 
plus the dependencies
\begin{itemize}
\item $\add(\Cohen_\kappa) =\min\{\cov(\Cohen_\kappa),\gb_\kappa\}$, 
\item $\cf(\Cohen_\kappa)= \max\{\non(\Cohen_\kappa),\gd_\kappa\}$,
\item $\cov(\bbQ_\kappa)\leq \non(\bbQ_\kappa)$ (see \ref{u24}(3)). 
\end{itemize}
Moreover, we may add that one of the following four diagrams holds (where
each arrow $\rightarrow$ represents the inequality $\leq$ and
$\uparrow{\neq}$ represents the strict inequality $<$). 
\medskip

\noindent
\underline{Case 1}:
\bigskip
\[\hspace{-1cm}
\begin{array}{ccccccccccc}
&&&&&&&&{\rm cf}(\bbQ_\kappa) &\rightarrow & 2^\kappa\\
&&&&&&&&\Big\uparrow&&\\
&&&&{\rm non}(\Cohen_\kappa)
    &=&{\rm cf}(\Cohen_\kappa) &\rightarrow&
{\rm non}(\bbQ_\kappa) &&\\
&&&&\uparrow &&\uparrow{\neq}&&&&\\ 
&&&& {\mathfrak b}_\kappa &\rightarrow &{\mathfrak d}_\kappa&&\\ 
&&&&\uparrow{\neq} &&\uparrow&&&&\\ 
&&{\rm cov}(\bbQ_\kappa)&\rightarrow &
{\rm add}(\Cohen_\kappa)&=&{\rm
  cov}(\Cohen_\kappa)&&&&\\
&&\Big\uparrow &&&&&&&&\\
\kappa^+&\rightarrow&{\rm add}(\bbQ_\kappa) &&&&&&&&\\
\end{array}\] 
\medskip

\noindent
\underline{Case 2}:
\bigskip
\[\hspace{-1cm}
\begin{array}{ccccccccccc}
&&&&&&&&{\rm cf}(\bbQ_\kappa) &\rightarrow & 2^\kappa\\
&&&&&&&&\Big\uparrow&&\\
&&{\rm cov}(\bbQ_\kappa)
&\rightarrow&{\rm non}(\Cohen_\kappa)
    &=&{\rm cf}(\Cohen_\kappa) &\rightarrow&
{\rm non}(\bbQ_\kappa) &&\\
&&\big\uparrow&&\uparrow &&\uparrow{\neq}&&&&\\ 
&&\big|&& {\mathfrak b}_\kappa &\rightarrow &{\mathfrak d}_\kappa&&\\ 
&&\big|&&\parallel &&\uparrow&&&&\\ 
\kappa^+&\rightarrow&{\rm add}(\bbQ_\kappa) & &
{\rm add}(\Cohen_\kappa)&\rightarrow
&{\rm cov}(\Cohen_\kappa)&&&\\
\end{array}\] 
\medskip

\noindent
\underline{Case 3}:
\bigskip

\[\hspace{-1cm}
\begin{array}{ccccccccccc}
&&&&{\rm non}(\Cohen_\kappa)
    &\rightarrow&{\rm cf}(\Cohen_\kappa) & &
{\rm cf}(\bbQ_\kappa) &\rightarrow & 2^\kappa\\
&&&&\uparrow &&\parallel&&\big\uparrow&&\\ 
&&&& {\mathfrak b}_\kappa &\rightarrow &{\mathfrak d}_\kappa&&\big|&\\ 
&&&&\uparrow{\neq} &&\uparrow&&\big|&&\\ 
&&{\rm cov}(\bbQ_\kappa) &\rightarrow &
{\rm add}(\Cohen_\kappa)&=&{\rm
  cov}(\Cohen_\kappa) &\rightarrow &
{\rm non}(\bbQ_\kappa) &&\\
&&\Big\uparrow &&&&&&&&\\
\kappa^+&\rightarrow&{\rm add}(\bbQ_\kappa) &&&&&&&&\\
\end{array}\] 
\medskip

\noindent
\underline{Case 4}:
\bigskip
\[\hspace{-1cm}
\begin{array}{ccccccccccc}
&&{\rm cov}(\bbQ_\kappa)&\rightarrow &{\rm non}(\Cohen_\kappa)
    &\rightarrow&{\rm cf}(\Cohen_\kappa) & &
{\rm cf}(\bbQ_\kappa) &\rightarrow&2^\kappa\\
&&\big\uparrow& &\uparrow&&\parallel &&\big\uparrow \\ 
&&\big|& & {\mathfrak b}_\kappa &\rightarrow &{\mathfrak d}_\kappa&&\big|
    \\ 
&&\big|& &\parallel &&\uparrow&&\big| \\ 
\kappa^+&\rightarrow&{\rm add}(\bbQ_\kappa)& &{\rm add}(\Cohen_\kappa)
&\rightarrow&{\rm cov}(\Cohen_\kappa)&\rightarrow&
{\rm non}(\bbQ_\kappa)&&
\end{array}\] 
\end{theorem}
\bigskip

\begin{remark}
  \begin{enumerate}
\item In a later work we prove that $\add(\nst^\pr_\kappa)\leq \gd_\kappa$
    and $\gb_\kappa\leq \cf(\nst^\pr_\kappa)$. Consequently, by \ref{p36hx}, 
    $\add(\bbQ_\kappa)\leq \gd_\kappa$ and
    $\cf(\bbQ_\kappa)\geq\gb_\kappa$. 
\item Remember that by \ref{u5} and \ref{u11}, if
  $\kappa>\sup(S^\kappa_\inc)$ or there is a stationary non-reflecting set
  $S\subseteq S^\kappa_\pr$, then $\add(\bbQ_\kappa)\leq
  \add(\Cohen_\kappa)$ and  $\cf(\bbQ_\kappa)\geq \cf(\Cohen_\kappa)$. 
  \end{enumerate}
\end{remark}

\section{$\bbQ_\kappa$ vs $\Cohen_\kappa$}
\label{VS}

\subsection{Effect on the ground model} 
\label{effect}

\begin{claim}
\label{p15}
If $\kappa$ is an inaccessible limit of inaccessibles, \then \, in
$\bold V^{\bbQ_\kappa}$ the set $({}^\kappa 2)^{\bold V}$ is
$\kappa$--meagre.
\end{claim}

\begin{remark}
\label{p16x}
1) The dual is \ref{p28}.

\noindent
2) The assumption is necessary by \ref{p5}.
\end{remark}

\begin{PROOF}{\ref{p15}}
Let $\langle \partial_i:i < \kappa\rangle$ list in increasing order
the (strongly) inaccessible cardinals below $\kappa$. We claim that 

\begin{equation*}
\begin{array}{clcr}
\Vdash_{\bbQ_\kappa} &``\text{if } \nu \in ({}^\kappa 2)^{\bold V}
  \text{ \then \, for every } i < \kappa \text{ large enough } 
  \name\eta \rest (\partial_i +1,\partial_{i+1}) \nsubseteq \nu, \\
  &\text{ moreover }\quad \alpha < \partial_{i+1} \Rightarrow \name\eta \rest
  (\alpha,\partial_{i+1}) \nsubseteq \nu".
\end{array}
\end{equation*}

\noindent
This clearly suffices by \ref{p31}(2).  Let $p \in \bbQ_\kappa$ and we shall
fix $\nu \in ({}^\kappa 2)^{\bold V}$ and we shall find $q$ and $i_* < 
\kappa$ such that $p \le_{\bbQ_\kappa} q$ and $q \Vdash$ ``if $i >
i_*$ then $\name\eta \rest (\partial_i +1,\partial_{i+1}) \nsubseteq
\nu"$.

Let $i_*$ be such that $\ell g(\tr(p))<\partial_{i_*}$ and let
$(\varrho,S_1,\bar\Lambda)$ be a witness for $p \in \bbQ_\kappa$.  Now let
$S_2 = \{\partial_{i+1}:i > i_*\}$ and if $\partial = \partial_{i+1} \in
S_2$ and $\alpha \in (\partial_i,\partial_{i+1})$ \then \, we let 
\[\cI_{\partial,\alpha} = \big\{r \in \bbQ_\partial:\ell g(\tr(r)) >
\alpha\mbox{ and }\tr(r) \rest [\alpha,\ell g(\tr(r)) \nsubseteq \nu\big\}.\]
Clearly, $\cI_{\partial,\alpha}$ is a dense open subset of $\bbQ_\partial$.
Now let $S' = S_1 \cup S_2$ and note that $S_2$ is nowhere stationary, so
$S'$ is too. Next, for $\partial \in S'$ put
\[\Lambda'_\partial=\left\{
\begin{array}{ll}
\Lambda_\partial&\mbox{ if }\partial\in S_1 \backslash S_2,\\
\Lambda_\partial \cup\{\cI_{\partial,\alpha}:\alpha \in
(\partial_i,\partial_{i+1})\} &\mbox{ if }\partial = \partial_{i+1} \in S_1
\cap S_2,\\
\{\cI_{\partial,\alpha}:\alpha \in (\partial_i,\partial_{i+1})\}&\mbox{ if }
\partial = \partial_{i+1} \in S_2 \backslash S_1,
\end{array}\right.\]
and let $\bar\Lambda' = \langle \Lambda'_\partial:\partial \in S'\rangle$. 
Easily the triple $(\tr(p),S',\bar\Lambda')$ is a witness for some $q \in
\bbQ_\kappa$ and this $q$ is as required.
\end{PROOF}

\begin{claim}
\label{p28}
If $\kappa$ is inaccessible limit of inaccessibles and $\bold V_1$ is an
extension of $\bold V$ (e.g. a forcing extension) \then \, $\bold V_1
\models ``({}^\kappa 2)^{\bold V} \in \id(\bbQ_\kappa)$'' provided that at
least one of the following holds (each implying $\kappa$ is still an
inaccessible limit of inaccessibles in $\bold V_1$):
\begin{enumerate}
\item[(a)]  $\bold V_1 = \bold V^{\Cohen(\kappa)}$, see Definition
  \ref{z4}(2). 
\item[(b)]  In $\bold V_1$, $\kappa$ is still inaccessible and there are
  sequences $\bar\eta = \langle \eta_\partial:\partial \in S\rangle$,
  $\bar\alpha = \langle \alpha_\partial:\partial \in S\rangle$ such that 
\begin{enumerate}
\item[$(\alpha)$]  $S \subseteq \kappa$ is unbounded in $\kappa$,
\item[$(\beta)$]  $\partial \in S\quad \Rightarrow\quad \alpha_\partial =
  \sup(S\cap \partial) < \partial$, 
\item[$(\gamma)$]  $S$ is a set of inaccessibles (in $\bold V_1$ hence in
  $\bold V$), 
\item[$(\delta)$]  $\eta_\partial \in {}^\partial 2$, really just
  $\eta_\partial \rest (\alpha_\partial,\partial)$ matter, 
\item[$(\varepsilon)$]  if $\eta \in ({}^\kappa 2)^{\bold V}$ \then \, for
  unboundedly many $\partial \in S$ we have $\eta \rest
  (\alpha_\partial,\partial) \subseteq \eta_\partial$. 
\end{enumerate}
\item[(c)]  In $\bold V_1$, $\kappa$ is still inaccessible limit of
  inaccessibles but $\cH(\kappa)^{\bold V} \ne \cH(\kappa)^{\bold V_1}$.
\item[(d)]  Like clause (b) but
\begin{enumerate}
\item[$(\beta)'$]  $S$ is unbounded nowhere stationary in $\kappa$, 
\item[$(\delta)'$]  $\bar\Lambda = \langle \Lambda_\partial:\partial\in
  S\rangle$, $\Lambda_\partial$ a set $\le \partial$ dense subset of
  $\bbQ_\partial$, 
\item[$(\varepsilon)'$]  if $\eta \in ({}^\kappa 2)^{\bold V}$ \then
\, for unboundedly many $\partial \in S$, $\eta \rest \partial$ does not 
fulfill $\Lambda_\partial$.
\end{enumerate}
\end{enumerate}
\end{claim}

\begin{remark}
\label{p29}
Of course, if $\kappa$ is inaccessible not limit of inaccessibles then
the conclusion of \ref{p28} fails because $\bbQ_\kappa$ is equivalent
to Cohen$_\kappa$, see \ref{p5}.
\end{remark}

\begin{PROOF}{\ref{p28}}

\noindent
\underline{\bf Clause (a)}:\qquad It suffices to prove that the assumptions of
(b) holds.\\  
Clearly the forcing preserves inaccessibility. Let $\name\eta \in {}^\kappa
2$ be the name of the $\kappa$-Cohen real and let:
\begin{enumerate}
\item[$\bullet$]  $S_1 = \{\partial < \kappa:\partial$ inaccessible in
  $\bold V_1$ or $\bold V$, those are equivalent$\}$,
\item[$\bullet$]  $S = \{\partial \in S_1:\partial > \sup(S_1 \cap
\partial)\}$,
\item[$\bullet$]  $\name\eta_\partial = \name\eta \rest \partial$,
\item[$\bullet$]  $\alpha_\partial = \sup(S_1 \cap \partial)$ for
  $\partial \in S$.
\end{enumerate}
Clearly clauses $(\alpha),(\gamma)$ of (b) are satisfied by $S_1$ and by $S$
and clause $(\beta)$ is satisfied by the $\alpha_\partial$'s and $S$.  Also
recalling $\name\eta \in {}^\kappa 2$, it is the $\kappa$--Cohen real, the
derived sequence $\langle \name\eta_\partial:\partial \in S\rangle$
satisfies clause $(\delta)$ by our choice above. Lastly, clause
$(\varepsilon)$ holds as Cohen$_\kappa = ({}^{\kappa>}2,\vartriangleleft)$,
so all the assumptions of clause (b) hold indeed.  
\medskip

\noindent
\underline{\bf Clause (b)}:\qquad We work in $\bold V_1$.\\ 
For $\alpha<\partial \in S$ let 
\[\cI^*_{\partial,\alpha} = \big\{p\in\bbQ_\partial:\mbox{for some $\beta$ we
  have }\alpha<\beta<\partial,\ \beta<\ell g(\tr(p))\mbox{ and }\tr(p)\rest
(\alpha,\beta)\nsubseteq\eta_\partial\big\}.\] 
Easily $\cI^*_{\partial,\alpha}$ is a dense open subset of $\bbQ_\partial$
and let 
\[\cI = \big\{p \in \bbQ_\kappa:\mbox{for some $\gamma <\kappa$ we have } S
\backslash \gamma \subseteq S_p \mbox{ and }\partial \in S \backslash \gamma  
\ \Rightarrow\ \cI^*_{\partial,\alpha_\partial} \in
\Lambda_{p,\partial}\big\}.\]  
Clearly $\cI$ is a dense open subset of $\bbQ_\kappa$ and $p \in \cI
\ \Rightarrow \ \lim_\kappa(p) \cap ({}^\kappa 2)^{\bold V} = \emptyset$, so
${\bold V}\cap {}^\kappa 2\in\id_2(\bbQ_\kappa)$ and we are done (remember
\ref{p4d}(5)).  
\medskip

\noindent
\underline{\bf Clause (c)}:\qquad Let $S_1$ be the set of inaccessibles in
$\bold V_1$ which are $< \kappa$.  Let $\alpha < \kappa$ and $\nu$ be such
that $\nu \in ({}^\alpha 2)^{\bold V_1}$ but $\nu \notin ({}^\alpha
2)^{\bold V}$. 

Now let
\begin{enumerate}
\item[$\bullet$]  $S = \{\partial \in S_1:\partial > \alpha$ and
$\partial > \sup(S_1 \cap \partial)\}$,
\item[$\bullet$] $\cI_\partial = \{p \in \bbQ_\partial$: for some $\beta$ we
  have $\beta + \alpha \le \ell g(\tr(p))$ and $\langle \tr(p)(\beta+i): i <
  \alpha\rangle = \nu\}$ for $\partial \in S$,
\item[$\bullet$]  $\Lambda_\partial = \{\cI_\partial\}$ for $\partial
\in S$.
\end{enumerate}
Why is $\cI_\partial$ a dense subset of $\bbQ_\partial$ for every $\partial
\in S$?  Let $p_1 \in \bbQ_\partial$ and we shall find $p_2$ such that $p_1
\le_{\bbQ_\partial} p_2 \in \cI_\partial$.  Let $p_2 \in \bbQ_\partial$ be
such that $p_1 \le_{\bbQ_\partial} p_2$ and $\ell g(\tr(p_2)) \ge \alpha +
\sup\{\theta:\theta < \partial$ is inaccessible$\}$.  (Why such $p_2$
exists?  As $\partial \in S$ implies that $\partial$ is (strictly) above the
ordinal on the right).  But this implies $S_{p_2} = \emptyset$ hence there
is $p_3$ such that $p_2 \le_{\bbQ_\partial} p_3$ and $(\tr(p_3))(\alpha +
\ell g(\tr(p_2) + i) = \nu(i)$ for $i < \alpha$ hence $p_3 \in
\cI_\partial$.  Hence the assumptions of clause (d) hold, so the result
follows.  \medskip

\noindent
\underline{\bf Clause (d)}:\qquad  Like the proof of clause (b).
\end{PROOF}

\begin{remark}
\label{p30}
If $\kappa$ is inaccessible not limit of inaccessibles and $\bold V_1$
extends $\bold V$ and $\cH(\kappa)^{\bold V_1} \ne \cH(\kappa)^{\bold V}$
then $({}^\kappa 2)^{\bold V} \in \id(\Cohen_\kappa)^{\bold V_1}$ and 
$({}^\kappa 2)^{\bold V} \in \id(\bbQ_\kappa)^{\bold V_1}$.
\end{remark}

\begin{claim}
\label{u24}
Assume $\kappa$ is inaccessible limit of inaccessibles. Then
\begin{enumerate}
\item $\Vdash_{\bbQ_\kappa}{\bold V}\cap {}^\kappa 2\in \id(\bbQ_\kappa)$. 
\item $\bbQ_\kappa$ is asymmetric; that is, if $\bold V_1 \subseteq \bold 
V_2 \subseteq \bold V_3$, $\eta_\ell \in ({}^\kappa 2)^{\bold V_{\ell +1}}$ is 
$(\bbQ_\kappa,\name\eta_\kappa)$--generic over $\bold V_\ell$, for 
$\ell=1,2$, \then \, $\eta_1$ is not $(\bbQ_\kappa,
\name\eta_\kappa)$--generic over $\bold V_1[\eta_2]$. 
\item $\cov(\bbQ_\kappa)\leq \non(\bbQ_\kappa)$.
\end{enumerate}
\end{claim}

\begin{PROOF}{\ref{u24}}
(1)\quad Let $\langle \partial_\vare:\vare < \kappa\rangle$ list
$S_{\inc}^\kappa$ in increasing order and let $S = \{\partial_{\vare
  +1}:\vare<\kappa\}$.  For $\eta \in {}^\kappa 2$ and $\partial \in S$ let
$\Lambda_{\eta,\partial}$ be a family of $\leq\partial$ dense subsets of
$\bbQ_\partial$ such that 
\[\set(\Lambda_{\eta,\partial})=\big\{\rho \in {}^\partial 2: \mbox{ for
  arbitrarily large }\zeta < \partial\mbox{ we  have }\rho(\zeta) \ne
\eta(\partial + \zeta)\big\}.\]   
Define 
\[A_\eta=\big\{\nu \in {}^\kappa 2:(\forall^\infty\partial \in S)(\nu
\rest \partial \in \set(\Lambda_{\eta,\partial}))\big\}.\]
 Clearly, the set $A_\eta$ is $\kappa$--Borel. Note that  
\[\big\{p\in\bbQ_\kappa: (S\setminus \ell g(\tr(p)))\subseteq S_p\ \wedge\
(\forall \partial\in S)(\ell g(\tr(p)<\partial \ \Rightarrow\
\Lambda_{\eta,\partial}\subseteq \Lambda_{p,\partial})\big\}\]
is an open dense subset of $\bbQ_\kappa$. Hence,
\begin{enumerate}
\item[$(*)_1$]  for every $\eta \in {}^\kappa 2$ we have ${}^\kappa
  2\setminus A_\eta\in \id(\bbQ_\kappa)$.
\end{enumerate}
We are going to argue that 
\begin{enumerate}
\item[$(*)_2$]  $\Vdash_{\bbQ_\kappa} {\bold V}\cap A_{\name
    \eta}=\emptyset$. 
\end{enumerate}
So let $\nu\in {}^\kappa2$. Suppose that $p \in \bbQ_\kappa$ and $\xi < 
\kappa$.  Choose $\partial\in S$ such that $\partial > \xi,\ell g(\tr(p))$
and then pick $\rho \in p\cap {}^\partial 2$. Now
$\varrho=\rho \char 94 (\nu \rest \partial) \in p$ and 
\[p^{[\varrho]} \Vdash_{\bbQ_\kappa} \nu\rest\partial\notin 
\set(\Lambda_{\name{\eta},\partial}).\] 
By standard density arguments we conclude that 
\[\Vdash_{\bbQ_\kappa} \big(\exists^\infty \partial\in S \big) \big
(\nu\rest \partial\notin\set(\Lambda_{\name{\eta},\partial}) \big) \]    
and thus $\Vdash_{\bbQ_\kappa} \nu\notin A_{\name{\eta}}$. 
\medskip

\noindent (2)\quad Assume that $\eta_1$ is
$(\bbQ_\kappa,\name\eta)$--generic over $\bold V$ and $\name\eta_2$ is
$(\bbQ_\kappa,\name\eta)$--generic over $\bold V[\eta_1]$. 

It follows from $(*)_2$ of part (1) that 
\begin{enumerate}
\item[$(*)_3$]  ${\bold V}[\eta_1,\eta_2]\models\eta_1\notin A_{\eta_2}$. 
\end{enumerate}
Therefore, by $(*)_1$, $\eta_1$ is not $(\bbQ_\kappa,\name\eta)$--generic
over $\bold V[\eta_2]$.  
\medskip

\noindent (3)\quad  Let $S,\Lambda_{\eta,\partial}$ and $A_\eta$ for
$\partial\in S$, $\eta\in {}^\kappa 2$ be defined as in \ref{u24}(1). Then
${}^\kappa 2\setminus A_\eta\in \id(\bbQ_\kappa)$. For $\nu\in {}^\kappa 2$
let $A^\nu=\{\eta\in {}^\kappa 2: \nu\in A_\eta\}$. The argument in the end
of part (1) shows that for each $\xi<\kappa$ the set 
\[\big\{p\in\bbQ_\kappa: \big(\exists \partial\in S\setminus\xi \big)
\big (\forall\eta\in {\lim}_\kappa(p) \big) \big (\nu\rest\partial \notin
\set(\Lambda_{\eta,\partial}) \big) \big\}\]
is open dense in $\bbQ_\kappa$. Hence $A^\nu\in\id(\bbQ_\kappa)$. 
 
Now suppose that $X\subseteq {}^\kappa 2$ is such that $X\notin
\id(\bbQ_\kappa)$. We claim that then 
\[\bigcup\{{}^\kappa 2\setminus A_\eta:\eta\in X\}={}^\kappa 2.\] 
So suppose $\nu\in {}^\kappa 2$. Let $\eta\in X\setminus A^\nu\neq
\emptyset$. By the definition this implies $\nu\notin A_\eta$ and we are
done. 

In \cite{Sh:F1580} we note that generally for a nice enough $\bold i$
asymmetry implies $\cov({\bold i})\leq \non({\bold i})$. 
\end{PROOF}

\subsection{When does $\bbQ_\kappa$ add a Cohen real?}
\label{addCoh}

\begin{definition}
\label{p53}
Let $S_{\awc}$ be the class of inaccessible $\kappa$ such that 
($\awc$ stands for ``anti weakly compact") in
$\bold V^{\bbQ_\kappa}$ there is a Cohen $\kappa$--real over $\bold V$;
equivalently:
\begin{enumerate}
\item[$(*)$]  there is a sequence $\langle \cI_\alpha:\alpha <
\kappa\rangle$, $\cI_\alpha \subseteq \bbQ_\kappa$ such
that\footnote{so $\cI_\alpha$ is not necessarily dense and not
  necessarily open; \wilog \, $\cI_\alpha$ is an antichain (but not
  necessarily maximal).  Of course the $\varrho$ later is not
  necessarily constant.} 
 for every $p\in \bbQ_\kappa$ there is $\alpha < \kappa$ such that:\\
 for every $\beta \in (\alpha,\kappa)$ and $\varrho \in
 {}^{[\alpha,\beta)}2$ there is $q$ such that
\begin{enumerate}
\item[$\bullet$]  $p \le_{\bbQ_\kappa} q$,
\item[$\bullet$]  if $\gamma \in [\alpha,\beta)$ and
$\varrho(\gamma)=1$ then $q$ is above some member of $\cI_\gamma$,
\item[$\bullet$]  if $\gamma \in [\alpha,\beta)$ and $\varrho(\gamma)
= 0$ then $q$ is incompatible with every member of $\cI_\gamma$.
\end{enumerate}
\end{enumerate}
\end{definition}

\begin{claim}
\label{p56}
If $\kappa$ is (strongly inaccessible but) not Mahlo \then \, $\kappa
\in S_{\awc}$.
\end{claim}

\begin{PROOF}{\ref{p56}}
It is similar to \ref{q7-item2}(2), but let us elaborate. Choose a closed
unbounded subset $E$ of $\kappa$ disjoint to $S_{\inc}^\kappa$. Let $A$ be
$E$ or any unbounded subset of $\kappa$ such  that $\partial \in
S_{\inc}^\kappa\ \Rightarrow \ \partial > \sup(A \cap \partial)$. 

Define functions $F_0:{}^{\kappa >}2 \longrightarrow {}^{\kappa
  >}2$ and $F_1:\bbQ_\kappa \longrightarrow \bbQ_\kappa$  and
$F_2:\bbQ_\kappa \longrightarrow \Cohen_\kappa$ by
\begin{itemize}
\item $F_0(\eta)$ is the $\nu \in {}^{\kappa >}2$ of length $\otp(\ell
  g(\eta) \cap A)$ and 
\[\alpha < \ell g(\eta) \wedge \alpha \in A\quad \Rightarrow\quad
\nu(\otp(\alpha \cap A))= \eta(\alpha)\]
(for $\eta\in {}^{\kappa >}2$), 
\item $F_1(p) = \{F_0(\eta):\eta \in p\}$ (for $p\in \bbQ_\kappa$),
\item $F_2(p)=F_0(\tr(p)) = \tr(F_1(p))$ (for $p\in \bbQ_\kappa$). 
\end{itemize}
Now, 
\begin{enumerate}
\item[$(*)_1$]  if $p \in \bbQ_\kappa$ and $\Cohen_\kappa \models
  ``F_2(p) \trianglelefteq \nu$''  \ then \, for some $q \in \bbQ_\kappa$ we have
  $\bbQ_\kappa \models ``p \le q"$ and $F_2(q) = \nu$.
\end{enumerate}
[Why?  By the choice of $A$ and we prove this by induction on $\ell
g(\nu)$ as in \S1.]
\begin{enumerate}
\item[$(*)_2$]  If $p \in \bbQ_\kappa$ then $F_1(p) = \{\rho:\rho
\trianglelefteq F_0(\tr(p))$ or $F_0(\tr(p)) \vartriangleleft \rho \in
{}^{\kappa >}2\}$.
\end{enumerate}
[Why?  As in \S1 or the proof of \ref{p62}.]
\begin{enumerate}
\item[$(*)_3$]  if $\bbQ_\kappa \models ``p \le q"$ then
  $\Cohen_\kappa \models ``F_2(p) \trianglelefteq F_2(q)$''. 
\end{enumerate}
[Why?  Obvious.]

Together we are done
\end{PROOF}

\begin{claim}
\label{p62}
\begin{enumerate}
\item Assume that $W \subseteq S^\kappa_\pr$ (see \ref{p36a}) is stationary
  but not reflecting. \Then \, forcing with $\bbQ_\kappa$ adds a Cohen 
$\kappa$--real.
\item Above also $\Pr(\kappa)$ holds.
\end{enumerate}
\end{claim}

\begin{remark}
\label{p65}
We can replace the assumption of \ref{p62}(1) by 
\begin{enumerate}
\item[$(*)$]  there is a sequence $\bar{\cI} = \langle \cI_i:i <
  \kappa\rangle$ of dense open sets such that for no $\partial \in
  S^\kappa_{\inc}$ and $p\in\bbQ_\partial$ do we have $\cI_i \rest \partial$
  is predense in $\bbQ_\partial$ above $p$ for every $i\in [\ell
  g(\tr(p)), \partial)$ where $\cI_i\rest \partial = \{p \cap {}^{\partial
    >}2:p \in \cI_i$ satisfies $\ell g(\tr(p)) < \partial\}$.
\end{enumerate}
That is, if $(*)$ holds true, then $\bbQ_\kappa$ adds a $\kappa$--Cohen
real. We intend to return to it in \cite{Sh:F1199}. 
\end{remark}

\begin{PROOF}{\ref{p62}}
\noindent (1)\quad Let $W\subseteq S^\kappa_\pr$ be a non-reflecting
stationary set. Choose a sequence $\bar\rho = \langle \rho_\partial:\partial
\in W\rangle$ such that:  
\begin{enumerate}
\item[$(\bullet)_1$]  $\partial \in W \Rightarrow \rho_\partial \in
  {}^{\kappa >} 2$
\item[$(\bullet)_2$]  for each $\rho \in {}^{\kappa >} 2$ the set
  $\{\partial \in W: \rho_\partial = \rho\}$ is stationary.
\end{enumerate}
For every $\partial \in W$ we fix open dense sets
$\cI_\vare^\partial\subseteq \bbQ_\partial$ (for  $\vare<\partial$) such
that:  
\begin{enumerate}
\item[$(\bullet)_3$] if $p\in\bbQ_\partial$ \then\;
  $\lim_\partial(p)\nsubseteq \bigcap\limits_{\vare<\partial}
  \set(\cI_\vare^\partial )$.  
\end{enumerate}
Then for $\partial \in W$  we define 
\begin{enumerate}
\item[$(\bullet)_4$]  $A_\partial:= {}^\partial 2\setminus 
  \bigcap\limits_{\vare<\partial} \set(\cI_\vare^\partial )$.
\end{enumerate}
Clearly, 
\begin{enumerate}
\item[$(\bullet)_5$] $A_\partial\in\id(\bbQ_\partial)$ but $\lim_\partial(p)
  \cap A_\partial\neq\emptyset$ for every $p \in \bbQ_\partial$. 
\end{enumerate}
Now, 
\begin{enumerate}
\item[$(\bullet)_6$]    for $\partial \in W$ we can find a partition
  $(A^1_\partial,A^2_\partial)$ of $A_\partial$ such that: for every
  $p \in \bbQ_\partial$ we have $\lim_\partial(p) \cap A^\ell_\partial \ne
  \emptyset$ for $\ell=1,2$, equivalently for every $\cX \in
  \id(\bbQ_\partial)$ and $p \in \bbQ_\partial$, $\lim_\partial(p) \cap
  A^\ell_\partial \ne \emptyset$ for $\ell=1,2$.
\end{enumerate} 
[Why?  Since $\bbQ_\partial$ has cardinality $2^\partial$ and
$\id(\bbQ_\partial)$ is generated by $2^\partial$ sets, it is enough to
prove that for every $p \in \bbQ_\partial$ and $\cX = {}^\partial 2 \setminus  
\set(\bar{\cI}) \in \id(\bbQ_\partial)$, where $\bar{\cI}$ is a sequence of
$\partial$ maximal antichains of $\bbQ_\partial$, the set $\cX \cap
\lim_\partial(p) \cap A_\partial$ has cardinality $2^\partial$. \Wilog \, 
$(S_\partial,\bar\Lambda_\partial, \bar p_\partial,\bar{\cI}_\partial)$ is
as in \ref{p59}.  Given $p$ and $\cX$, i.e. $(S_\partial,
\bar\Lambda_\partial, \bar p_\partial,\bar{\cI}_\partial)$ we let $E$ be a
club of $\partial$ disjoint to $S_p,S_\partial$ and $W$ and to $[0,\ell
g(\tr(p))$.  So consider the tree $\cT = (\bigcup\limits_{\alpha \in
  E}{}^\alpha 2) \cup {}^\kappa 2$. Recall $p \cap \cT$ is a really closed
subtree and for each $\vare < \partial$, $\langle p \cap \cT_\partial:p \in 
\cI_{\partial,\vare}\rangle$ is a sequence of closed subtrees with no
maximal nodes such that $\lim_\partial(p) = \lim(p \cap \cT_\gamma)$ are 
pairwise disjoint.  The rest should be clear.]  

We let $\name\ell_\partial$ be a $\bbQ_\kappa$--name for an element of
$\{0,1,2\}$ such that 
\begin{enumerate}
\item[$(\bullet)_7$]  $\Vdash_{\bbQ_\kappa}$ ``$\name\ell_\partial =\iota$
  \Iff \, $\name\eta \rest \partial\in A_\partial^\iota$'' for $\iota=1,2$
  and  $\Vdash_{\bbQ_\kappa} ``\name\ell_\partial =0$ \Iff 
  \, $\name\eta \rest \partial\notin A_\partial"$.
\end{enumerate}
Lastly, let $\name\nu$ be (the $\bbQ_\kappa$--name for) the concatenation of    
$\langle\rho_\partial: \partial \in W$ and $\name\ell_\partial=2\rangle$.
We will argue that $\Vdash_{\bbQ_\kappa}
``\name\nu$ is Cohen over $\bold V"$. To this end we will prove that:
\begin{enumerate}
\item[$(\boxplus)$]  if $p \in \bbQ_\kappa$, $\partial \in W$, $\partial >
  \ell g(\tr(p))$ \then \, there is $\tau\in p\cap {}^\partial 2$ such that: 
\begin{enumerate}
\item[(a)]  $\tau\in A_\partial^2$, equivalently $p^{[\tau]}
  \Vdash``\name\ell_\partial =2"$, 
\item[(b)]  if $\theta \in W \cap \partial$, $\theta > \ell g(\tr(p))$
  \then \, $\tau\rest \theta\notin A_\theta^2$, equivalently $p^{[\tau]}
  \Vdash$`` $\name\ell_\theta$ is 0 or is 1''.
\end{enumerate}
\end{enumerate}

\noindent
\underline{Why is $(\boxplus)$ enough}?  Recalling \ref{p31}, let
$(\eta,\bar\alpha)$ be as there, and we shall show that
$\Vdash_{\bbQ_\kappa} ``\name{\nu} \notin X_{\eta,\bar{\alpha}}"$.  Let $p
\in\bbQ_\kappa$, $j<\kappa$ and let $\nu_*$ be the concatenation of  
\[\big\{\rho_\partial:\partial \in W,\ \partial \le \ell g(\tr(p)) 
\text{ and } \tr(p)\rest\partial\in A_\partial^2\big\}.\]
Let $\rho_* \in {}^{\kappa >}2$ be such that for some $i \in
[j,\kappa)$ we have 
\begin{enumerate}
\item[$(\bullet)_8$]  $\nu_*\char 94 \rho_*$  has length $\ge \alpha_{i+1}$
  and it does include $\eta \rest [\alpha_i,\alpha_{i+1})$''. 
\end{enumerate}
Clearly it suffices to prove that for some $q$:
\begin{enumerate}
\item[$(\bullet)_9$]  $p \le_{\bbQ_\kappa} q$ and $q \Vdash ``\nu_* \char 94
  \rho_*\trianglelefteq \name\nu"$.  
\end{enumerate}
By the choice of $\bar\rho$, the set $W' = \{\partial \in W:\partial\notin
S_p,\ \partial >\ell g(\tr(p))$ and $\rho_\partial = \rho_*\}$ is a
stationary subset of $\kappa$.  Pick $\partial_* \in W'$ and then
choose $\tau \in p \cap {}^{\partial_*}2$ as in (a),(b) of $(\boxplus)$.
Let $q =p^{[\tau]}$. 

So the conclusion of \ref{p62} follows and $(\boxplus)$ is indeed 
enough, but we still owe: 
\medskip

\noindent
\underline{Why $(\boxplus)$ is true}?  Let  $p \in \bbQ_\kappa$ as
witnessed by $(\tr(p),S_p,\bar{\Lambda}_p)$, and let $\partial \in W$,
$\partial > \ell g(\tr(p))$. Put  
\begin{itemize}
\item $\tr(q)=\tr(p)$,
\item $S_q=S_p\cup (W\cap\partial)$, and 
\item if $\theta\in S_q\setminus S_p$, then  $\Lambda_{q,\theta}=
  \{\cI^\theta_\vare: \vare<\theta\}$, and 
\item if $\theta\in S_p\cap (W\cap\partial)$, then $\Lambda_{q,\theta}
  =\Lambda_{p,\theta}\cup \{\cI^\theta_\vare: \vare<\theta\}$. 
\end{itemize}
This determines a condition $q\in\bbQ_\kappa$ stronger than $p$.  It follows
from the definition of $\bar{\Lambda}_q$ and $S_q$ that 
\begin{enumerate}
\item[$(\bullet)_{10}$]  if $\ell g(\tr(q))<\theta\in W\cap \partial$, then
  $q\cap {}^\theta 2\subseteq \set(\Lambda_{q,\theta})\subseteq {}^\theta 2
  \setminus A_\theta$. 
\end{enumerate}
Anyhow by $(\bullet)_6$ we are done.
\medskip

\noindent (2)\quad Let $A_\partial^\iota$ for $\partial \in W$ be as in (1)
above such that
\begin{enumerate}
\item[$(\bullet)_{11}$]   $\eta\in A^2_ \partial$ implies that
  $\{\alpha<\partial: \eta(\alpha)=1\}$ is stationary. 
\end{enumerate}
For $\alpha<\kappa$ define 
\[\cI_\alpha=\{p\in \bbQ_\kappa: \ell g(\tr(p))>\alpha \ \mbox{ and for
  some }\partial\in (\alpha,\ell g(\tr(p)))\cap W\ \mbox{ we have }\
\tr(p)\rest \partial\in A_\partial^2\}.\]
Clearly each $\cI_\alpha$ is a dense open subset of $\bbQ_\kappa$. We will
argue that $\langle\cI_\alpha:\alpha<\kappa\rangle$ witnesses $\Pr(\kappa)$, 
that is we show that for each $p\in\bbQ_\kappa$ we have
$\lim_\kappa(p)\nsubseteq \bigcap\limits_{\alpha<\kappa} \set(\cI_\alpha)$.  

Let $p\in\bbQ_\kappa$ be witnessed by $(\eta,S,\bar{\Lambda})$ and let
$\alpha=\ell g(\eta)$. We will show that $\lim_\kappa(p)\nsubseteq
\set(\cI_{\alpha+1})$. Towards this let $E$ be a club of $\kappa$ disjoint
from $S$ with $\min(E)=\alpha=\ell g(\tr(p))$ and
\[\min(E)<\alpha\in E \wedge \alpha >\sup(\alpha\cap E )\quad
\Rightarrow\quad \alpha \mbox{ is singular.}\] 
Let $\langle\alpha_i:i<\kappa\rangle$ be an increasing enumeration of
$E$. By induction on $i<\kappa$ we choose $\eta_i$ so that 
\begin{enumerate}
\item[$(*)_i$] 
  \begin{enumerate}
\item[(a)] $\eta_i\in p\cap {}^{(\alpha_i)} 2$,
\item[(b)] $j<i\ \Rightarrow\ \eta_j\vartriangleleft \eta_i\wedge \eta _i (
  \alpha_i)=0$,  
\item[(c)] if $\partial \in W\cap (\alpha_0,\alpha_i]$, then
  $\eta_i\rest \partial \notin A_\partial^2$.  
  \end{enumerate}
\end{enumerate}
\underline{This is enough} as letting $\eta=\bigcup\limits_{i<\kappa}
\eta_i$ we will have $\eta\in \lim_\kappa(p)\setminus
\set(\cI_{\alpha+1})$. 

\noindent \underline{Why can we carry out the induction?}

For $i=0$ we put $\eta_0=\tr(p)$,

for a limit $i$ we put $\eta_i=\bigcup_{i<j}\eta_j$ noting that if $ \alpha _ i\in W$ 
then $\eta _i$ is not in $A^2_{\alpha _i}$ by $(\bullet)_{11}$,

for a successor $i=j+1$ we proceed as in the proof of $(\boxplus)$ of the
first part recalling $\alpha _i \notin W$.  
\end{PROOF}

\begin{claim}
\label{p68}
\begin{enumerate}
\item The assumption of \ref{p62}(1) holds when $\bold V = \bold L$ and
  $\kappa$ is Mahlo not weakly compact.   
\item When the assumption of \ref{p68}(1) or of \ref{p62}(1) hold for 
$\kappa$, \then 
\[\cov(\bbQ_\kappa)\leq \cov(\Cohen_\kappa)\ \mbox{ and }\ 
\cov(\bbQ_\kappa) \leq \non(\Cohen_\kappa)\leq \non(\bbQ_\kappa).\]  
\end{enumerate}
\end{claim}

\begin{remark}
\label{p68a}
\begin{enumerate}
\item So when \ref{p68}(1) applies, the Cicho\'n diagram for
  $\id(\Cohen_\kappa)$ and $\id(\bbQ_\kappa)$ is very different than the
  $\kappa = \aleph_0$ case, i.e., we have additional inequalities.  
\item In \ref{p68}(1), note that if  $\kappa$ is inaccessible not Mahlo
  then the conclusion of \ref{p62}(1) holds by \ref{p56}.
\end{enumerate}
\end{remark}

\begin{PROOF}{\ref{p68}}
  1)\quad Since $\kappa$ is Mahlo not weakly compact, by a result of Jensen
  we know that every stationary subset of $\kappa$ contains a non-reflecting
  stationary subset. So we may use Observation \ref{obs30}(4) and argue that
  again we are in the case of \ref{p62}(1).   
\medskip

\noindent
2)\quad It follows from \ref{p62}, that there is a  $\bbQ_\kappa$--name
$\name\varrho$ such that for some Borel function $\bold B:{}^\kappa
2\longrightarrow {}^\kappa \kappa$ we have 
\begin{enumerate}
\item[$(*)_1$]  $\Vdash_{\bbQ_\kappa} ``\name\varrho$ is a $\kappa$--Cohen
  real over ${\bold V}$ and $\name\varrho = \bold B(\name\eta)$''.
\end{enumerate}
Hence
\begin{enumerate}
\item[$(*)_2$] $\cov(\bbQ_\kappa)\leq\cov(\Cohen_\kappa)$
\end{enumerate}
Why? Let $\mu=\cov(\Cohen_\kappa)$ and let $\langle X_\zeta:\zeta<\mu
\rangle$ be a sequence of $\kappa$--meagre $\kappa$--Borel sets with union
${}^\kappa 2$. Let $\bB_\zeta\in\id(\bbQ_\kappa)$ be such that 
\[\eta\in {}^\kappa2\setminus \bB_\zeta\quad\Rightarrow\quad \bB(\eta)\notin
X_\zeta.\] 
We claim that then $\bigcup\limits_{\zeta<\mu} \bB_\zeta={}^\kappa 2$. If
not, then we may pick $\eta\in {}^\kappa 2\setminus
\bigcup\limits_{\zeta<\mu} \bB_\zeta$. But now, for every $\zeta<\mu$,
$\bB(\eta)\notin X_\zeta$, so $\bigcup\limits_{\zeta<\mu} X_\zeta\neq
{}^\kappa 2$ --- a contradiction. 
\medskip

Similarly,
\begin{enumerate}
\item[$(*)_3$]  $\non(\Cohen_\kappa)\leq \non(\bbQ_\kappa)$.
\end{enumerate}
Why? Let $\{\eta_\zeta:\zeta<\mu\}\subseteq {}^\kappa 2$ be a set not
belonging to $\id(\bbQ_\kappa)$. Then $\{\bB(\eta_\zeta):\zeta<\mu\}$
exemplifies $\non(\Cohen_\kappa)\leq\mu$. 

Also, 
\begin{enumerate}
\item[$(*)_4$]  $\cov(\bbQ_\kappa)\leq \non(\Cohen_\kappa)$.
\end{enumerate}
Why? By \ref{p39}(1), noting that its assumption
``$\kappa = \sup(S^\kappa_{\inac})$'' follows by our present assumptions.
\end{PROOF}

\begin{claim}
\label{p68fromF1580}
If $\bold V = \bold L$, then an inaccessible $\kappa$ satisfies $\Pr(\kappa)$ iff
$\kappa$ is not weakly compact iff $\bbQ_\kappa$ adds a $\kappa$-Cohen. 
\end{claim}

\begin{PROOF}{\ref{p68fromF1580}}
We prove this by considering possible cases.
\medskip

\noindent {\sc Case 1}: $\kappa$ is not Mahlo.\\
Then 
\begin{enumerate}
\item[(a)] $\kappa$ is not weakly compact,
\item[(b)] $\bbQ_\kappa$ add a $\kappa$--Cohen real by \ref{p56},
\item[(c)] $\Pr(\kappa)$ holds by \ref{obs30}(1).
\end{enumerate}
\medskip

\noindent {\sc Case 2}: $\kappa$ is Mahlo not weakly compact.\\ 
By \ref{obs30}(4), $S^\kappa_\pr$ is a stationary subset of $\kappa$. By a
result of Jensen there is a stationary $W\subseteq S^\kappa_\pr$ which does
not reflect. Hence by \ref{p62} the forcing notion $\bbQ_\kappa$ adds a
$\kappa$--Cohen real and $\Pr(\kappa)$ holds true. 
\medskip

\noindent {\sc Case 3}: $\kappa$ is weakly compact.\\ 
Then $\bbQ_\kappa$ is $\kappa$--bounding hence does not add a $\kappa$-Cohen 
by \ref{n13} and $\Pr(\kappa)$ fails by \ref{obs30}(2), i.e., \ref{p8}(2). 
\end{PROOF}

\section{What about the parallel to ``amoeba forcing''?}
\label{amoeba}

\begin{definition}
\label{p84leftover}
\begin{enumerate}
\item We say that $\cJ\subseteq\bbQ$ is nice if $\cJ^{[\alpha,\pi]}
  \subseteq \cJ$ for every $\alpha<\kappa$ and a permutation $\pi: {}^\alpha
  2\longrightarrow {}^\alpha 2$ (remember \ref{p84}(2)).
\item We say that a family $\Lambda$ of subsets of $\bbQ_\kappa$ is nice
  \when \,: $\Lambda^{[\alpha]} \subseteq \Lambda$ for every $\alpha <
 \kappa$ (remember \ref{p84}(3)).\\
(Equivalently,  if $\cI_1 \in \Lambda$, $\cI_2 \subseteq \bbQ_\kappa$,
$\alpha<\kappa$ and $\cI^{[\alpha,\pi]}_1 = \cI_2$ \then \, $\cI_2 \in
\Lambda$). 
\item  For $p \in \bbQ_\kappa$ let $\nb(p) = \{p^{[\eta,\nu]}:\eta \in p
   \cap {}^\alpha 2,\nu \in {}^\alpha 2$ for some $\alpha < \kappa\}$.
\end{enumerate}

\end{definition}

\begin{claim}
\label{p87leftover}
If $\Lambda \subseteq \{\cI:\cI \subseteq \bbQ_\kappa$ is predense$\}$
has cardinality $\le \kappa$ \then \, so is $\Lambda^{[< \kappa]}$ and it is
nice.
\end{claim}

\begin{PROOF}{\ref{p87leftover}}
It follows from \ref{p87}.
\end{PROOF}

\begin{claim}
\label{p90}
\begin{enumerate}
\item If $p \in \bbQ_\kappa$ \then \, $\nb(p)$ is a predense subset of 
$\bbQ_\kappa$.
\item If $p \in \bbQ_\kappa$ \then \, $\nb(p)$ is nice and 
\[\set(\nb(p))=\big\{\eta \in {}^\kappa 2:\mbox{ there is }\nu \in
{\lim}_\kappa(p)\mbox{ such that }(\forall^\infty \alpha <
\kappa)(\eta(\alpha) = \nu(\alpha))\big\}.\] 
\item {[$\kappa$ weakly compact]}\quad If $X \in \id(\bbQ_\kappa)$ \then \,
  for a dense set of $p \in \bbQ_\kappa$ we have $\set(\nb(p)) \subseteq
{}^\kappa 2 \backslash X$.
\end{enumerate}
\end{claim}

\begin{PROOF}{\ref{p90}}
(1)\quad Clearly for every $p,q \in \bbQ_\kappa$ we can choose $\alpha \ge
\max\{\ell g(\tr(p),\ell g(\tr(q))\}$ such that $\alpha < \kappa$ and then
choose $\eta \in p \cap {}^\alpha 2,\nu \in q \cap {}^\alpha 2$ and $\pi
\in \Sym({}^\alpha 2)$ such that $\pi(\eta)= \nu$, so $q_1 =
p^{[\eta,\nu]} \in \nb(p)$ and $q_1,q$ have a common member $\nu$ which
is of length $\ge \ell g(\tr(q_1)),\ell g(\tr(q))$, hence $q_1,q$ are
compatible.  
\medskip

\noindent
(2)\quad Should be clear.
\medskip

\noindent
(3)\quad There is a family $\Lambda$ of $\le \kappa$ maximal antichains of 
$\bbQ_\kappa$ such that $X \cap \set(\Lambda) = \emptyset$.  \Wilog \,
$\Lambda = \Lambda^{[< \kappa]}$ and hence the set $Y = {}^\kappa 2
\backslash \set(\Lambda) \in \id(\bbQ_\kappa)$ satisfies:
\begin{itemize}
\item  if $\eta_1,\eta_2 \in {}^\kappa 2$ and $\kappa > \sup\{\alpha <
  \kappa:\eta_1(\alpha) \ne \eta_2(\alpha)\}$, then $\eta_1 \in Y
  \Leftrightarrow \eta_2 \in Y$.
\end{itemize} 
Now, as $Y \in \id(\bbQ_\kappa)$ by \ref{p8}(2) for a dense set of $p \in  
\bbQ_\kappa$, $\lim_\kappa(p)$ is disjoint to $Y$, but by the choice of $Y$
this holds for any $p' \in \nb(p)$, so we are done. 
\end{PROOF}

\begin{definition}
\label{p83}
Let $\bbQ^{\am}_\kappa$ be the following forcing notion:
\begin{enumerate}
\item[(A)]  a member of $\bbQ^{\am}_\kappa$ has the form $(\alpha,p,E)$ with
  $\alpha < \kappa,p \in \bbQ_\kappa,E$ a club of $\kappa$ disjoint to $S_p$
  and $\tr(p) = \langle \rangle$, 
\item[(B)]  the order on $\bbQ^{\am}_\kappa$ is: $(\alpha_1,p_1,E_1) \le
(\alpha_2,p_2,E_2)$ \Iff \,
\begin{enumerate}
\item[(a)]  $\alpha_1 \le \alpha_2$, 
\item[(b)]  $p_1 \le_{\bbQ_\kappa} p_2$,
\item[(c)]  $p_1 \cap {}^{(\alpha_1)}2 = p_2 \cap {}^{(\alpha_1)}2$, 
\item[(d)]  $E_1 \supseteq E_2$ and $E_1 \cap \alpha_1 = E_2 \cap
  \alpha_1$.
\end{enumerate}
\item[(C)]  The generic of $\bbQ^{\am}_\kappa$ is $\name p_\kappa =
  \bigcup\{p \cap {}^{\alpha \ge}2:(\alpha,p,E) \in \name{\bold
    G}_{\bbQ^{\am}_\kappa}\}$. 
\end{enumerate}
\end{definition}

\begin{claim}
\label{p97}
\begin{enumerate}
\item $\bbQ^{\am}_\kappa$ is a $\kappa$--strategically complete
  $\kappa^+$--cc (nicely definable) forcing notion and $\name p_\kappa$ is
  indeed a generic for $\bbQ^{\am}_\kappa$. 
\item $\Vdash_{\bbQ^{\am}_\kappa}$``$\name p_\kappa \in \bbQ_\kappa$''. 
\item Assume $\kappa$ is weakly compact. If $\cI$ is a predense subset of
  $\bbQ_\kappa$ (in $\bold V$) \then \, $\Vdash_{\bbQ^{\am}_\kappa}
  ``\set(\cI) \supseteq \set(\nb(\name p_\kappa))"$.
\item  Assume $\kappa$ is weakly compact. Then $\Vdash_{\bbQ^{\am}_\kappa}
  ``{}^\kappa 2 \backslash \set(\nb(\name p_\kappa)) \subseteq {}^\kappa 2$
  is a  member of $\id(\bbQ_\kappa)$ including all the old $\kappa$--Borel
  sets from $\id(\bbQ_\kappa)"$. 
\end{enumerate}
\end{claim}

\begin{PROOF}{\ref{p97}}
(1)\quad Easy.
\medskip

\noindent
(2)\quad Recall that for every $p \in \bbQ_\kappa$ there is a canonical
   witness $(\tr(p),S_p,\bar\Lambda_p)$ (see \ref{n5}(C)(a)). Let us define
   some $\bbQ^{\am}_\kappa$--names: 
\begin{enumerate}
\item[$(*)_1$]  
\begin{enumerate}
\item[(a)] $\name E = \bigcap\{E_p:p \in \name{\bold G}\}$,
\item[(b)] $\name S = \bigcup\{\name S_p:p \in \name{\bold G}\}$, 
\item[(c)] for every $\partial \in \name S$, $\name\Lambda_\partial =
  \bigcup\{\Lambda_{p,\partial}:p \in \name{\bold G}$ satisfies $\partial
  \in S_p\}$,
\item[(d)] $\name{\bar\Lambda} = \langle \name\Lambda_\partial:\partial \in
  \name S\rangle$, 
\item[(e)] $\name\varrho$ is $\langle\rangle$. 
\end{enumerate}
\end{enumerate}
Now,
\begin{enumerate}
\item[$(*)_2$]  for every $\beta < \kappa$, the set 
\[\cI_\beta := \big\{(\alpha,p,E) \in \bbQ^{\am}_\kappa:\alpha \ge
\beta\big\}\] 
is a dense open subset of $\bbQ^\am_\kappa$.
\end{enumerate}
[Why?  If $\beta < \kappa$ and $(\alpha_1,p_1,E_1) \in\bbQ^{\am}_\kappa$
then $(\alpha_1 + \beta,p_1,E_1) \in \bbQ^{\am}_\kappa$ is above
$(\alpha_1,p_1,E_1)$ and belongs to $\cI_\beta$.]

\begin{enumerate}
\item[$(*)_3$]  $\Vdash ``\name E$ is a club of $\lambda$".
\end{enumerate}
[Why?  Unbounded as for every $\beta < \kappa$ and $(\alpha_0,p_0,E_0) \in
\bbQ^{\am}_\kappa$, let $\alpha_1 = \min\{\delta \in E_0:\delta >
\alpha_0,\delta > \beta\}$ so $(\alpha_1 +1,p_0,E_0)$ is above
$(\alpha_0,p_0,E_0)$ and forces $\delta \in \name E$.  Being closed is easy,
too.]

\begin{enumerate}
\item[$(*)_4$]  $\Vdash ``\name S$ is a nowhere stationary subset of 
$S^\kappa_{\inc}"$.
\end{enumerate}
[Why?  First, for every $\beta < \kappa$, by $(*)_2$ for a dense set
of $(\alpha,p,E) \in \bbQ^{\am}_\kappa$ we have $\alpha > \beta$. Since 
$(\alpha,p,E) \Vdash ``\name S \cap \alpha = S_p \cap \alpha"$, we get  
that $\name S \cap \alpha$ is nowhere stationary and hence $\name S \cap 
\beta$ is nowhere stationary.  Second, $\Vdash ``\name S$ is not stationary"
because $\Vdash ``\name E$ is a club of $\kappa$ disjoint to $\name S$" by
the definition of $\bbQ^{\am}_\kappa$.  Together we are done.]

\begin{enumerate}
\item[$(*)_5$]  $\Vdash ``\name\Lambda_\partial$ is a set of $\le \partial$
  predense subsets of $\bbQ_\partial$ for $\partial\in\name{S}$''. 
\end{enumerate}
[Why?  Given $(\alpha_0,p_0,E_0) \in \bbQ^{\am}_\kappa$, \wilog \, 
$\alpha_0 > \partial$ and hence it forces $\name\Lambda_\partial$
is $\Lambda_{p_0,\partial}$ if $\partial \in S_{p_0}$, not defined (or
$\emptyset$) otherwise; the rest is clear.]

\begin{enumerate}
\item[$(*)_6$]  $\Vdash ``(\name\varrho,\name S,\name{\bar\Lambda} )$
  witnesses $\name p_\kappa \in \bbQ_\kappa$.
\end{enumerate}
[Why?  Read \ref{p83}(C) and $(*)_3$--$(*)_5$.]
\medskip

\noindent
(3)\quad It suffices to prove the following: 
\begin{enumerate}
\item[$(*)_1$]   if $\alpha < \kappa$ and $\eta \in {}^\alpha 2$, $\nu \in
  {}^\alpha 2$ then 
\[\Vdash_{\bbQ^{\am}_\kappa}\mbox{ `` if }\eta \in 
  \name p_\kappa \cap {}^\alpha 2\mbox{ then }\lim( \name{p}^{[\eta,\nu]})  
  \subseteq \set(\cI)\mbox{ ''.}\]
\end{enumerate}
Now, 
\begin{enumerate}
\item[$(*)_2$]   fixing $\alpha$, \wilog \, for every $\pi \in 
\Sym({}^\alpha 2)$ we have $\cI^{[\alpha,\pi]} = \cI$.
\end{enumerate}
\mn
[Why?  Let $\cI_1 = \{p \in \bbQ_\kappa$: for every $\pi \in 
\Sym({}^\alpha 2)$, $p$ is above some member of $\cI^{[\alpha,\pi]}\ \}$.
Clearly: 
\begin{itemize}
\item  $\cI_1 \subseteq \bbQ_\kappa$ is predense,
\item  $\cI^{[\alpha,\pi]}_1 = \cI_1$ for every $\pi \in
  \Sym({}^\alpha 2)$,
\item  $\set(\cI_1) \subseteq \set(\cI)$.
\end{itemize}
Hence we can replace $\cI$ by $\cI_1$ so finishing the proof of $(*)_2$.]

So
\begin{enumerate}
\item[$(*)_3$]   in $(*)_1 + (*)_2$, \wilog \, $\nu = \eta$ so 
$\name p^{[\nu,n]}_\kappa = \name p_\kappa$.
\end{enumerate}
Let 
\begin{enumerate}
\item[$(*)_4$] $(\alpha_0,p_0,E_0) \in \bbQ^{\am}_\kappa$ and $\eta \in
  {}^{\alpha}2$.  
\end{enumerate}
We shall find $(\alpha_1,p_1,E_1) \in \bbQ^{\am}_\kappa$ above
$(\alpha_0,p_0,E_0)$ and  forcing that $\eta \notin \name p_\kappa$ or
forcing the statement in $(*)_1$.  First, by $(*)_2$ of the proof of part
(2), \wilog \, $\ell g(\eta) < \alpha_0$; so if $\eta \notin p_0$ then
$(\alpha_0,p_0,E_0) \Vdash ``\eta \notin \name p_\kappa"$ and we are
done. So we can assume $\eta \in p_0$. 

As $\kappa$ is weakly compact for some $\partial \in S^\kappa_{\inc}$
which is $>\alpha_0$ we have:
\begin{enumerate}
\item[$(*)_5$]  the set 
\[\cI_\partial = \big\{q \cap {}^{\partial >}2:q \in \cI\mbox{ and }\ell
g(\tr(q)) < \partial\big\}\]
is predense  in $\bbQ_\partial$.
\end{enumerate}
Next,
\begin{enumerate}
\item[$(*)_6$]  for every $\nu \in \set(\cI_\partial) \cap p_0$ choose
  $q_\nu \in \cI$ such that $\ell g(\tr(q_\nu)) < \partial$ and $\nu \in
  {\lim}_\partial (q_\nu \cap {}^{\partial >}2)$ equivalently $\nu \in q_\nu
  \cap {}^\partial 2$. 
\end{enumerate}
Let
\begin{enumerate}
\item[$(*)_7$]  
\begin{enumerate}
\item[(a)]  $S' = \bigcup\{S_{q_\nu} \backslash \partial:\nu \in
  \set(\cI_\partial) \cap p_0\} \cup S_{p_0} \cup \{\partial\}$,
\item[(b)]  for $\theta \in S'$ let $\Lambda'_\theta$ be:
\begin{enumerate}
\item[$(\alpha)$] $\bigcup \big\{\Lambda:\Lambda =\Lambda_{q_{\nu,\theta}}$ and
  $\nu\in \set(\cI_\partial) \cap p_0$ and $\theta \in S_{q_\nu} \backslash \partial^+$ \underline{or}
 $\Lambda = \Lambda_{p_0,\theta}$ and $\theta\in S_{p_0}\big\}$\qquad if $\theta \in
S' \backslash \partial^+$,
\item[$(\beta)$]  $\Lambda_{p_0,\theta} \cup\{\cI_\partial\}$\qquad if
  $\theta = \partial \wedge \partial \in S_{p_0}$, 
\item[$(\gamma)$] $\{\cI_\partial\}$\qquad if $\theta \in \partial
  \wedge \partial \notin S_{p_0}$,
\item[$(\delta)$] $\Lambda_{p_0,\theta}$\qquad if $\theta \in S_{p_0}
  \cap \partial$. 
\end{enumerate}
\end{enumerate}
\end{enumerate}
Let $p_1 \in \bbQ_\kappa$ be defined by
\begin{enumerate}
\item[$(*)_8$]   $(\langle\rangle,S',\bar\Lambda')$ will witness $p_1$,
  where  
\begin{enumerate}
\item[$\bullet$]  $S'$ is from $(*)_7$,
\item[$\bullet$]  $\bar\Lambda' = \langle \Lambda'_\theta:\theta \in
  S'\rangle$, see $(*)_7$,
\end{enumerate}
\end{enumerate}
and let $\alpha_1=\alpha_0$ and $E_1\subseteq E_0$ be a club disjoint from 
$S'$ and such that $E_1\cap\partial =E_0\cap \partial$. Now one easily
verifies that $(\alpha_1,p_1,E_1)\in \bbQ^{\am}_\kappa$ is a condition
stronger than $(\alpha_0,p_0,E_0)$ and it forces that  
\[\eta \vartriangleleft \nu \in \name p_\kappa \cap {}^\partial 2\
\Rightarrow \ (\exists q \in \cI)(\tr(q) \vartriangleleft \nu \in q   \wedge
\name p^{[\nu]}_\kappa \subseteq q).\] 
\medskip

\noindent
(4)\quad Follows by part (3).
\end{PROOF}
\bigskip

\section {Generics and Absoluteness}
\label{absolut}

Recall from Definition \ref{y6} that we say that a set $\bold B \subseteq
{}^\kappa \cH(\kappa)$ is 
\begin{itemize}
\item a $\kappa$--stationary Borel if for some $\kappa$--Borel function
$F:{}^\kappa \cH(\kappa) \longrightarrow \cP(\kappa)$  we have $\eta \in B 
  \Leftrightarrow F(\eta)$ is stationary,
\item $\kappa$--nowhere stationary Borel if there is a $\kappa$--Borel
function   $F:{}^\kappa \cH(\kappa) \longrightarrow \cP(\kappa)$  such that
for every  $\eta \in {}^\kappa \cH(\kappa)$ we have: $\eta \in {\bold B}$
\Iff \, $F(\eta)$ is  a nowhere stationary subset of $\kappa$.  
\end{itemize}

\begin{claim}
\label{k3}
\begin{enumerate}
\item ``$p \in \bbQ_\kappa$" is\footnote{Using coding it does not matter
    whether we use ${}^\kappa 2$ or $\cP(\kappa)$ or ${}^\kappa \cH(\kappa)$
    or $\cP(\cH(\kappa))$, etc.} a $\kappa$--nowhere stationary Borel
  relation (see \ref{y6}(5)),  also it is $\Sigma^1_1(\kappa)$.
\item Both ``$p \le_{\bbQ_\kappa} q$'' and ``$p,q \in \bbQ_\kappa$ are
  compatible'' are $\kappa$-Borel relations (but pedantically there are
  $\kappa$--Borel relations whose restrictions to $\bbQ_\kappa$ are
  the above relations). 
\item If $\kappa$ is weakly compact, \then \, ``being $\kappa$--nowhere
  stationary Borel'' is  equivalent to ``being $\kappa$--Borel''. 
\item If $\kappa$ is weakly compact \then \, ``$\{p_i:i < \kappa\} \subseteq
  \bbQ_\kappa$ is predense'' is $\kappa$--stationary Borel. 
\item Changing the definition of $\bbQ_\kappa$, we may get that the
  relations ``$p \in \bbQ_\kappa$'', ``$p \le_{\bbQ_\kappa} q$'' as well as
  ``$p,q \in \bbQ_\kappa$ are compatible'' are $\kappa$--Borel and for every
  limit $\delta < \kappa$ there is an $\delta$--place $\kappa$--Borel
  function giving an increasing sequence of length $\delta$ an upper bound. 

The change does not affect the generic and the derived ideal.
\end{enumerate}
\end{claim}

\begin{PROOF}{\ref{k3}}
(1,2)\quad Straightforward. Note that for ``$p \in \bbQ_\kappa$" the main
point is ``there is a club $E$ of $\kappa$ disjoint to $S_p$'', as for $S
\subseteq \kappa$ statement ``$(\forall \alpha<\kappa)(S \cap \alpha$ is not
stationary)'' is $\kappa$--Borel.
\medskip

\noindent
(3)\quad Let $F:{}^\kappa \cH(\kappa) \longrightarrow \cP(\kappa)$ be
$\kappa$--Borel and let $X=\{A\subseteq\cH(\kappa): F(A)$ is nowhere
stationary$\}$. To show that $X$ is $\kappa$--Borel it is enough to note
that 
\begin{center}
  $A\subseteq\kappa$ is nowhere stationary if and only if $A$ does not
  reflect. 
\end{center}
So the assertion should be clear. 
\medskip

\noindent
(4)\quad We define $F:{}^\kappa(\bbQ_\kappa) \longrightarrow \cP(\kappa)$ as 
follows.  For $\bar p \in {}^\kappa(\bbQ_\kappa)$ let 
\[F(\bar p) = \big\{\partial \in S^\kappa_{\inc}: \{p_i \cap {}^{\partial >}2:i
< \partial\mbox{ and }\tr(p) \in {}^{\partial >}2\}\mbox{ is predense in }
\bbQ_\partial \big\}.\]  
Clearly,  $F$ is a $\kappa$-Borel function (well, replacing ${}^\kappa 2$ by
${}^\kappa(\bbQ_\kappa)$) and we have:
\begin{enumerate}
\item[$(*)$]   $\{p_i:i < \kappa\} \subseteq \bbQ_\kappa$ is predense 
\Iff \, $F(\bar p)$ is stationary in $\kappa$.
\end{enumerate}
Why?  First, if $\{p_i:i < \kappa\}$ is not predense let $q \in
\bbQ_\kappa$ be incompatible with every $p_i$ which means $(\tr(q)
\notin p_i) \vee (\tr(p_i) \notin q)$, so easily for every $\partial \in
(\ell g(\tr(q),\kappa)$, $q \cap {}^\partial 2$ witnesses $\partial \notin
F(\bar{p})$.  Second, if $\{p_i:i < \kappa\}$ is predense, use the proof of
``$\bbQ_\kappa$ is $\kappa$-bounding''.  So we are done 
(replacing ${}^\kappa(\bbQ_\kappa)$ by ${}^\kappa 2$ via coding).
\medskip

\noindent
(5)\quad We define $\bbQ'_\kappa$ as the set of all quadruples $q = 
(\varrho_q,S_q,\bar\Lambda_q,E_q)$ such that $(\varrho_q,S_q,
\bar\Lambda_p)$ is as in Definition \ref{n5}(A), for a unique $T_q = T[q]$ a
subtree of ${}^{\kappa >}2$ and $E_q$ is a club of $\kappa$ disjoint to $S_p
\backslash (\ell g(\varrho_q)+1))$ such that $\ell g(\varrho_q) \in E_q$.
We let $q_1 \le q_2$ \Iff \,: 
\begin{enumerate}
\item[(a)]  $\varrho_{q_1} \trianglelefteq \varrho_{q_2}$, $S_{q_2} \supseteq
  S_{q_1} \backslash (\ell g(\varrho_2)+1)$, 
\item[(b)]  $\partial \in S_{q_1} \backslash (\ell g(\varrho_2)+1)
\ \Rightarrow\ \Lambda_{q_1,\partial} \subseteq \Lambda_{q_2,\partial}$,
\item[(c)]  $\bbQ_\kappa \models T[q_1] \le T[q_2]$,
\item[(d)]  $E_{q_1} \supseteq E_{q_2}$,
\item[(e)]  if $q_1 \ne q_2$ then $\varrho_{q_1} \ne \varrho_{q_2}$.
\end{enumerate}
[Why the choice of (e)?  The motivation is that otherwise an increasing
sequence $\bar{p}=\langle p_\alpha: \alpha<\delta<\kappa\rangle$ with
$\tr(p_ \alpha)$ constant may have no upper bound because
$\bigcup\limits_{\alpha<\delta} S_{p_ \alpha}$ may reflect in some
$\partial>{\ell g}(\tr(p_ \alpha))$. But by the present definition: if
$\bar{p}$ is eventually constant this is trivial; if not then $\rho=
\bigcup\limits_{i<\delta }\tr(p_\alpha)$ has length which belongs to
$\bigcap\limits_{\alpha<\delta} E_{p_\alpha}$ and we can finish easily.]
\end{PROOF}

\begin{observation}
  \label{Claim8.2(4)}
Assume $\kappa$ is weakly compact. For a set  $X\subseteq {}^\kappa\kappa$
we have $(a) \Leftrightarrow (c)$ and $(b) \Leftrightarrow (d)$, where
\begin{enumerate}
\item[(a)] $X$ is $\kappa$--stationary Borel,
\item[(b)] ${}^\kappa\kappa\setminus X$ is $\kappa$--stationary Borel,
\item[(c)]  $X$ is $\Sigma^1_1(\kappa)$,
\item[(d)]  $X$ is $\Pi^1_1(\kappa)$.
\end{enumerate}
\end{observation}

\begin{remark}
Note that the family $\big\{X\subseteq {}^\kappa\kappa: X \mbox{ is }
\Sigma^1_1(\kappa)\big\}$ is closed under $(\exists Y\subseteq\kappa)$
and unions/intersections of  $\leq\kappa$ elements. 
\end{remark}

\begin{PROOF}{\ref{Claim8.2(4)}}
\noindent
\underline{Clause (a) implies  clause (c)}:

Let $F_1$ be a $\kappa$-Borel function from ${}^\kappa \kappa$ to
$\cP(\kappa)$ such that $X = \{\eta \in {}^\kappa \kappa:\bold
B(\eta)$ is stationary$\}$.  \Wilog \,
\begin{enumerate}
\item[$(*)_1$]   $F_1$ is defined by the sequence $\bar{\bold B}_1 =
  \langle \bold B_{1,\alpha}:\alpha < \lambda\rangle$, $\bold
  B_{1,\alpha}$ a Borel subset of ${}^\kappa \kappa$ such that $F_1(\eta) =
  \{\alpha:\eta \in \bold B_{1,\alpha}\}$. 
\end{enumerate}
Let $M_\kappa \prec (\cH(2^\kappa)^+,\in)$ of cardinality $\kappa$ be such
that $[M_\kappa]^{<\kappa} \subseteq M_\kappa$, $F_1 \in M_\kappa$
(necessarily $\kappa +1\subseteq M_\kappa)$. Let $\langle M_\alpha:\alpha
<\kappa\rangle$ be $\prec$--increasing continuous with union $M_\kappa$ such
that $\|M_\alpha\| \le |\alpha| + \aleph_0$ and $F_1 \in M_0$ (necessarily
$\kappa \in M_0$).    

Let $E = \{\mu:\mu < \kappa$ is strong limit cardinal such that $M_\mu \cap
\kappa = \mu$ hence $M_\mu \cap \cH(\kappa) = \cH(\mu)$ and $\alpha < \mu
\Rightarrow \|M_\alpha\| < \mu\}$.  Clearly $E$ is a club of $\kappa$.  For
$\mu \in E$ let $N_\mu$ be the Mostowski collapse of $M_\mu$ and let
$\pi_\mu$ be the isomorphism from $M_\mu$ onto $N_\mu$. Let $F^1_\mu =
\pi_\mu(F_1)$ and $\bar{\bold B}_\mu = \langle \bold B_{\mu,\alpha}:\alpha
< \mu\rangle = \pi_\mu(\bar{\bold B}_1)$.
Now, 
\begin{enumerate}
\item[$(*)_2$]  for $\mu \in E$ (only inaccessible interests us) we have
  $F^1_\mu:{}^\mu \mu \rightarrow \cP(\mu)$, 
\item[$(*)_3$]  for $\eta \in {}^\kappa \kappa$ the following conditions are
  equivalent: 
\begin{enumerate}
\item[$(\alpha)$]  $\eta \in X$,
\item[$(\beta)$]  $\cU_\eta := \{\partial < \kappa:\eta \rest \partial \in
  {}^\partial \partial$ and $F^1_\partial(\eta \rest \partial)$ is a
  stationary subset of $\partial\}$ is stationary in $\kappa$,
\item[$(\gamma)$]  the tree $\cT_\eta$ has no $\kappa$-branch, where
  $\cT_\eta = \bigcup\limits_{\alpha < \kappa} T_{\eta,\alpha}$ where
  $T_{\eta,\alpha}$ is the set of $\rho \in {}^\alpha \kappa$
such that:
\begin{enumerate}
\item[$\bullet_1$]  $\rho$ is an increasing continuous sequence of
  cardinals from $E$,
\item[$\bullet_2$]  $\eta \rest \rho(B) \in {}^{\rho(\beta)}\rho(\beta)$, 
\item[$\bullet_3$]  $\langle F^1_{\rho(\beta)}(\eta \rest \beta):\beta
  < \ell g(\alpha)\rangle$ is increasing, i.e., if $\beta_1 < \beta_2 =
   \ell g(\rho)$ then $F^1_{\rho(\beta_1)}(\eta \rest \beta_1) =
   F^1_{\rho(\beta_2)}(\eta \rest \beta_2) \cap \beta_1$, 
\item[$\bullet_4$]  $F^1_{\rho(\beta)}(\eta \rest \beta)$ is a
  non-stationary subset of $\rho(\beta)$, 
\end{enumerate}
\item[$(\delta)$]  for a stationary set of $\partial < \kappa$, the tree
  $\cT_\eta \cap {}^{\partial >}\partial$ has no $\partial$-branch.
\end{enumerate}
\end{enumerate}
This suffices because by $(\alpha) \Leftrightarrow (\gamma)$ in
$(*)_3$, clearly $X$ is defined by $(\gamma)$ and this can be
expressed by a $\Pi^1_1$-formula.
\medskip

Why does $(*)_3$ hold?
\medskip

\underline{$(\alpha) \Rightarrow (\beta)$}:

Let $M'_\lambda$ be like $M_\lambda$ but $\{M_\lambda,\bar M,\eta\}
\in M'_\lambda$ and let $\bar M' = \langle M'_\alpha:\alpha <
\kappa\rangle$ be like $\bar M$ for $M'_\lambda$ and $\{\bar
M_\gamma,\bar M,\eta\} \in M'_0$ and $E' \subseteq E$ is like $E$ for
$\bar M'$ and also $N'_\alpha,\pi_\alpha(\alpha \in E')$.  

Easily $\partial \in E'\ \Rightarrow\ \pi_\partial(\bar{\bold B}
\rest \partial) = \pi'_\partial(\bar{\bold B}_1 \rest \gamma)$, etc.
So for a club of $\partial < \kappa$, $F_1(\eta) \cap \partial = 
F^1_\partial(\eta \rest \partial)$ and we are easily done.
\medskip

\underline{$(\beta) \Leftrightarrow (\gamma)$}:

Easy, too.
\medskip

\underline{$(\gamma) \Leftrightarrow (\delta)$}:

Because $\kappa$ is weakly compact.
\medskip

\noindent
\underline{Clause (c) implies (a)}:

Similarly.
\medskip

\noindent
\underline{Clause (b) \Iff \, clause (d)}:

Similarly.
\end{PROOF}

\begin{claim}
\label{k4d}
Assume $\kappa$ is weakly compact.
\begin{enumerate}
\item ``$\{p_i:i < \kappa\} \subseteq \bbQ_\kappa$ is predense'' is 
 $\Pi^1_1(\kappa)$; this means $\{(i,\eta):\eta \in p_i,i <
\kappa\}$ is $\Pi^1_1(\kappa)$--set recalling \ref{y6}(3).
\item ``$X = {}^\kappa 2 \backslash \bigcup\{\lim_\kappa(\cT_\alpha): \alpha 
< \kappa\}$ belongs to $\id(\bbQ_\kappa)$ each $\cT_\alpha$ a subtree
of ${}^{\kappa >} 2$'' is a $\kappa$-stationary-Borel realtion.
\end{enumerate}
\end{claim}

\begin{PROOF}{\ref{k4d}}
(1)\quad By \ref{k3}(4) and \ref{Claim8.2(4)}
\medskip

\noindent
(2)\quad  As $\kappa$ is weakly compact, $X \in \id(\bbQ_\kappa)^+$ \Iff \,
there is $p \in \bbQ_\kappa$ such that $\lim_\kappa(p) \subseteq X$
\Iff \, there are $\alpha < \kappa$ and $q$ as in \ref{k3}(5) above $p$ 
such that $T[q] \subseteq \cT_\alpha$.  So $X \in \id(\bbQ_\kappa)^+$ is
a $\Sigma^1_1(\kappa)$ condition hence ``$X \in \id(\bbQ_\kappa)$'' is a
$\Pi^1_1(\kappa)$ condition and we finish by \ref{Claim8.2(4)}.
\end{PROOF}

\begin{claim}
\label{k5a}
1) Assume $\bbP$ is $({<} \kappa)$--complete or just strategically
$\kappa$-complete (i.e. for games with $\kappa$ moves, COM winning if a
play takes $\kappa$-moves).
\begin{enumerate}
\item[(a)]  Satisfying a $\kappa$-stationary-Borel is absolute between $\bold
   V$ and $\bold V^{\bbP}$.
\item[(b)]  Satisfying a $\Sigma^1_1(\kappa)$ relation is absolute between
  $\bold V$ and $\bold V^{\bbP}$. 
\end{enumerate}
\mn
2) If $\bbP$ is strategically $\theta$-complete for every $\theta <
\kappa$, then ``$p \in \bbQ_\kappa$'' is upward absolute from $\bold V$
to $\bold V^{\bbP}$.
\end{claim}

\begin{PROOF}{\ref{k5a}}
  Should be clear.
\end{PROOF}

\begin{observation}
\label{k6}
Being $\kappa$-stationary Borel is not equivalent to being
$\kappa$-Borel.
\end{observation}

\begin{PROOF}{\ref{k6}}
  Consider ${\bold A}_1 = \{S \subseteq \kappa:S$ is stationary$\}$ and
  ${\bold A}_0=\cP(\kappa)\setminus {\bold A}_1$. Clearly ${\bold A}_1$ is
  $\kappa$--stationary Borel and ${\bold A}_0$ is $\kappa$--non-stationary
  Borel (defined naturally).  Assume towards contradiction that
  ${\bold A}_1$ is equal to a $\kappa$--Borel set $\bold B$. Let
  $\Cohen_\kappa = ({}^{\kappa >}2,\triangleleft)$, and let $\name\eta$ be
  the $\kappa$-generic real. Then for some truth value $\bold t$ and
  $\nu \in {}^{\kappa >}2$ we have $\nu \Vdash_{\Cohen_\kappa}$``
  $\name\eta^{-1}\{1\} \in \bold B$ \Iff \, $\bold t = 1$ ''.  Let
  $\iota < 2$, $M \prec (\cH(\kappa^+),\in)$ be of cardinality
  $\kappa$, $[M]^{< \kappa} \subseteq M$ and $\bold B,\kappa \in M$.  Now we 
  can find $\nu_\iota \in {}^\kappa 2$ such that $\nu \triangleleft
  \nu_\iota$ and $\{\nu_\iota \rest \alpha:\alpha < \kappa\}$ is a subset of  
  $\Cohen_\kappa$ generic over $M$ and $\nu_\iota(\alpha) = \iota$ for a
  club of $\alpha < \kappa$.  By easy absoluteness we get
  $\nu_\iota \in \bold B$ \Iff \, $\bold t=1$, easy contradiction.
\end{PROOF}

\begin{claim}
\label{k8}
\begin{enumerate}
\item Consistently, $\kappa$ is weakly compact but being predense in 
$\bbQ_\kappa$ is not absolute under $\kappa$--complete forcing and hence it
is not $\kappa$-Borel.  
\item Assume $\kappa$ is weakly compact and moreover (can be gotten by
  preliminary forcing) this is preserved by adding $\kappa^+,\kappa$-Cohen.
  \Then \, adding a $\kappa^+,\kappa$-Cohens (i.e. forcing with
  $\Cohen_{\kappa,\kappa^+}$) we get the above.
\item In part (2) also $\{S \subseteq \kappa:S$ stationary in $\kappa\}$
(is $\kappa$-stationary Borel but) its complement is not
$\kappa$-stationary Borel.
\end{enumerate}
\end{claim}

\begin{PROOF}{\ref{k8}}
  The counterexample will be gotten by forcing by
  $\Cohen_{\kappa,\kappa^+}$, e.g., when $\kappa$ is Laver indestructible
  supercompact but similarly for $\kappa$ weakly compact by a preliminary
  forcing and the set $S_2$ below being $\{\partial < \kappa:\partial$ not
  Mahlo$\}$. 

Assume $\kappa$ is Mahlo and let $S_1 \subseteq S^\kappa_{\inc}$ be
nowhere stationary but unbounded. Let $S_2 \subseteq S^\kappa_{\inc}$ be a
stationary subset of ${\rm acc}(S_1)$. We define a representation $\bbQ_1$
of $\Cohen_\kappa$ as follows: 
\begin{enumerate}
\item[$(*)_1$]  
\begin{enumerate}
\item[(A)]  $p \in \bbQ_1$ iff:
\begin{enumerate}
\item[(a)] $p = \langle \eta_\partial:\partial \in S_2\cap \alpha\rangle
  =\langle \eta_{p,\partial}:\partial \in S_2 \cap\alpha_p\rangle$ for some
  $\alpha = \alpha_p < \kappa$,  
\item[(b)] for each $\partial \in S_2 \cap\alpha_p$, $\eta_\partial \in
  {}^\partial 2$. 
\end{enumerate}
\item[(B)]  $\bbQ_1$ is ordered by $\trianglelefteq$. 
\item[(C)]  The generic of $\bbQ_1$ is $\name{\bar\eta} = \bigcup\{p:p 
  \in \name{\bold G}_{\bbQ}\}$ and let $\name Y =
  \{\name\eta_\partial:\partial \in S_2\}$, where $p \Vdash$``
  $\name\eta_\partial = \nu$'' if $\partial \in S_2 \cap \alpha_p 
  \wedge \eta_{p,\partial} = \nu$.
\item[(D)]  The length $\ell g(p)$ of $p$ is the minimal $\alpha < \kappa$
  such that $\dom(p) = S_2 \cap \alpha$.
\end{enumerate}
\end{enumerate}
Next we let $p_\eta = \{\rho \in {}^{\kappa >}2:\rho \trianglelefteq \eta
\vee \eta \trianglelefteq \rho\} \in \bbQ_\kappa$ for $\eta \in
{}^{\kappa >}2$. Now
\begin{enumerate}
\item[$(*)_2$]  $\Vdash_{\bbQ_1} ``\{p_\eta:\eta \in \name Y\}$ is a
  predense subset of $\bbQ_\kappa"$.
\end{enumerate}
[Why?  If not, let $q \in \bbQ_1$, $q \Vdash_{\bbQ_1} ``\name p = (\nu,\name 
  S,\langle \name\Lambda_\partial:\partial \in \name S\rangle) \in
\bbQ_\kappa$ is incompatible with every $p_\eta$ for $\eta \in
\name Y$ and $\name E_1$ is a club of $\kappa$ disjoint to $\name S"$.

Let $\langle q_i:i < \kappa\rangle$ be increasing continuous in
$\bbQ_1$, $q_0=q$ and $q_{i+1}$ forces a value to $\name S \cap i$, $\langle  
\name\Lambda_\partial:\partial \in \name S \cap i\rangle$ and to
$\min(\name E_1 \backslash i)$ called $\gamma_i$.  Let
\[E=\big\{\delta < \kappa:\delta\mbox{ is a limit ordinal and }i < \delta  
\Rightarrow \ell g(q_i) < \delta \wedge \gamma_i < \delta\big\}.\]    
Clearly $E$ is a club of $\kappa$, so we can choose $\partial \in S_2 \cap
E$. Then $q_\partial \in \bbQ_1$ is well defined and of length $\partial$
and it forces a value $(S',\langle \Lambda'_\theta:\theta\in S'\rangle)$ to
$(\name S \cap \partial,\langle \name\Lambda_\theta:\theta \in \name S \cap 
\partial\rangle)$ and this value represents a condition $r \in
\bbQ_\partial$. Moreover, $q_\partial$ forces that $\partial =
\sup\{\gamma_i:i < \partial\} = \sup(\name E_1 \cap \partial) \in \name E_1$
and hence it forces $\partial \notin \name S$.  Choose $\nu \in
\lim_\partial(r) \in {}^\partial 2$  and let $q'_{\partial +1}$ be above
$q_\partial$ such that $q'_{\partial + 1}(\partial) = \nu$, i.e. $q'_{\delta
  +1} \Vdash ``\nu \in \name Y"$ and we arrive to an easy contradiction.]  
\medskip

Next, in $\bold V^{\bbQ_1}$ we define $\bbQ_2=\bbQ_2[\name{\eta}^1_\kappa]$,
$\name{\eta}^1_\kappa$ the generic for $\bbQ_1$, by
\begin{enumerate}
\item[$(*)_3$]  
\begin{enumerate}
\item[(A)]  $p \in \bbQ_2$ \Iff \,
\begin{enumerate}
\item[(a)] $p = (\alpha,\bar\Lambda) = (\alpha_p,\bar\Lambda_p)$, 
\item[(b)] $\alpha_p < \kappa$, $\bar\Lambda_p = \langle
  \Lambda_{p,\partial}:\partial \in S_1 \cap \alpha_p\rangle$,
\item[(c)] each $\Lambda_{p,\partial}$ is a family of $\leq\partial$ dense
  subsets of $\bbQ_\partial$ (for $\partial \in S_1 \cap \alpha_p$),
\item[(d)] if $\theta \in S_2 \cap (\alpha +1)$, \then  \, $\theta =
  \sup\{\partial \in S_1 \cap \theta:\eta_\theta \rest \partial \notin
  \set(\Lambda_{p,\partial})\}$\\  
(recall $S_2 \subseteq {\rm acc}(S_1)$); 
\end{enumerate}
\item[(B)]  the order is being an initial segment.
\item[(C)]  The generic is $\name{\bar\Lambda} =
  \langle \name\Lambda_\partial:\partial \in S_1\rangle$.
\end{enumerate}
\end{enumerate}

Now in $\bold V^{\bbQ_1}$ the forcing notion $\bbQ_2$ is not $(<
\kappa)$-complete and even not strategically $\kappa$-complete but it is 
strategically $({<}\kappa)$-complete. (It is not strategically
$\kappa$-complete because given {\bf st}, let $M \prec
(\cH(\chi),\in)$, $\chi = (2^\kappa)^+$, $M \cap \kappa = \partial \in 
S_2$, $\|M\| = \partial$, $[M]^{< \partial} \subseteq M$, ${\bf st}\in
M$, $\name Y \in M)$.

Now in $\bold V^{\bbQ_1 * \name{\bbQ}_2}$ easily $p =
(\langle\rangle,S_1,\name{\bar\Lambda})$ belongs to $\bbQ_\kappa$ 
and it exemplifies that $\langle p_\eta:\eta \in \name Y\rangle$ is not
predense.  Also $\bbQ_1 * \name{\bbQ}_2$ has a dense set closed subset
equivalent to $\kappa$-Cohen and similarly $\bbQ_1$, hence $\Vdash_{\bbQ_1 * 
  \name{\bbQ}_2}$``$\kappa$ is weakly compact'' and
$\Vdash_{\bbQ_1}$``$\kappa$ is weakly compact''. So there are
$\kappa$--Borel functions ${\bold B}_1,{\bold B}_2$ with domain ${}^\kappa
2$ and such that 
\[\begin{array}{ll}
\Vdash_{\Cohen_\kappa} &\mbox{`` }{\bold B}_1(\name{\eta}_\kappa) \mbox{ is
                         generic over }{\bold V} \mbox{ for }\bbQ_1\mbox{
                         and}\\ 
&\ \ {\bold B}_2(\name{\eta}_\kappa) \mbox{ is generic over }{\bold
  V}[{\bold B}_1(\name{\eta}_\kappa)] \mbox{ for }\bbQ_2[{\bold 
  B}_1(\name{\eta}_\kappa)] \mbox{ ''.} 
  \end{array}\]
  Assume that in $\bold V^{\bbQ_1}$, $\bold B$ is a (definition of a)
  $\kappa$--Borel subset of $[\cH(\kappa)]^\kappa$ which is the set of
  predense subsets of $\bbQ_\kappa$, so in $\bold V^{\bbQ_1 *
    \name{\bbQ}_2}$, $\bold B$ no longer satisfies this.  This is somewhat
  weaker than the desired conclusion, but if $\bar\eta = \langle
  \name\eta_\gamma:\gamma < \kappa^+\rangle$ is generic for
  $\Cohen_{\kappa,\kappa^+}$ and $\bold B \in \bold V[\bar\eta]$ is a
  (definition of a) $\kappa$-Borel subset of $[\cH(\kappa)]^\kappa$, for
  some $\alpha < \kappa$, $\bold B \in \bold V[\bar\eta \rest \alpha]$ and
  interpret $\name\eta_\alpha$ as the generic $\bbQ_1 *
  \name{\bbQ}_2$. Consider $\bar{p}={\bold  B}_1(\name{\eta}_\alpha)$. 

Now we can compute ${\bold B}_1(\bar{p})$ in ${\bold V}[\name{\eta}\rest
\alpha,\bar{p}]$ and in ${\bold V}[\name{\eta}\rest\alpha,
\name{\eta}_\alpha]$. As ${\bold B}$ is $\kappa$--Borel, we should get the
same result, but they are not the same. A contradiction. 
\end{PROOF}

\begin{definition}
\label{k5}
1) We say $M$ is a $\kappa$-model \when \,:
\mn
\begin{enumerate}
\item[(a)]  $M \subseteq (\cH(\kappa^+),\in)$ is transitive of
cardinality $\kappa,[M]^{< \kappa} \subseteq M$ and $M$ is a model of
ZFC$^-$ (i.e. power set axiom omitted);
\item[(b)]  similarly for $(\cH_{< \kappa^+}(\bold U),\in)$, $\bold U$ a set
  of ure-elements.
\end{enumerate}
\mn
2) We say $\eta$ is a $(M,\bbQ,\name\eta)$--generic $\kappa$--real 
\when \, (as in \cite{Sh:630}):
\mn
\begin{enumerate}
\item[(a)]  $\bbQ$ is a forcing notion definable in $M$, (absolutely
  enough in the interesting cases),
\item[(b)]  $\name\eta \in M$ a $\bbQ$--name of $\kappa$--real, 
 defined by a Borel function from a sequence of $\kappa$ truth values
 of the form  ``$p \in \name{\bold G}_{\bbQ}$'',
\item[(c)] there is $\bold G \subseteq \bbQ^M$ generic over $M$ such that
  $\name\eta[\bold G] = \eta$.
\end{enumerate}
\end{definition}

\begin{observation}
\label{k7}
1)\quad A $\kappa$--Borel set $\bold B$ belongs to $\id(\bbQ_\kappa)$ \Iff
\, for some $\kappa$--real ${\bold c}={\bold c}_\bB$ for every $\kappa$--model $M$
to which $\bold c$  belongs we have:
\begin{enumerate}
\item[$\bullet$]  if $\nu$ is $(M,\bbQ_\kappa,{\name\eta})$-generic real
  \then \, $\nu \notin \bB$.
\end{enumerate}
\mn 
2)\quad If $M$ is a $\kappa$-model, $M \models ``\bbQ$ is
$({<}\kappa)$--strategically complete forcing notion (set or class in $M$
sense) (or a definition of $\bbQ)"$ and $\bold G \subseteq \bbQ^M$ is generic over
$M$ \then \, $M[\bold G]$ is a $\kappa$-model.
\end{observation}

\begin{definition}  
\label{k31}
\begin{enumerate}
\item We say a set $X \subseteq {}^\kappa \cH(\kappa)$ is
  $\kappa-\id_\kappa$-Borel \when \,: 
\begin{enumerate}
\item[(a)]  $\id_\kappa$ is an ideal on $\cP(\kappa)$, 
\item[(b)]  for some $\kappa$-Borel function $F:{}^\kappa \cH(\kappa)
  \longrightarrow \cP(\kappa)$ for every $\eta \in {}^\kappa \cH(\kappa)$ 
  we have: $\eta \in X$ iff $F(\eta) \in \id$.
\end{enumerate}
Here (in (2),(3)) we may omit $\kappa$ when clear from the context.
\item Similarly for $\id^+_\kappa$.
\item Let $\id_{\wc}(\kappa)$ be the weakly compact ideal on $\kappa$.
\end{enumerate}
\end{definition}

So

\begin{observation}  
\label{k34}
Letting $\id_{\nst}(\kappa)$ be the non-stationary ideal on $\kappa$,
$\kappa$--$\id^+_{\nst}(\kappa)$-Borel means $\kappa$-stationary Borel.   
\end{observation}
\bigskip\bigskip

\begin{acknowledgements}
The author thanks Alice Leonhardt for the beautiful typing. 
We also owe a great debt to the referee for doing so much to improve the
paper.  
\end{acknowledgements}

\newpage 


\end{document}